\documentclass[preprint,sort&compress,12pt]{elsarticle}
\usepackage{bbm}
\usepackage{amsmath}
\usepackage{mathrsfs}
\usepackage{stmaryrd}
\usepackage{mathrsfs}
\usepackage{bm}
\usepackage{amsmath}
\usepackage{amsfonts}
\usepackage{amssymb}
\usepackage{amsthm}
\usepackage{lineno}
\usepackage[RGB]{xcolor}
\newcommand{\mm}{\mathrm}
\newcommand{\ml}{\mathcal}

\newcommand{\be}{\begin{equation}}
\newcommand{\bea}{\begin{equation}\begin{aligned}}
\newcommand{\beas}{\begin{equation*}\begin{aligned}}
\newcommand{\eeas}{\end{aligned}\end{equation*}}
\newcommand{\eea}{\end{aligned}\end{equation}}
\newcommand{\ee}{\end{equation}}

\renewcommand{\div}{{\rm div }}

\begin{document}
\begin{frontmatter}
\title{On Inhibition of Rayleigh--Taylor Instability by a Horizontal
Magnetic Field in Non-resistive MHD Fluids: the Viscous Case}

\author[FJ,1FJ,2FJ]{Fei Jiang}
\ead{jiangfei0951@163.com}
\author[SJ]{Song Jiang}
 \ead{jiang@iapcm.ac.cn}
\author[SJ]{Youyi Zhao}\ead{zhaoyouyi0591@163.com}
\address[FJ]{School of Mathematics and Statistics, Fuzhou University, Fuzhou, 350108, China.}
\address[SJ]{Institute of Applied Physics and Computational Mathematics,
 Beijing, 100088, China.}
 \address[1FJ]{Center for Applied Mathematics of Fujian Province, Fuzhou 350108, China.}
\address[2FJ]{Key Laboratory of Operations Research and Control of Universities in Fujian, Fuzhou 350108, China.}
\begin{abstract}
It is still open whether the phenomenon of inhibition of Rayleigh--Taylor (RT) instability by a horizontal magnetic field can
be mathematically verified for a non-resistive \emph{viscous} magnetohydrodynamic (MHD) fluid in a two-dimensional (2D) horizontal slab domain,
since it was roughly proved in the linearized case by Wang in \cite{WYC}. In this paper, we prove  such inhibition phenomenon by the (nonlinear) inhomogeneous, incompressible, \emph{viscous case} with \emph{Navier (slip) boundary condition}. More precisely, we show that
there is a critical number of field strength $m_{\mm{C}}$, such that if the strength $|m|$ of a horizontal magnetic field is bigger than $m_{\mm{C}}$,
then the small perturbation solution around the magnetic RT equilibrium state is {algebraically}  stable in time.
In addition, we also provide a nonlinear instability result for the case $|m|\in[0, m_{\mm{C}})$. The instability result presents that
a horizontal magnetic field can not inhibit the RT instability, if it's strength is too small.
\end{abstract}
\end{frontmatter}
\newtheorem{thm}{Theorem}[section]
\newtheorem{lem}{Lemma}[section]
\newtheorem{pro}{Proposition}[section]
\newtheorem{concl}{Conclusion}[section]
\newtheorem{cor}{Corollary}[section]
\newproof{pf}{Proof}
\newdefinition{rem}{Remark}[section]
\newtheorem{definition}{Definition}[section]

\section{Introduction}\label{introud}

The equilibrium of a heavier fluid on top of a lighter one, subject to gravity, is unstable. In fact, small disturbances acting on the equilibrium will grow and lead to the release of potential energy,
 as the heavier fluid moves down under the gravity, and the lighter one is displaced upwards. This phenomenon was first studied
 by Rayleigh \cite{RLIS} and then Taylor \cite{TGTP}, is called therefore the Rayleigh--Taylor (RT) instability.
In the last decades, RT instability had been extensively investigated from mathematical,
physical and numerical aspects, see \cite{CSHHSCPO,WJH,desjardins2006nonlinear,hateau2005numerical} for examples. It has been also widely analyzed how physical factors,
such as elasticity \cite{JFJWGCOSdd,FJWGCZXOE}, rotation \cite{CSHHSCPO,BKASMMHRJA},
internal surface tension \cite{GYTI2,WYJTIKC2014ARMA,JJTIWYJ2016CMP},
magnetic field \cite{JFJSWWWOA,JFJSJMFMOSERT,JFJSWWWN} and so on, influence the dynamics of the RT instability.

In this paper we are interested in the phenomenon of inhibition of RT instability by magnetic fields. This topic goes back to
the theoretical work of Kruskal and Schwarzchild \cite{KMSMSP}. They analyzed the effect of  the (impressed) horizontal magnetic
field upon the growth of the RT instability in a horizontally periodic motion of a completely ionized plasma  with zero resistance
in three dimensions in 1954, and pointed out that the curvature of the magnetic lines can influence the development of instability,
but can not inhibit the growth of the RT instability. The inhibition of RT instability by the vertical magnetic field was first
verified for the inhomogeneous, incompressible, non-resistive magnetohydrodynamic (MHD) fluids in three dimensions by Hide \cite{CSHHSCPO,HRWP}.
In 2012, Wang also noticed that the horizontal magnetic field can inhibit the RT instability in a non-resistive MHD fluid in two dimensions \cite{WYC}. Later, Jiang--Jiang further found that impressed magnetic fields always inhibit the RT instability,
if a non-slip velocity boundary-value condition is imposed in the direction of magnetic fields \cite{JFJSJMFM}.
Such boundary condition is called the ``\emph{fixed condition}'' for the sake of simplicity.

All results mentioned  above are based on the {\it linearized} non-resistive MHD equations.
Thanks to the multi-layers method developed in the well-posedness theory of surface wave problems \cite{GYTIAED}, recently the phenomenon of inhibition of RT instability by magnetic fields has been rigorously proved based on the (nonlinear) non-resistive viscous MHD equations under the fixed condition, for example, Wang verified the inhibition phenomenon by the non-horizontal magnetic field in a stratified incompressible viscous MHD fluid in a 2D/3D slab domain \cite{WYJ2019ARMA}. Moreover, he also proved that the  horizontal magnetic field can not inhibit the RT instability for the horizontally periodic motion for the 3D case \cite{WYJ2019ARMA}, but can  \emph{inhibit the RT instability  based on a 2D linearized motion equations} in \cite{WYC}. Similar results can be also found in other magnetic inhibition phenomena, see \cite{JFJSSETEFP} for the Parker instability and \cite{JFJSOUI} for the thermal instability.  The previous nonlinear stability/instability results in the non-resistive \emph{viscous} magnetic RT problem can be summarized as in the following table.
\begin{center}
Can an impressed horizontal/vertical magnetic field inhibit the RT instability \\ in a non-resistive \emph{viscous} MHD fluid  in a slab domain?\\[0.4em]
\begin{tabular}{|c|c|c|}
\cline{2-2}  \hline    & \vspace{0.1em}{horizontal} & vertical  \\
   \hline
  2D &   No clear & Yes \\
   \hline
  3D & No & Yes \\
  \hline
\end{tabular}
\end{center}

In \cite{JFJSARMA2019} Jiang--Jiang further established a so-called magnetic inhibition theory in viscous non-resistive MHD fluids, which reveals the physical effect of the \emph{fixed condition} in magnetic inhibition phenomena. Roughly specking, let us consider
an element line along an impressed magnetic field in the rest state of a non-resistive MHD fluid, then the element line of fluids can be regarded as an elastic string. The two endpoints of the element line are fixed due to the fixed condition.
Thus, the bent element line will restore to its initial location under the magnetic tension   as well as viscosity.

By the magnetic inhibition mechanism in non-resistive MHD fluids, the positive assertions in the table above seem to be obvious; moreover, we easily predict the phenomenon of  inhibition of RT instability by a horizontal magnetic field  in a non-resistive viscous MHD fluid in a 2D slab domain, see \cite{WYC} for the linear case. However, it is still an open problem to rigorously prove this prediction based on the nonlinear motion equations.  Recently the authors noted that this prediction  can be mathematically verified by the \emph{inviscid case with velocity damping}, i.e. the viscous term is replaced by the velocity damping term. Under such case, some difficulties arising from the viscous term can be avoided, when we exploit cur estimates.  In this paper, we further find that such inhibition phenomenon can be also proved in the inhomogeneous, incompressible, \emph{viscous} and non-resistive MHD fluids with \emph{Navier (slip) boundary condition} in two-dimensions, and thus move a first step to this open problem with the viscous case. More precisely, there exists a critical number $m_{\mm{C}}$, such that if the strength $|m|$ of a horizontal magnetic field is bigger than $m_{\mm{C}}$, then the small perturbation solution around the magnetic RT equilibrium state is {algebraically} stable in time, i.e. the RT instability can be inhibited by a horizontal magnetic field in a 2D slab domain.
Finally we further mention our stability result.
\begin{enumerate}[(1)]
\item Ren--Xiang--Zhang proved the existence of  the global(-in-time) small
perturbation solutions  of a non-resistive viscous MHD fluid with a horizontal magnetic field and with a Navier boundary condition  in a 2D slab domain \cite{ren2016global}. However, they cannot obtained the decay-in-time for the global solutions. As a by-product of our stability result, we further provide the algebraical decay-in-time behavior for the solutions.
\item
Wang mathematically verified the the inhibition phenomenon by the non-horizontal magnetic field in a stratified incompressible viscous MHD fluid in a 2D/3D slab domain \cite{WYJ2019ARMA} and also obtained unique global solutions with decay-in-time. It should be noted the decay-in-time plays an important role to derive the existence of global solution in Wang's result. However, our proof for the existence of global-in-time solutions is independent of the decay-in-time.  \emph{We provide the additional derivation of decay-in-time in our result, since it may be useful in the further investigation of the case of the non-slip boundary condition in future}.
\item Our stability result  can be viewed as \emph{a continuation of the previous
work} of the inviscid case with velocity damping in \cite{ZHAOYUI}. However, we develop a new idea to capture the high-order normal (spacial) derivatives of the deviation function of fluid particles  from the viscosity term under the Navier boundary condition, and the details will be further discussed after introducing our stability result in Theorem \ref{thm2}.
\item By the magnetic inhibition mechanism in non-resistive MHD fluids,
the horizontal magnetic field plays a role of tension in the horizontal direction, and thus can inhibit the RT instability as well as the surface tension \cite{WYJTIKC2014ARMA,JJTIWYJ2016CMP}. Our result mathematically verifies such physical phenomenon. It seems that our proof idea can be  extended to verify that the horizontal magnetic field can also inhibit other flow instabilities, such us thermal instability in \cite{JFJSOUI}. In addition, we will further use the basic idea in this paper to prove that the RT instability can be also inhibited by other stabilizing forces  in the horizontal direction, such as capillary action on RT instability in capillary
fluids in a forthcoming paper.
\end{enumerate}

\subsection{Mathematical formulation for the magnetic RT problem}\label{subsec:01}
\numberwithin{equation}{section}
Before stating our results in details, we shall mathematically formulate the physical
problem of  inhibition of RT instability by a horizontal magnetic field.
The governing equations of an inhomogeneous, incompressible, viscous, non-resistive MHD fluid in the presence of a uniform gravitational field in a 2D slab domain $\Omega$ read as follows.
\begin{equation}
\label{0101}
\begin{cases}
\rho_t+ v\cdot \nabla \rho=0,\\
\rho v_{t}+ \rho v\cdot\nabla v
+\nabla \left(P+\lambda|M|^2/2\right)-\mu \Delta v =\lambda M\cdot\nabla M-\rho g \mathbf{e}_2, \\
 {M}_{t}+ v\cdot\nabla {M}=M\cdot\nabla v, \\
\div  v =\mm{div}M=0.
\end{cases}
\end{equation}
Below, we explain the mathematical notations in the system \eqref{0101}.

The unknowns $\rho:=\rho(x,t)$, $v:={v}(x,t)$, $M:= {M}(x,t)$ and $P:=P(x,t)$ denote the density, velocity,
magnetic field and kinetic pressure of a MHD fluid, resp.. $x\in\Omega\subset\mathbb{R}^2$ and $t>0$ are the spatial
and temporal variables resp.. The constants $\lambda$, $g> 0$ and $\mu >0$ stand for the permeability of vacuum, the gravitational
constant and the viscosity coefficient, resp.. $\mathbf{e}_2=(0,1)^{\mm{T}}$ represents the normal (or vertical) unit vector,
and $-\rho g \mathbf{e}_2$ the gravity, where the superscript ${\mm{T}}$ denotes the transposition.

Since we consider horizontally periodic motion solutions of \eqref{0101}, we define a horizontally periodic domain:
\begin{align}\label{0101a}
\Omega:= 2\pi L\mathbb{T} \times(0,h),
\end{align}
where $\mathbb{T}=\mathbb{R}/\mathbb{Z} $ and $L>0$. For the horizontally periodic domain $\Omega$, the 1D periodic domain
$2\pi L\mathbb{T}\times \{0,h\}$, denoted by $\partial\Omega$, which customarily is regarded as a boundary of $\Omega$.
For the well-posedness of the system \eqref{0101}, we impose the following initial and boundary conditions:
\begin{align}
& (\rho,v,M)|_{t=0}=(\rho^0,v^0,M^0),\label{20210031303} \\
& \label{0101b}
v|_{\partial\Omega}\cdot\vec{\mathbf{n}}=0,  \ 2(\mathbb{D}v|_{\partial\Omega})\vec{\mathbf{n}})_{\mm{tan}}=0,
\end{align}
where $\vec{\mathbf{n}}=(\vec{{n}}_1,\vec{{n}}_2)^{\mm{T}}$ denotes the outward normal unit  vector on  $\partial\Omega$, $\mathbb{D}v=(\nabla v+\nabla v^{\mm{T}})/2$ the strain tensor, and the subscript ``$\mm{tan}$" the tangential component of a vector (for example $v_\mm{tan}=v-(v\cdot \vec{\mathbf{n}})\vec{\mathbf{n}}$) \cite{tapia2021stokes,ding2020rayleigh,ding2018stability,li2019global}.
  Here and in what follows, we always use the superscript $0$ to emphasize the initial data.

We call the boundary conditions in \eqref{0101b} the Navier (slip) boundary condition.
 Since $\Omega$ is a slab domain, the  Navier boundary  condition
 is equivalent to
\begin{align}
 (v_2,\partial_2 v_1)|_{\partial\Omega}=0 .
 \label{20220202081737}
 \end{align}

Now, we choose a RT density profile $\bar{\rho}:=\bar{\rho}(x_2)$, which is independent of $x_1$ and satisfies
\begin{align}
&\label{0102}
\bar{\rho}\in {C^2}(\overline{\Omega}),\ \inf_{ x\in {\Omega}}\bar{\rho}>0,\\[1mm]
&\label{0102n}\bar{\rho}'|_{x_2=y_{2}}>0 \;\;
\mbox{ for some } y_{2}\in \{x_2~|~(x_1,x_2)^{\mm{T}}\in \Omega\},
\end{align}
where $\bar{\rho}':=\mm{d}\bar{\rho}/\mm{d}x_2$ and $\overline{\Omega}:=\mathbb{R}\times [0,h]$.
We remark that the second condition in \eqref{0102}
prevents us from treating vacuum, while  the condition in \eqref{0102n} is called the RT condition,
which assures that there is at least a region in which the density is larger with increasing height $x_2$, thus
leading to the classical RT instability, see \cite[Theorem 1.2]{JFJSO2014}.

With the RT density profile in hand, we further define a magnetic RT equilibria $r_M:=(\bar{\rho}, 0, \bar{M})$,
where $\bar{M} =(m,0)^{\mm{T}}$ with $m$ being a constant. Usually, $\bar{M}$ is called an impressed horizontal magnetic field (or \emph{horizontal magnetic field} for the sake of simplicity).
The pressure profile $\bar{P}$ under the equilibrium state is determined by the relation
\begin{equation}
\nabla \bar{P}=-\bar{\rho}g  \mathbf{e}_2\mbox{ in } {\Omega}.  \label{equcomre}
\end{equation}

Denoting the perturbation around the magnetic RT equilibrium by
$$\varrho=\rho -\bar{\rho},\ v= v- {0},\ N=M-\bar{M} $$
 and then using the relation \eqref{equcomre}, we obtain the system of perturbation equations from \eqref{0101}:
\begin{equation}\label{0103} \begin{cases}
\varrho_t+{  v}\cdot\nabla (\varrho+\bar{\rho})=0, \\[1mm]
(\varrho+\bar{\rho}){  v}_t+(\varrho+\bar{\rho}){  v}\cdot\nabla
{ v}+\nabla  \beta-\mu \Delta v
=  \lambda (N+\bar{M})\cdot \nabla  N - \varrho  g \mathbf{e}_2,\\[1mm]
N_t+ v\cdot\nabla  N =(N+\bar{M})\cdot \nabla v,\\[1mm]
 \mm{div}v= \mathrm{div} N=0,\end{cases}\end{equation}
where $\beta:= P-\bar{P}+\lambda  (| M |^2-|\bar{M}|^2)/2$ is called the total perturbation pressure.
The corresponding initial and boundary conditions read as follows.
\begin{align} \label{c0104}
&(\varrho,v, {N} )|_{t=0}=(\varrho^0,v^0,N^0) ,  \\
& \label{0105}
 (v_2,\partial_2 v_1)|_{\partial\Omega}=0 \mbox{ on }{\partial\Omega}.
\end{align}
We call the initial-boundary value problem \eqref{0103}--\eqref{0105} the magnetic RT problem for the sake of simplicity.
Obviously, to mathematically prove the inhibition of RT instability by a horizontal magnetic field in a 2D slab domain,
it suffices to verify the stability in time for the solutions of the magnetic RT problem with \emph{some non-trivial initial data}.

\subsection{Reformulation in Lagrangian coordinates}\label{subsec:02}

  To proceed, as in \cite{WYJ2019ARMA,JFJSSETEFP,JFJSOUI}, we shall first reformulate the magnetic RT problem
in Lagrangian coordinates.
Let the flow map $\zeta$  be the solution to the initial value problem:
\begin{equation}
\label{201806122101}
            \begin{cases}
\partial_t \zeta(y,t)=v(\zeta(y,t),t),
\\[1mm]
\zeta(y,0)=\zeta^0(y),
                  \end{cases}
\end{equation}
where the invertible mapping $\zeta^0:=\zeta^0(y)$ maps $\Omega$ to $\Omega$, and satisfies
\begin{align}
&\label{zeta0inta}
J^0:=\det \nabla \zeta^0=1 ,\\[1mm]
&\label{zeta0inta0}
\zeta^0_2 =y_2\mbox{ on } \partial\Omega.
\end{align}
Here and in what follows, ``$\det$'' denotes the determinant of a matrix.
In our   results, we will see  that the flow map $\zeta$ satisfies, for each fixed $t>0$,
\begin{align}
& \zeta|_{y_2=r}   : \mathbb{R}\to \mathbb{R}\mbox{ is a } C^1(\mathbb{R})\mbox{-diffeomorphism mapping for }r=0,\ h,\label{20210301715x}\\
&\zeta   :  \overline{\Omega}\to \overline{\Omega} \mbox{ is a } C^1(\overline{\Omega})\mbox{-diffeomorphism mapping}.\label{20210301715}
\end{align}

Since $v$ satisfies the divergence-free condition and non-slip boundary condition
$v_2|_{\partial\Omega}=0$, we can deduce from \eqref{201806122101}--\eqref{zeta0inta0} that
\begin{align}
&\nonumber
J:=\det\nabla \zeta=1 ,\\
&\nonumber
\zeta_2=y_2 \mbox{ on } \partial\Omega.
\end{align}

We define the matrix $\mathcal{A}:=(\mathcal{A}_{ij})_{2\times 2}$ via
\begin{align}\nonumber
\mathcal{A}^{\mm{T}}=(\nabla\zeta)^{-1}:=
(\partial_j \zeta_i)^{-1}_{2\times 2}.
\end{align}
Then we further define
the differential operators $\nabla_{\mathcal{A}}$, $\mm{div}_{\mathcal{A}}$ and $\mm{curl}_{\mathcal{A}}$ as follows: for a scalar function $f$ and a  vector function $X:=(X_1,X_2)^{\mm{T}}$,
\begin{align}
&\nabla_{\mathcal{A}}f:=(\mathcal{A}_{1k}\partial_kf,
\mathcal{A}_{2k}\partial_kf)^{\mm{T}},\ \mm{div}_{\mathcal{A}}(X_1,X_2)^{\mm{T}}:=\mathcal{A}_{lk}\partial_k X_l\nonumber
\end{align}
and
\begin{align}
\mm{curl}_{\mathcal{A}}X:=\mathcal{A}_{1k}\partial_{k}X_2-\mathcal{A}_{2k}\partial_{k}X_1,
 \nonumber \end{align}
where we have used the Einstein convention of summation over repeated indices, and $\partial_k:=\partial_{y_k}$. In particular, $\mm{curl}X:=\mm{curl}_IX$, where $I$ represents an identity matrix.

Defining the Lagrangian unknowns:
\begin{equation*}
(  \vartheta, u ,Q, B)(y,t)=(\rho,v,P+\lambda|M|^2/2,M)(\zeta(y,t),t) \mbox{ for } (y,t)\in \Omega \times\mathbb{R}^+_0,
\end{equation*}
then in Lagrangian coordinates, the initial-boundary value problem \eqref{0101}, \eqref{20210031303}  and  \eqref{20220202081737} can be
rewritten as follows:
\begin{equation}\label{01dsaf16asdfasf00}
\begin{cases}
\zeta_t=u ,  \ \vartheta_t=0 ,\
\div_\mathcal{A}u=0 ,    \\[1mm]
\vartheta u_t+\nabla_{\mathcal{A}}Q-\mu \Delta_{\mathcal{A}}  u=\lambda B\cdot\nabla_{\mathcal{A}}B-\vartheta g \mathbf{e}_2 , \\[1mm]
B_t=B\cdot \nabla_{\mathcal{A}}u   ,  \
\div_\mathcal{A}B=0 ,   \\[1mm]
(\zeta,\vartheta,u, B)|_{t=0}=(\zeta^0,\vartheta^0,u^0, B^0) , \\[1mm]
(\zeta_2-y_2,u_2, \mathcal{A}_{2i}\partial_iu_1)|_{ \partial\Omega} =0  ,
\end{cases}
\end{equation}
where $(\vartheta^0,  u^0,B^0 ):=
 (\rho^0(\zeta^0),v^0(\zeta^0), M^0(\zeta^0)) $. In addition, the relation \eqref{equcomre} in Lagrangian coordinates reads as follows.
 \begin{equation}
\label{dstist01}
\nabla_{\mathcal{A}}\bar{P}(\zeta_2) = -\bar{\rho}(\zeta_2)g\mathbf{e}_2 .
\end{equation}

Let $\eta=\zeta-y$, $\eta^0=\zeta^0-y$, $q=Q-\bar{P}(\zeta_2)-\lambda|\bar{M}|^2/2$, $\mathcal{A}=(I+\nabla \eta)^{-\mm{T}}$ and
\begin{align} \nonumber
{G}_{\eta}:=g(\bar{\rho}(\eta_2(y,t)+y_2)-\bar{\rho}(y_2)).
\end{align}
In particular, we can   calculate that
\begin{align*}
\mathcal{A}=
\begin{pmatrix}
1+\partial_2\eta_2 &-\partial_1\eta_2\\[1mm]
-\partial_2\eta_1 &1+\partial_1\eta_1
        \end{pmatrix}.
\end{align*}

If $\eta^0$, $\vartheta^0$ and $B^0$ additionally satisfy
 \begin{align}
\nonumber
\vartheta^0  =\bar{\rho}(y_2)\mbox{ and }
B^0= m\partial_1(\eta^0+y) ,
\end{align}
then the initial-boundary value problem \eqref{01dsaf16asdfasf00}, together with the relation \eqref{dstist01}, implies that
  \begin{align}\label{01dsaf16asdfasf}
          &                    \begin{cases}
\eta_t=u ,\\[1mm]
\bar{\rho}u_t+\nabla_{\mathcal{A}} q-\mu \Delta_{\mathcal{\mathcal{A}}} u=\lambda m^2\partial_1^2\eta+ {G}_{\eta}\mathbf{e}_2  ,\\[1mm]
\div_{\mathcal{A}} u=0  , \\[1mm]
(\eta,u)|_{t=0}=(\eta^0,u^0)  ,
\end{cases} \\
&
 (\eta_2, u_2) |_{\partial\Omega}=0, \label{20202201182345} \\
& ( (1+\partial_1\eta_1  )\partial_2u_1 - \partial_2\eta_1 \partial_1u_1  )|_{\partial\Omega}=0  \label{202345}
\end{align}
and
\begin{align}\label{202012280945}
\vartheta=\bar{\rho}(y_2),\
B=m\partial_1(y+\eta),
\end{align}
please refer to \cite{JFJSJMFMOSERT} for the derivation. We mention that the term $\lambda m^2\partial_1^2\eta$ physically represents the magnetic tension, which can inhibit flow instabilities \cite{JFJSARMA2019}. It should be remarked that \eqref{01dsaf16asdfasf}--\eqref{202012280945} also implies \eqref{01dsaf16asdfasf00},
and that $q$, still called the perturbation pressure for simplicity,
is in fact the sum of the perturbation pressure and perturbation magnetic pressure in Lagrangian coordinates.

Unfortunately, it seems to be difficult to capture the estimates of high-order normal  derivatives of $\eta$ due to the absence the boundary condition of $\eta_1$.
Thus we shall pose an additional boundary condition
\begin{equation}
\partial_2\eta_1|_{\partial\Omega}=0,
\label{202220110182343}
 \end{equation}
 which, together with \eqref{01dsaf16asdfasf}$_1$, formally yields
 \begin{equation}
\partial_2u_1|_{\partial\Omega}=0  .\nonumber
 \end{equation}
 It is easy to see that \eqref{202345} automatically holds under the above two boundary conditions. Hence we use \eqref{202220110182343} to replace \eqref{202345}, and thus pose the new boundary condition
\begin{align}
& (\eta_2,\partial_2\eta_1, u_2,\partial_2u_1) |_{\partial\Omega}=0 \label{20safd45}
\end{align}
From now on, we call the
initial-boundary value problem \eqref{01dsaf16asdfasf} and \eqref{20safd45} {\it the transformed MRT problem}.
The stability problem of
the magnetic RT problem reduces to investigating the stability of the transformed MRT problem.

We mention that the boundary condition
\begin{align}
(\eta_2,\partial_2\eta_1) |_{\partial\Omega}=0
\label{202220118000}
 \end{align}
is called   the characteristic boundary condition. Indeed, if the initial data $\eta^0$ satisfies $(\eta_2^0,\partial_2\eta_1^0) |_{\partial\Omega}=0$, then $\eta$ automatically satisfies \eqref{202220118000}
due to the facts \eqref{01dsaf16asdfasf}$_1$ and the boundary condition
\begin{align}
\label{202201180853}
(u_2,\partial_2u_1) |_{\partial\Omega}=0.
 \end{align}
It should be noted that the   boundary condition  \eqref{202220118000} automatically implies
\begin{align}
\label{2022011091553}
\mm{cur}_{\mathcal{A}}\partial_1^i\eta|_{\partial\Omega} =0\mbox{ for }i=0,\ 1,
\end{align}
which will plays an important to capture the high-order normal estimates for $\eta$,  see Lemma \ref{2055nnn}. This is also a key idea in the mathematical proof for the magnetic inhibition phenomenon under the horizontal field in our paper.

\subsection{Notations}\label{subsec:04}

Before stating our main results on  the transformed MRT problem, we shall  introduce simplified notations throughout this paper.

\begin{enumerate}[(1)]
  \item Simplified basic notations: $\mathbf{e}_1:=(1,0)^{\mm{T}}$, $I_a:=(0,a)$ denotes a time interval, in particular, $I_\infty=\mathbb{R}^+$.  $\overline{S}$ denotes the closure of a set $S\subset \mathbb{R}^n$ with $n\geqslant 1$, in particular, $\overline{I_T} =[0,T]$ and $\overline{I_\infty} = \mathbb{R}^+_0$. $\Omega_t:=\Omega\times I_t$, $\int:= \int_{(0,2\pi L)\times (0,h)}$. $(u)_{\Omega}$ denotes the mean value of $u$ in a periodic cell $ (0,2\pi L)\times (0,h)$.
      $a\lesssim b$ means that $a\leqslant cb$ for some constant $c>0$.
      If not stated explicitly, the positive constant $c$ may depend on $\mu$, $g$,   $\lambda$, $m$, $\bar{\rho}$ and $\Omega$
      in the transformed MRT problem, and may vary from one place to other place. Sometimes, we use $c_i$
      to replace $c$ in order to emphasize that $c_i$ is a fixed value for $1\leqslant i\leqslant 3$. The letter $\alpha$ always denotes the multi-index with respect to
      the variable $y$, $|\alpha|=\alpha_1+\alpha_2$ is called the order of multi-index,
  $\partial^{\alpha}:=\partial_{1}^{\alpha_1} \partial_{2}^{\alpha_2}$ and     $[\partial^{\alpha},\phi]\varphi:=\partial^{\alpha}(\phi\varphi)-\phi\partial^{\alpha}\varphi$.
  \item  Simplified Banach spaces, norms and semi-norms:
  \begin{align}
&L^p:=L^p (\Omega)=W^{0,p}(\Omega),\
{H}^i:=W^{i,2}(\Omega ), \  H^{j}_{\mathrm{s}}:=\{w\in {H}^{j}~|~w_2|_{\partial\Omega}=0\}\nonumber \\[1mm]
&H^{j}_\sigma:=\{w\in {H}^{j}_{\mathrm{s}}~|~\div w=0\},  \    H^3_{\gamma }:=\{w\in H^3_{\mm{s}}~|~ \|\nabla w\|_2 \leqslant \gamma\},\nonumber\\ & \mathcal{H}^{k}_{\mathrm{s}}:=\{w\in H^{k}_{\mm{s}}~|~ \partial_2w_1|_{\partial\Omega} =0\}
,\ {H}^{j}_1:=\{w\in H^{j}~|~\det(\nabla w+I)=1\},\nonumber
  \\ & ^0\!{X }  := \{w\in X~|~(\bar{\rho}w_1)_\Omega=0,\ w_1\mbox{ is the first compent of $w$}\},  \nonumber  \\
&  \underline{X}:=\{w\in X~|~(w)_{\Omega}=0\},\ {\mathcal{H}}^k_\sigma:= \mathcal{H}^k_{\mm{s}}\cap {H}^1_\sigma,\ { {\mathcal{H}}^{3,\mm{s}}_{1,\gamma}}:= {H}^3_1\cap{ \mathcal{H}^3_{\mm{s}}}\cap H_{\gamma}^3,\nonumber  \\
& \textstyle{ {\mathcal{H}}^{\infty}_{\sigma}}:= \cap_{n=2}^\infty \mathcal{H}^n_{\sigma},\  \|\cdot \|_i :=\|\cdot \|_{H^i},\ \|\cdot\|_{l,i}:= \|\partial_{1}^{l}\cdot\|_{i},\ \|\cdot\|_{\underline{l},i}:=\sqrt{\sum_{0\leqslant n \leqslant l}\|\cdot\|_{n,i}^2}, \nonumber
\end{align}
where $1\leqslant p\leqslant \infty$, $i$, $l\geqslant 0$,  $j \geqslant 1$, $k\geqslant 2$, $X$ denotes a general Banach space and
  $\gamma \in(0,1)$ is the constant in Lemma \ref{pro:1221}. It should be noted that if $w\in H^3_{\gamma }$, then $\psi:=w+y$ (after possibly being redefined on a set of measure zero with respect to variable $y$) satisfies
the same diffeomorphism  properties as $\zeta$ in \eqref{20210301715x} and \eqref{20210301715} by Lemma \ref{pro:1221}. In addition, for simplicity, we denote $\sqrt{\sum_{1\leqslant n\leqslant j}\|f^k\|_{\mathcal{X}}^2}$ by $\|(f^1,\ldots,f^j)\|_{\mathcal{X}}$, where $\|\cdot\|_{\mathcal{X}}$ represents a norm or a semi-norm, and $f^k$ may be a scalar function, a vector or a matrix for $1\leqslant n\leqslant j$.
 \item Simplified spaces of functions with values in a Banach space:
\begin{align}
& L^p_TX:=L^p(I_T,X),\nonumber \\
& {\mathcal{U}}_{ T} =  \{u\in C^0(\overline{I_T},  {\mathcal{H}^2_{\mm{s}}})\cap L^2_T { {H}}^3  ~|~
 u_t\in C^0(\overline{I_T} ,L^2)\cap  L^2_TH^1_{\mm{s}}\}, \nonumber\\
&   \widetilde{\mathfrak{H}}^{1,3}_{\gamma,T}:=\{\eta\in  C^0(\overline{I_T} , { \mathcal{H}^{3}_{ \mathrm{s}}}) ~|~ \eta(t) \in  {^0 {\mathcal{H}}^{3,\mm{s}}_{1,\gamma}}\mbox{ for each }t\in \overline{I_T}\} .
\nonumber
\end{align}
It should be noted that $L^2_TL^2= L^2(\Omega_T)$.
\item A functional of potential energy: for any given $w\in H^1$,
\begin{align}\nonumber
E(w):=g\int\bar{\rho}'w_2^2\mm{d}y-\lambda \| m\partial_1w\|_0^2 .
\end{align}
  \item Energy and dissipation functionals (generalized):
\begin{align}
& \mathcal{E}:= \|   \eta \|_3^2+\| u\|_2^2+\|u_t\|_0^2+\| q\|_1^2,\nonumber \\
 & \mathcal{D}:= \|\partial_1 \eta_1\|_2^2+\|\eta_2\|_{3}^2+\| u \|_3^2+\|u_t\|_1^2+\|q\|_2^2.\nonumber
\end{align}
We call $\mathcal{E}$, resp. $\mathcal{D}$ the total energy, resp. dissipation functionals.
 \item Other notations for decay-in-times:
 \begin{align}
 \mathfrak{E}:=&  \langle t\rangle (\|\partial_2^3\eta_2\|_0^2+  \|\partial_2^2\eta\|_{1,0}^2)  + \langle t\rangle^2 (\|  \partial_2^2 \eta_2 \|_{0 }^2+ \|\partial_2\partial_1\eta\|_{\underline{1},0}^2 ) \nonumber \\
& +\langle t\rangle^3(\|(\eta_2,\partial_2\eta_2)\|_0^2 +\|\partial_1\eta\|_{\underline{2},0}^2+\|u\|_2^2
+ \|q\|_1^2+ \|u_t\|^2_0 ), \label{2022202180904} \\
  \mathfrak{D}:=&   \langle t\rangle  (\|\partial_2 \eta\|_{2,0}^2+\|u\|_3^2) + \langle t\rangle^2  ( \|(\eta_2,\partial_2\eta_2)\|_{0}^2
 +\|\partial_1\eta\|_{\underline{2},0}^2\nonumber \\
 &+ \|u\|_{\underline{1},2}^2+\|q\|_{\underline{1},1}^2)+
 \langle t \rangle^3(\|  \partial_1 u\|_{\underline{1},1}^2+\|   u_t\|_1^2) .\label{2022202180904x}
\end{align}
\end{enumerate}

\subsection{Main results}\label{subsec:04xx}
Now, we introduce the stability result for the transformed MRT problem.
\begin{thm}[Stability]\label{thm2}
Let  $\bar{\rho}$ satisfy \eqref{0102}, \eqref{0102n} and \begin{align}
\label{2020102241504}
|m|>m_{\mm{C}}:=\sqrt{\sup_{ { w}\in H_{\sigma}^1}\frac{g\int\bar{\rho}' { w}_2^2\mm{d}y}
{\lambda\|\partial_1  w \|^2_0}}.
 \end{align}
Further assume    $(\eta^0,u^0)\in{^0{\mathcal{H}}^{3,\mm{s}}_{1,\gamma}} \times {^0\mathcal{H}^2_{\mm{s}}}$
and
$\mm{div}_{\mathcal{A}^0}u^0=0$,  where $\mathcal{A}^0:=(\nabla \eta^0+I)^{-\mm{T}}$.
Then there is a sufficiently small constant $\delta >0$, such that for any $(\eta^0,u^0)$ satisfying
 $$\|(\nabla \eta^0,u^0)\|_2 \leqslant\delta ,$$
the transformed MRT problem \eqref{01dsaf16asdfasf} and \eqref{20safd45} admits a unique global strong solution $(\eta,u,q)$ in the function class $\widetilde{\mathfrak{H}}^{1,3}_{\gamma,\infty}\times {^0\mathcal{U}_{\infty} }\times (C^0(\mathbb{R}^+ _0, \underline{H}^1)\cap L^2_\infty {{H}^2})$. Moreover, the solution enjoys the following properties:
\begin{enumerate}[(1)]
   \item stability estimate of total energy: for a.e. $t>0$,
   \begin{align}\label{1.200} \mathcal{E}(t)+\int_0^t\mathcal{D}(\tau)\mm{d} \tau
\lesssim \| (\nabla \eta^0,u^0)\|_2^2.\end{align}
  \item  algebraic decay-in-time: for a.e. $t>0$,
    \begin{align}
    & \mathfrak{E} (t)+c\int_0^t\mathfrak{D}(\tau)\mm{d}\tau
\lesssim  \|(\nabla\eta^0,u^0)\|_2^2   , \label{1.200n0} \\
& \|\eta_1(t)-\eta^\infty_1\|_2^2\lesssim   (\|  \nabla \eta^0\|_{2}^2+ \|u^0\|_2^2  )\langle t\rangle,\label{1.200xx}
\end{align}  where $\eta_1^\infty\in H^2$ only depends on $y_2$.
\end{enumerate}
\end{thm}
\begin{rem}
By the assumptions of $\bar{\rho}$, we easily find that
\begin{align}
0<m_{\mm{C}}\leqslant \frac{h  }{ \pi }\sqrt{\frac{g\|\bar{\rho}'\|_{L^\infty}}{\lambda}},
\label{2022201281805}
\end{align}
please refer to (4.25) and Lemma 4.6 in \cite{JFJSARMA2019}. Thus, in view of Theorem \ref{thm2}, we see that
the horizontal field can inhibit the RT instability,
if the  field strength  is properly large.
Since \eqref{2022201281805} also holds for the domain $\Omega=\mathbb{R}\times (0,h)$, we naturally believe \emph{the conclusion that the properly large horizontal field can also inhibit the RT instability in the domain $\Omega=\mathbb{R}\times (0,h)$}. Such conclusion  will be further investigated  in an independent paper in future.
\end{rem}
\begin{rem}
It is easy to see from the proof of Theorem \ref{thm2} that
\begin{enumerate}[(1)]
  \item if the assumptions \eqref{0102n} and \eqref{2020102241504} are replaced by
$$ \bar{\rho}'\leqslant 0\;\mbox{ in }\Omega \mbox{ and }  |m|>0,$$
then the conclusions in Theorem \ref{thm2} still hold.
  \item Theorem \ref{thm2} is also valid for the case $g=0$.
\end{enumerate}
\end{rem}
\begin{rem}\label{202201271239}
For each fixed $t\in \mathbb{R}_0^+$, the solution $\eta(y,t)$ in Theorem \ref{thm2} belongs to $ H^3_{\gamma}$. Let $\zeta=\eta+y$, then $\zeta $  satisfies \eqref{20210301715x} and \eqref{20210301715} for each $t\in \mathbb{R}_0^+$ by Lemma \ref{pro:1221}. We denote the inverse transformation of $\zeta$ by $\zeta^{-1}$, and then define that
\begin{align}
(\varrho, v,N,\beta)(x,t):=(\bar{\rho}(y_2)-\bar{\rho}(\zeta_2), u(y,t), m\partial_1\eta(y,t),q(y,t))|_{y=\zeta^{-1}(x,t)}.
\label{2022201271250}
\end{align}
Consequently, $(\varrho,v,N,\beta)$ is a strong solution of the magnetic RT problem  \eqref{0103}--\eqref{0105}
and enjoys stability estimates, which are similar to \eqref{1.200}--\eqref{1.200n0} for sufficiently small $\delta$.
\end{rem}
\begin{rem}
In Theorem \ref{thm2}, we have assumed
\begin{align*}
 (\bar{\rho}\eta^0_1)_{\Omega}=(\bar{\rho}u^0_1)_{\Omega}=0 .
\end{align*}
If $(\bar{\rho}\eta^0_1)_{\Omega}$, $(\bar{\rho} u^0_1)_{\Omega} \neq 0$, we  define
$\bar{\eta}^0_1:=\eta^0_1-(\bar{\rho}\eta^0_1)_\Omega (\bar{\rho})_\Omega^{-1}$, $\bar{u}^0_1:=u^0_1-(\bar{\rho}u^0_1)_\Omega(\bar{\rho})_\Omega^{-1}$ and
$( \bar{\eta}^0_2,\bar{u}^0_2):=( {\eta}^0_2, {u}^0_2)$. Then, by virtue of Theorem \ref{thm2}, there exists a unique global strong solution $(\bar{\eta},\bar{u}, {q})$ to the transformed MRT problem   with initial data $( \bar{\eta}^0,\bar{u}^0)$. It is easy to verify that
$(\eta_1,\eta_2,u_1,u_2,q):=(\bar{\eta}_1+ t (\bar{\rho}u^0_1)_\Omega(\bar{\rho})_\Omega^{-1} +  (\bar{\rho}\eta^0_1)_\Omega(\bar{\rho})_\Omega^{-1},\bar{\eta}_2, \bar{u}_1+( \bar{\rho}u^0_1)_\Omega(\bar{\rho})_\Omega^{-1},\bar{u}_2, {q})$
is just the unique strong solution of  the transformed MRT problem  with initial data $({\eta}^0,{u}^0)$.
\end{rem}
\begin{rem}
If additionally, the initial data $(\eta^0,u^0)$ in Theorem \ref{thm2} satisfies the odevity conditions:
\begin{align*}
&
(\eta_1^0,u_1^0)(y_1,y_2)=-(\eta_1^0,u_1^0)(-y_1,y_2),\\[1mm]
&
(\eta_2^0,u_2^0)(y_1,y_2)=(\eta_2^0,u_2^0)(-y_1,y_2).
\end{align*}
then the solution $(\eta,u,q)$ established in Theorem \ref{thm2} also satisfies the odevity conditions:
\begin{align*}
&(\eta_1,u_1)(y_1,y_2,t)=-(\eta_1,u_1)(-y_1,y_2,t),\\ &(\eta_2,u_2)(y_1,y_2,t)=(\eta_2,u_2)(-y_1,y_2,t), \ q(y_1,y_2,t)=q(-y_1,y_2,t).
\end{align*}
Hence we have $\|\eta_1\|_0\lesssim \| \eta_1\|_{1,0}$,
which, together with \eqref{1.200n0}, yields $\eta^\infty_1=0$ in \eqref{1.200xx}.
This presents that all particles of the fluid restore to their initial locations,
and thus the odevity conditions  also strengthen  the stabilizing effect of horizontal magnetic field as well as the fixed condition in \cite{WYJ2019ARMA}.
\end{rem}

Now we roughly sketch the proof of Theorem \ref{thm2}, and the details will be presented in Section \ref{sec:global}. The key step
in the existence proof for global small solutions is to derive an \emph{a priori} energy inequality  \eqref{1.200}.
To this purpose, let $(\eta,u)$ be a solution
to the transformed MRT problem, satisfying that, for some $T>0$,
\begin{align}
&(\bar{\rho} \eta_1)_\Omega\equiv 0\mbox{ for any }t\in \overline{I_T},\label{apresnew}  \\
&  \det(I+\nabla \eta)=1\mbox{ in }\Omega\times \overline{I_T},\label{aprpiosasfesnew}\\
&\sup_{t\in\overline{I_T}} \|(\nabla \eta,u)(t)\|_2\leqslant {\delta} \in (0,1] . \label{aprpiosesnew}
\end{align}

For sufficiently small $\delta$, similarly to \cite{WYJ2019ARMA} where
Wang verified that the vertical magnetic field can inhibit the RT instability
in a stratified incompressible \emph{viscous} MHD fluid in a 3D slab domain,
the first step in our proof is also to derive the tangential energy inequality (i.e.  \eqref{202008250856n0} including the estimates of the both horizontal derivatives and temporal derivative). The next step is to  capture the estimates of high-order normal derivatives of $\eta$. For the vertical magnetic field considered by Wang in \cite{WYJ2019ARMA}, the magnetic tension in 3D case is given by $\lambda m^2\partial_3^2\eta$, which can be rewritten as follows
$$\lambda m^2\Delta \eta-\lambda m^2(\partial_1^2+\partial_2^2 )\eta.$$
Thus the normal estimates of $\nabla \eta$ (not only includes the horizontal derivatives $\partial_1\eta$ and $\partial_2\eta$, but also the normal derivative  $\partial_3\eta$) can be converted into the tangential estimates  by exploiting the regularity theory of Stokes equations. Obviously, this key idea fails to the horizonal magnetic field, and thus we shall seek a new idea.

In view of the first two equations in \eqref{01dsaf16asdfasf}, we easily consider
other  two roads to  capture the high-order normal estimates for $\eta$: one is to use the transport equation \eqref{01dsaf16asdfasf}$_1$, and the other one is to exploit the viscosity term $\Delta_{\mathcal{A}}u$ in the momentum equation \eqref{01dsaf16asdfasf}$_2$.
Since the first road seems to be more difficult, we naturally turn to the second one.
By careful analysis of the structure of \eqref{01dsaf16asdfasf}$_2$, we find that the energy estimates of  $\nabla_{\mathcal{A}}\partial_1^i\mm{curl}_{\mathcal{A}}\eta$ and $\nabla \partial_2 \mm{curl}_{\mathcal{A}}\eta$ (associated with the dissipation estimates of $\partial_1^i\mm{curl}_{\mathcal{A}}\partial_1\eta$ and $\partial_2\mm{curl}_{\mathcal{A}}\partial_1\eta$, resp.)
can  be established under the Navier boundary condition, see Lemma \ref{2055nnn}.
Thus we further derive the normal estimates of $ \eta$
  by using the curl estimates of $\eta$, the nonlinear estimates of $\mm{div}\eta$ and Hodge-type elliptic estimate.

  Summing up the tangential energy inequality  and the $\mm{curl}$-estimates of $\eta$, we can arrive at the total energy inequality
\begin{align}\label{for:0202n}
\frac{\mm{d}}{\mm{d}t}\tilde{\mathcal{E}}+ \mathcal{D}\lesssim \sqrt{\mathcal{E}}
 \mathcal{D}
\end{align}
for some  energy functional $\tilde{\mathcal{E}}$, which is equivalent to $\mathcal{E}$ under \emph{the stability condition \eqref{2020102241504}}. In particular, \eqref{for:0202n} further implies
\begin{align}
2\frac{\mm{d}}{\mm{d}t}\tilde{\mathcal{E}}+ \mathcal{D}\leqslant 0, \label{1.200xyx}
\end{align}  which yields the \emph{priori} stability estimate
\eqref{1.200}.  Thanks to the \emph{priori} estimate \eqref{1.200} and the unique local (-in-time) solvability of the transformed MRT problem  in Proposition \ref{202102182115}, we immediately get the unique global solvability for the transformed MRT problem.
 We mention that the derivation for the \emph{a priori} stability estimate strongly depends on the 2D structures of $\mm{div}\eta$ and $\mm{div}u$.

The decay-in-time estimate \eqref{1.200n0} can be easily observed from linear analysis.
  However the rigours derivation  is very complicated due to the nonlinear terms.
   In \cite{jiang2021asymptotic}, Jiang--Jiang  investigated the decay-in-time of solutions
to the incompressible non-resistive viscous
 MHD  equations in two-dimensional periodic domains, and used a bootstrap method in decay-in-time to obtain the higher  rate of decay-in-time of solutions, similarly to the method to improve the regularity of solutions of elliptic equations. However
    such method is too complicated to be applied our problem.
     To simplify the proof, we fore an additional  \emph{a priori} assumption
     \begin{align}
     \label{202220201915}
 \langle t\rangle^2  (\|\eta\|_{2,1}^2+ \|u\|_2^2)  \leqslant \delta.
     \end{align}
   Then we can also follow the idea in  \cite{jiang2021asymptotic} with simplified derivation to quickly establish \eqref{1.200n0}.
 It should be noted the derivation for the decay-in-time of $ \|u(t )\|_{2} $ in \eqref{1.200n0} is different to the one  in \cite{jiang2021asymptotic}. In fact, Jiang--Jiang  obtained the rate of decay-in-time $ \langle t\rangle^{-1}$ for $\|u(t )\|_{2}$ by directly using the momentum equation \eqref{01dsaf16asdfasf}$_2$ \cite{jiang2021asymptotic}. However,  we further get the better rate of decay-in-time $ \langle t\rangle^{-3/2}$ for $\|u(t )\|_{2}$ by using the estimate of temporal derivative of $u$ and the Stoke estimates, see \eqref{202225}.  Finally, we eaily further get \eqref{1.200xx}  from \eqref{1.200n0} by an asymptotic analysis method.

We can not expect the stability result for the transformed MRT problem under the condition $|m|\in [0,m_{\mm{C}})$. In fact,
this condition results in the RT instability.
\begin{thm}[Instability]\label{thm1}
Let  $\bar{\rho}$ satisfy \eqref{0102} and \eqref{0102n}. If $|m|\in [0,m_{\mm{C}})$,
then the equilibria $(\bar{\rho},0,\bar{M})$ is unstable in the Hadamard sense, that is, there are positive constants $\varpi $, $\epsilon$,
$\delta_0$, and $(\tilde{\eta}^0,\eta^\mm{r},\tilde{u}^0,u^\mm{r})\in  {^0\!\mathcal{H}^3_{\mm{s}}}$,
such that for any $\delta\in (0,\delta_0]$ and the initial data
 $$ (\eta^0, u^0):=\delta(\tilde{\eta}^0,\tilde{u}^0)
 +\delta^2(\eta^\mm{r},u^\mm{r})  , $$
there exists a unique strong solution $(\eta,u,q) $ to the transformed MRT problem \eqref{01dsaf16asdfasf} and  \eqref{20safd45}, where
$(\eta,u,q)\in \widetilde{\mathfrak{H}}^{1,3}_{\gamma, \tau  }\times {^0\mathcal{U}_{\tau}} \times (C^0(\overline{I_\tau},\underline{H}^1)\cap L^2_\tau {{H}^2})$ for any
$\tau\in I_{T^{\max}}$ and  $T^{\max}$ denotes the maximal time of existence of the solution. However,
the solution satisfies
\begin{align}
&\| \bar{\rho}(y_2)-\bar{\rho}(\chi_2(y, T^\delta)+y_2) \|_{L^1},\ \|\chi_i(T^\delta)\|_{L^1}, \ \| \partial_1\chi_{i} (T^\delta)\|_{L^1} , \nonumber  \\
& \|\partial_2\chi_i(T^\delta)\|_{{L^1}}, \ \| \mathcal{A}_{1k}\partial_k \chi_{i} (T^\delta)\|_{L^1},\ \|\mathcal{A}_{2k}\partial_k\chi_i(T^\delta)\|_{{L^1}}  \geqslant  \epsilon \label{201806012326}
\end{align}
for some escape time $T^\delta:= {\Lambda}^{-1}\mm{ln}({2\epsilon}/{\varpi \delta})\in I_T$,
where $i=1$, $2$ and $\chi$ means $\eta$ or $u$.
\end{thm}
\begin{rem}
Following the arguments of Theorem \ref{thm1} and \cite[Corollary 2.2]{JFJSZWC},
it is easy to check that the  corresponding 3D transformed MRT problem is always unstable for any $|m|\geqslant 0$.
\end{rem}
\begin{rem}
By the inverse transformation of Lagrangian coordinates in  \eqref{2022201271250} in Remark \ref{202201271239} and  the instability relation in  \eqref{201806012326}, we easily obtain the instability expressions in Eulerian coordinates: for $i=1$, $2$,
$$\|\varrho(T^\delta)\|_{L^1},\ \|v_i(T^\delta)\|_{L^1},\ \|\partial_1 v_i(T^\delta)\|_{L^1},\ \|\partial_2 v_i(T^\delta)\|_{L^1} \geqslant  \epsilon  $$
and
$$ \| N_i(T^\delta)\|_{L^1} \geqslant  m \epsilon . $$
\end{rem}

The proof of Theorem \ref{thm1} is based on the so-called bootstrap instability method. The bootstrap instability method has its origin
in \cite{GYSWIC,GYSWICNonlinea}, and adapted and generalized by many authors to investigate other flow instabilities, see \cite{FSSWVMNA,GYHCSDDC,JFJSZWC} for examples.
In particular, recently Jiang--Jiang--Zhan proved the existence of the RT instability solution under $L^1$-norm for the stratified viscous,
 non-resistive MHD fluids \cite{JFJSZWC}. In this paper, we will adapt the version of the bootstrap instability method in \cite{JFJSZWC}
 to prove Theorem \ref{thm1}. For the completeness, we will present the detailed proof in Section \ref{sec:instable}.

The rest of this paper is organized as follows. In Sections \ref{sec:global}--\ref{sec:instable}, we provide the proofs for  Theorems \ref{thm2}  and  \ref{thm1} in sequence.  Finally, in \ref{sec:09} we list some mathematical results, which will be used in Sections \ref{sec:global}--\ref{sec:instable}.

\section{Proof of Theorem \ref{thm2}}\label{sec:global}
This section is devoted to the proof of Theorem \ref{thm2}. The key step in the proof is to \emph{a priori}  derive the  total energy estimate  \eqref{1.200}
and the decay-in-time \eqref{1.200n0} for the transformed MRT problem \eqref{01dsaf16asdfasf} and \eqref{20safd45}.
To this end, let $(\eta,u,q)$ be a solution to the transformed MRT problem,  and satisfy \eqref{apresnew}--\eqref{aprpiosesnew},
where ${\delta}$ is sufficiently small, and the smallness of
$\delta$ depends on $\mu $, $g$, $\lambda$, $m$, $\bar{\rho}$ and $\Omega$. It should be noted  that   $m$ and $\bar{\rho}$ satisfy
the assumptions in Theorem \ref{thm2}. Next, we start with \emph{a priori} estimates.
 \subsection{Preliminary estimates}
First, we shall establish some preliminary estimates involving $(\eta,u)$.
\begin{lem}[Nonlinear estimates]
\label{201805141072}
 For any given $t\in \overline{I_T}$, we have
\begin{enumerate}[(1)]
\item  the estimates of $\mm{div} { {\eta}}$:
\begin{align}
&\label{improxtian1}
\|\mm{div} { {\eta}}\|_{i} \lesssim
\|\nabla \eta\|_2\| \eta\|_{1,i}\mbox{ for }0\leqslant i\leqslant  2,\\
&\label{improtian1}
\|\mm{div} { {\eta}}\|_{i,0} \lesssim
\begin{cases}
\|\nabla \eta\|_2\|\eta\|_{1+i,0}&\mbox{for }i=0,\ 1;\\
 \|\eta\|_{3,0}\|\nabla \eta\|_2+ \|\eta\|_{2,1}^2 &\mbox{for }i=2. \end{cases}  \end{align}
\item the estimate  involving the gravity term: for sufficiently small $\delta$,
\begin{align}
&\left\| \mathcal{G}\right\|_{\underline{1},0}\lesssim
\|\eta_2\|_0\|\eta_2\|_{\underline{2},1}  \label{2022011130957}
\end{align}
where $\mathcal{G}:=G_{\eta}-g  \bar{\rho}'\eta_2 $.
\item the estimates involving cur:
\begin{align}
&\label{201907291432} \|  \operatorname{curl}_{\tilde{\mathcal{A}}} \eta  \|_2\lesssim\| \eta\|_{1,2}\|\nabla \eta\|_2,\\
&\| \mm{curl}_{\partial_1\mathcal{A}}    \partial_1\eta \|_{ 0}\lesssim \|\eta\|_{2,1}^2.  \label{202220202011843}
\end{align}
\end{enumerate}
\end{lem}
\begin{rem}
Here and in what follows, we define  $\tilde{\mathcal{A}}:=\mathcal{A}-I$. Then
\begin{align}
\tilde{\mathcal{A}}=
\begin{pmatrix}
\partial_2\eta_2 \quad-\partial_1\eta_2\\[1mm]
-\partial_2\eta_1\quad\partial_1\eta_1
             \end{pmatrix}. \nonumber
\end{align}
\end{rem}
\begin{pf}
(1)
Recalling \eqref{aprpiosasfesnew}, we can calculate that
\begin{align}
&\mm{div}\eta=\partial_1\eta_2\partial_2\eta_1-\partial_1\eta_1\partial_2\eta_2 \nonumber
\end{align}
and \begin{align}
\partial_1\mm{div}\eta=\partial_1 \eta_2\partial_2\partial_1 \eta_1+\partial_1^2\eta_2\partial_2\eta_1-\partial_1\eta_1\partial_2\partial_1  \eta_2
-\partial_1^2\eta_1\partial_2\eta_2
. \label{202201162022}
\end{align}
Exploiting   the product estimates \eqref{fgestims}, \eqref{fgessfdims} and Poincar\'e's inequality \eqref{202012241002}, it is easy to see from the above two relations that \eqref{improxtian1} and \eqref{improtian1} hold for $i=0$, $1$.
 If we further apply $\partial_1$ to the  identity \eqref{202201162022}, we also check that $\|\mm{div} { {\eta}}\|_{2,0}$ satisfies \eqref{improtian1} with $i=2$.


(2)
  By virtue of \eqref{aprpiosesnew} and Lemma \ref{pro:1221}, $\zeta:=\eta+y$ satisfies the diffeomorphism  properties \eqref{20210301715x} and \eqref{20210301715} for sufficiently small $\delta$. Thus, $\bar{\rho}^{(j)}(y_2+\eta_2)$ for any $y\in\overline{\Omega}$ makes sense, and
\begin{align}
\label{20200830asdfa2114}
\bar{\rho}^{(j)}(y_2+\eta_2)-\bar{\rho}^{(j)}(y_2) =  \int_{0}^{\eta_2} \bar{\rho}^{(j+1)}(y_2+z)\mm{d}z
 \mbox{ for }   j=0, \ 1.
\end{align}
Moreover, for any given $t\in \overline{I_T}$,
\begin{align}
&\label{esmmdforiasdfaasfanfty}
 \sup_{y\in \overline{\Omega}} \sup_{z\in \Psi}\left| \bar{\rho}^{(j+1)}(y_2+z)\right| \lesssim 1,
\end{align}
where $\Psi:= \{\tau~|~0\leqslant\tau\leqslant\eta_2\}$ for $\eta_2\geqslant 0$ and $:=(\eta_2,0]$ for $\eta_2< 0$.

Making use of \eqref{20200830asdfa2114}, \eqref{esmmdforiasdfaasfanfty},  \eqref{fgessfdims}  and the relation
\begin{align}\nonumber
\bar{\rho}(y_2+\eta_2)-\bar{\rho}(y_2) =\bar{\rho}'(y_2)\eta_2+{ \int_{0}^{\eta_2}\left(\eta_2(y,t) -
z\right)\bar{\rho}''(y_2+z)\mm{d}z},
\end{align}
 it is easy to estimate that
$$\begin{aligned}
& \left\|\mathcal{G}\right\|_0 =g\left\|\int_{0}^{\eta_2}\left(\eta_2(y,t) -
z\right)\bar{\rho}^{''}(y_2+z)\mm{d}z \right\|_0 \lesssim\|\eta_2^2\|_{0} \lesssim\|\eta_2\|_{0}\|\eta_2\|_{\underline{1},1}
 \end{aligned}
$$
and
\begin{align}
\left\| \mathcal{G}\right\|_{1,0}
&=g \left\|\big(\bar{\rho}'(y_2+\eta_2)-\bar{\rho}'\big)\partial_1
\eta_2\right\|_0\lesssim\|\eta_2\partial_1\eta_2\|_{0}  \lesssim
\|\eta_2\|_0\|\partial_1\eta_2\|_{\underline{1},1} .
\label{202022011201820}
\end{align}
Thanks to the above two estimates, we immediately obtain   \eqref{2022011130957}.

(3) Noting that, for $0\leqslant k+l\leqslant 2$,
\begin{align*}
&\operatorname{curl}_{\partial_1^k\tilde{\mathcal{A}}  } \partial_1^l\eta= (\partial_2\partial_1^k\eta_1\partial_{1}-\partial_1^{1+k}\eta_1\partial_{2})\partial_1^l\eta_1+ (\partial_2\partial_1^k\eta_2\partial_{1}-\partial_1^{1+k}\eta_2\partial_{2})\partial_1^l\eta_2 ,
\end{align*}
thus it is easy to check that \eqref{201907291432} and \eqref{202220202011843}  hold by following the arguments  of \eqref{improxtian1} and \eqref{improtian1}.
\hfill $\Box$
\end{pf}

\begin{lem}\label{lem:202012242115}
We have
\begin{enumerate}[(1)]
\item the estimate of $\eta_2$:
\begin{align}
&\label{omessetsim122n}
\|\eta_2\|_i\lesssim
\begin{cases}
\|\eta\|_{1,0} &\mbox{for } i=0,\ 1;\\
\|\eta\|_{1,i-1}&\mbox{for } i=2,\ 3.
\end{cases}
\end{align}
\item Poincar\'e inequality for $\eta$, $u$ and $u_t$:  for $j=1$ and $2$,
\begin{align}
&\|  \eta_j\|_1\lesssim \|\nabla \eta_j\|_0 ,  \label{202211}\\
&\|    u_j\|_1\lesssim \| \nabla u_j\|_0, \label{2022111522311}\\
&\|  \partial_t u_j\|_1\lesssim \|\nabla \partial_t u_j\|_0. \label{2022220111522311}
\end{align}
\item  the estimate  involving the gravity term: for sufficiently small $\delta$,
\begin{align}
 \left\| \mathcal{G}\right\|_{1}\lesssim \|  \eta_2 \|_{ 0} ,
\label{2022201201821}
\end{align}
\item curl estimates:   for sufficiently small $\delta$,
\begin{align}
\|    \eta\|_{k,3-k}\lesssim   \|   \mm{curl}  \eta  \|_{k,2-k}\mbox{ where }0\leqslant k \leqslant 2 .\label{2022202011749}
\end{align}
\end{enumerate}
\end{lem}
\begin{pf}
 (1) Noting that
\begin{align}
\eta_2|_{\partial\Omega}=0 \mbox{ and } \partial_2\eta_2=\mm{div}\eta-\partial_1\eta_1,
\label{2020101131413}
\end{align} we use \eqref{aprpiosesnew}, \eqref{improxtian1},  \eqref{2020101131413}  and \eqref{poinsafdcasfaressadf1} to get
\begin{align}
&\nonumber \|\eta_2\|_0\lesssim\| (\partial_1\eta_1,\mm{div}\eta)\|_{0}\lesssim \|\eta\|_{1,0},\\
& \nonumber
\|\eta_2\|_1\lesssim\|\eta_2\|_{0}+\|\nabla\eta_2\|_{0}\lesssim\|(\partial_1\eta ,\mm{div}\eta)\|_{0}\lesssim \|\eta\|_{1,0},\\
&
\|\eta_2\|_{i}\lesssim
\|\eta_2\|_0+ \|\eta_2\|_{1,i-1} + \|\partial_2\eta_2\|_{i-1} \lesssim
 \|\eta\|_{1,i-1} \mbox{ for }i=2,\ 3 .
\end{align}
Thus, we immediately get \eqref{omessetsim122n}   from the above four estimates.

(2) By \eqref{01dsaf16asdfasf}$_1$ and \eqref{apresnew}, it is easy to see that
\begin{align}
( \bar{\rho}u_1)_\Omega= 0
\label{2022202052011}
\end{align}
Thus the estimates \eqref{202211} and \eqref{2022111522311} obviously hold due to \eqref{apresnew}, \eqref{2022202052011}, \eqref{poinsafdcasfaressadf1}, Lemma \ref{lem:0102} and the boundary  condition $(\eta_2, u_2)|_{\partial\Omega}=0$.

  Multiplying \eqref{01dsaf16asdfasf}$_2$ by $\mathbf{e}_1$ in $L^2$, and then using the integral by parts and the relation \begin{align}
\label{202201152005}
\partial_j(\partial_t^k\mathcal{A}_{ij}f)=\partial_t^k\mathcal{A}_{ij}\partial_jf\mbox{ for }k=0,\ 1,
\end{align} we get
$$ \int \bar{\rho}\partial_t u_1 \mm{d}y+ \int_{\partial\Omega} \vec{n}_2\left(  \mathcal{A}_{12}q -\mu \mathcal{\mathcal{A}}_{i2}\mathcal{\mathcal{A}}_{ij} \partial_j u_1\right)\mm{d}y=0 ,$$
which, together with the boundary condition \eqref{20safd45}, yields
\begin{align}
\label{20222011181844}
( \bar{\rho} \partial_t u_1)_\Omega=0.
\end{align}
Thanks to \eqref{20222011181844} and Lemma \ref{lem:0102}, thus we  have \eqref{2022220111522311} for $j=1$.
 Noting that $\partial_t u_2|_{\partial \Omega}=0$, thus, by \eqref{poinsafdcasfaressadf1}, we also have \eqref{2022220111522311} for $j=2$. Hence \eqref{2022220111522311} holds.

(3) Similarly to \eqref{202022011201820}, we can estimate that
 \begin{align*}\left\|\partial_2\mathcal{G}\right\|_0
=&g\left\| \big(\bar{\rho}'(y_2+\eta_2)-\bar{\rho}'\big)(1+\partial_2
\eta_2)-\bar{\rho}''\eta_2 \right\|_0
  \\
 & \lesssim \| ( \eta_2, \eta_2  \partial_2\eta_2 )\|_0\lesssim \|  \eta_2  \|_0,
 \end{align*}
 which, together with \eqref{2022011130957}, yields \eqref{2022201201821}.

 (4) Making use of \eqref{improxtian1}, \eqref{improtian1}, \eqref{202211}, \eqref{202012241002} and
\eqref{202005021302}, we have
\begin{align*}
\|    \eta\|_{k,3-k}\lesssim \|    \nabla \eta\|_{k,2-k}\lesssim  \| (  \mm{curl}  \eta  ,\mm{div}   \eta ) \|_{k,2-k}\lesssim    \|    \mm{curl}  \eta \|_{k,2-k}
+\|\nabla \eta\|_2\|\eta\|_{k,3-k} ,
\end{align*}
which yields \eqref{2022202011749} for sufficiently small $\delta$.
\hfill $\Box$
\end{pf}

\subsection{Tangential estimates}\label{subsec:Horizon}
This section is devoted to establishing the tangential estimates by the following three lemmas, which include the estimates of horizontal derivatives of $(\eta,u)$ and temporal derivative of $u$.
\begin{lem}\label{lem:082sdaf41545}
For sufficiently small $\delta$, it holds that,  for $0\leqslant i\leqslant 2$,
\begin{align}
&
\frac{\mm{d}}{\mm{d}t}\left(\int\bar{\rho}\partial_1^{i}\eta\cdot\partial_1^{i} u\mm{d}y
+\frac{\mu}{2}\| \nabla \partial_1^{i} \eta\|_{0}^2\right)
-E(\partial_1^{i}\eta)
\nonumber \\
& \leqslant   \|\sqrt{\bar{\rho}}\partial_1^{i}u\|_{0}^2
+\|\nabla \eta\|_{2}\|\eta\|_{3,0} ( \|\eta_2\|_{0}+ \|u\|_{\underline{1},2}+\|q\|_{\underline{1},1}
 )\nonumber \\
& + \|\eta\|_{2,1}^2 \|q\|_{\underline{1},1}+\begin{cases}
\|\eta\|_{1,1}\|\nabla \eta\|_{1}\| u\|_1 &\mbox{for }i=0;\\
0 &\mbox{for } i=1,\ 2.
 \end{cases}
\label{202008241446}
\end{align}
\end{lem}
\begin{pf} We apply $\partial_1^{i}$ to \eqref{01dsaf16asdfasf} and \eqref{20safd45},  and then use the relation
\eqref{202201152005}
 to derive that
\begin{equation}\label{01dsaf16asdfasf03n}
                              \begin{cases}
\partial_1^{i}\eta_t=\partial_1^{i}u ,\\[1mm]
\partial_1^{i}(\bar{\rho}u_t-\mu\Delta  u)
=\partial_1^{i}(\lambda m^2 \partial_1^{ 2}\eta+g\bar{\rho}'\eta_2\mathbf{e}_2
+ \mathcal{G}\mathbf{e}_2+\mathcal{N}^\mu-  \nabla_{\mathcal{A}} q ) ,\\[1mm]
[\partial_1^{i}, \mathcal{A}_{kl} ]\partial_l u_k+\mathcal{A}_{kl}\partial_1^{i}\partial_l u_k =0 , \\[1mm]
\partial_1^{i}(\eta_2, u_2,\partial_2(\eta_1, u_1))|_{\partial\Omega} =0,
\end{cases}
\end{equation}
where
\begin{align*}
 \mathcal{N}^\mu :=& \partial_l(\mathcal{N}^\mu _{1,l}, \mathcal{N}^\mu _{2,l})^{\mm{T}} ,\\
\mathcal{N}^\mu _{j,1} :=  &\mu (\mathcal{A}_{k1}\tilde{\mathcal{A}}_{km}+\tilde{\mathcal{A}}_{m1} )\partial_mu_j\\
=&\mu (
( 2\partial_2\eta_2+ (\partial_2\eta_1)^2+ (\partial_2\eta_2)^2 )\partial_1u_j
-  \Theta\partial_2u_j), \\
\mathcal{N}^\mu _{j,2} :=  &\mu (\mathcal{A}_{k2}\tilde{\mathcal{A}}_{km}+\tilde{\mathcal{A}}_{m2} )\partial_mu_j\\
=&\mu ((2\partial_1\eta_1 + (\partial_1\eta_1)^2+(\partial_1\eta_2)^2 )\partial_2u_j  -\Theta\partial_1u_j)
\end{align*}
and
$$ \Theta:=\partial_1\eta_2+\partial_2\eta_1 + \partial_1\eta_2\partial_2\eta_2  + \partial_1\eta_1 \partial_2\eta_1 .$$
Moreover, by the boundary condition \eqref{20safd45}, we have
\begin{align}
\partial_1^i\mathcal{N}^\mu _{1,2}|_{\partial\Omega}=0.  \label{202220112002138}
\end{align}

Let $0\leqslant i\leqslant 2$.
Multiplying \eqref{01dsaf16asdfasf03n}$_2$ by $\partial_1^{i}\eta$, then we us the
integral by parts,  \eqref{01dsaf16asdfasf03n}$_1$ and \eqref{01dsaf16asdfasf03n}$_4$ to obtain
\begin{align}
&\frac{\mm{d}}{\mm{d}t}\left(\int\bar{\rho}\partial_1^{i}\eta\cdot\partial_1^{i} u\mm{d}y
+\frac{\mu }{2}\| \nabla \partial_1^{i} \eta\|_{0}^2\right)
-E(\partial_1^{i}\eta) =\|\sqrt{\bar{\rho}}\partial_1^{i}u\|_{0}^2+\sum_{j=1}^4I_{j,i},\label{202008241510}
\end{align}
where we have defined that
\begin{align*}
&I_{1,i}:=\int\partial_1^{i} \mathcal{G}\partial_1^{i} \eta_2\mm{d}y,\ I_{2,i}:=-  \int\partial_1^{i} \mathcal{N}^\mu _{j,1}  \partial_1^{i+1}\eta_j\mm{d}y \\
&I_{3,i}:=\int \partial_2\partial_1^{i} \mathcal{N}^\mu _{j,2} \partial_1^{i}\eta_j\mm{d}y\mbox{ and }I_{4,i}:=-\int \partial_1^{i}\nabla_{\mathcal{A}}q\cdot\partial_1^{i} \eta\mm{d}y.
\end{align*}
Next we estimate for the above four integrals $I_{1,i}$--$I_{4,i}$ in sequence.

  (1) By the integral by parts, H\"older's inequality  and \eqref{2022011130957}, we infer that
\begin{align}
&\label{202008241546}
I_{1,i}\leqslant
\begin{cases}
\| \eta_2\|_{0}\| \mathcal{G}\|_{0}\lesssim
\|\eta_2\|_{0}^2 \|\eta_2\|_{\underline{2},1}  &\mbox{for }i=0;\\
\| \eta_2\|_{1+i,0}\| \mathcal{G}\|_{i-1,0}\lesssim
\|\eta_2\|_0\|\eta_2\|_{\underline{2},1}   \| \eta_2\|_{1+i,0} &\mbox{for }i=1,\ 2.
\end{cases}
\end{align}

(2) Exploiting H\"older's inequality, \eqref{aprpiosesnew}, \eqref{fgestims} and \eqref{fgessfdims}, we can see that
\begin{align}
I_{2,i}=&  \mu \int\partial_1^{i+1}\eta \cdot \partial_1^{i} ( \Theta\partial_2u-
( 2\partial_2\eta_2+ (\partial_2\eta_1)^2+ (\partial_2\eta_2)^2 )\partial_1u
  ) \mm{d}y  \nonumber \\
\lesssim&
\|\nabla \eta\|_{2}\|\eta\|_{i+1,0}\|  u\|_{\underline{1},2}  . \label{2022201161405} \end{align}

(3)
Using  the integral by parts,  \eqref{01dsaf16asdfasf03n}$_4$, \eqref{202220112002138}  and  the product estimate \eqref{fgestims},  we  can estimate
\begin{align}
I_{3,0}
\nonumber = &  \mu \int  \partial_2\eta \cdot(  \Theta\partial_1u -(2\partial_1\eta_1 + (\partial_1\eta_1)^2+(\partial_1\eta_2)^2 )\partial_2u  )\mm{d}y\\
= &   \mu \int (\partial_2\eta\cdot ((\partial_1\eta_2 + \partial_1\eta_2\partial_2\eta_2  + \partial_1\eta_1 \partial_2\eta_1 )\partial_1u- (2\partial_1\eta_1 + (\partial_1\eta_1)^2\nonumber \\
 &+(\partial_1\eta_2)^2 )\partial_2u ) -\partial_2\eta_1\partial_2\partial_1\eta \cdot u-\partial_2\partial_1\eta_1\partial_2\eta\cdot u)\mm{d}y\nonumber \\
 \lesssim &\|\eta\|_{1,1}\|\nabla \eta\|_{1}\| u\|_1   .   \label{2061032}
\end{align}
 Similarly to \eqref{2022201161405}, we also have
\begin{align}
 I_{3,i}
&=  \mu \int \partial_1^{i+1} \eta \cdot \partial_2\partial_1^{i-1}  (\Theta\partial_1u -(2\partial_1\eta_1 + (\partial_1\eta_1)^2+(\partial_1\eta_2)^2 )\partial_2u  ) \mm{d}y\nonumber \\
&  \lesssim\|\nabla \eta\|_{2}\|\eta\|_{i+1,0}\| u\|_{\underline{1},2}    \mbox{ for }i=1,\ 2.   \label{20222011161447} \end{align}

(4)
Finally we bound the last integral $I_{4,i}$. Noting
$$\mm{div}_{\tilde{\mathcal{A}}}\eta=2(\partial_1\eta_1 \partial_2\eta_2-\partial_1\eta_2\partial_2 \eta_1),$$
  making use of the above identity, the integral by parts, \eqref{improtian1}, \eqref{202201152005}, \eqref{01dsaf16asdfasf03n}$_4$, \eqref{fgessfdims} and  \eqref{202012241002}, we can estimate that
\begin{align}
I_{4,0} =\int \mm{div}_{\tilde{\mathcal{A}}}\eta q\mm{d}y+ \int \mm{div}\eta q\mm{d}y
\lesssim   \| \eta\|_{1,0} \|\nabla \eta\|_2\|  q\|_0   \label{201910040902n02}
\end{align}
and, for $i=1$, $2$,
\begin{align}
I_{4,i}
=&\int\partial_1^{i+1} \eta \cdot\partial_1^{i-1}\nabla_{\tilde{\mathcal{A}}}q\mm{d}y
+\int\mm{div}\partial_1^{i} \eta\partial_1^{i}q\mm{d}y\nonumber \\[1mm]
\leqslant&
\|\eta\|_{{i+1},0}\| \nabla_{\tilde{\mathcal{A}}}q\|_{{i-1},0}
+\|\mm{div}\eta\|_{i,0}\|q\|_{i,0}
\nonumber \\
\lesssim& (\|\eta\|_{3,0} \|\nabla \eta\|_2 +\| \eta\|_{2,1}^2 )\|q\|_{\underline{1},1} . \label{202008241550}
\end{align}

Consequently, putting \eqref{202008241546}--\eqref{202008241550} into \eqref{202008241510}, and then using \eqref{omessetsim122n} and \eqref{202012241002}, we arrive at \eqref{202008241446}.
This completes the proof.
\hfill $\Box$
\end{pf}

\begin{lem}\label{lem:08241445}
For sufficiently small $\delta$, it holds that,  for $0\leqslant i\leqslant 2$,
\begin{align}
& \frac{\mm{d}}{\mm{d}t}\left(\|\sqrt{ \bar{\rho} }  u\|^2_{i,0}
-E(\partial_1^{i} \eta) \right)+ c\|   u\|_{i,1}^2 \nonumber \\
&\lesssim
 \|\eta\|_{2,1} (\|\eta_2\|_{0}  + \|  u\|_{\underline{1},2}  ) \| u\|_{i,1} + ( \|\eta\|_{3,0} \| u\|_{\underline{1},2}+\|\eta\|_{2,1} \| u\|_{2,1})\| q\|_{\underline{1},1 }     . \label{202008241448}
\end{align}
\end{lem}
\begin{pf}
Multiplying \eqref{01dsaf16asdfasf03n}$_2$ by  $\partial_1^{i}u$ in $L^2$, and then using the integrate by parts, \eqref{01dsaf16asdfasf03n}$_1$ and the boundary conditions \eqref{01dsaf16asdfasf03n}$_4$ and  \eqref{202220112002138}, we have
\begin{align}
& \frac{1}{2}\frac{\mm{d}}{\mm{d}t}\left(\|\sqrt{\bar{\rho}} u\|_{i,0}^2
-E(\partial_1^{i}\eta)\right)+\mu \|  \nabla u\|_{i,0}^2
\nonumber\\
 &=\int\partial_1^{i} \mathcal{G}\partial_1^{i} u_2\mm{d}y-  \int \partial_1^{i} \mathcal{N}^\mu _{j,l}\partial_l \partial_1^{i}u_j\mm{d}y-\int \partial_1^{i}\nabla_{\mathcal{A}}q \cdot\partial_1^{i}u\mm{d}y=:\sum_{j=5}^7I_{j,i}.\label{201510n}
\end{align}

Similarly to
\eqref{202008241546}, we have \begin{align}
& \label{202008241624}
I_{5,i} \lesssim \begin{cases}
\|\eta_2\|_{0}  \|\eta_2\|_{\underline{2},1} \|u_2\|_{0} &\mbox{for }i=0;\\
\|\eta_2\|_0\|\eta_2\|_{\underline{2},1}   \| u_2\|_{1+i,0} &\mbox{for }i=1,\ 2.
\end{cases}
\end{align}
By \eqref{fgestims} and \eqref{fgessfdims}, we can estimate that
\begin{align}
I_{6,i}=& \mu \int (\partial_1^{i}(
( 2\partial_2\eta_2+ (\partial_2\eta_1)^2+ (\partial_2\eta_2)^2 )\partial_1u
-  \Theta\partial_2u ) \cdot \partial_1^{i+1}u
\nonumber  \\
&  +\partial_1^{i} ((2\partial_1\eta_1 + (\partial_1\eta_1)^2+(\partial_1\eta_2)^2 )\partial_2u   -\Theta\partial_1u ) \cdot\partial_2\partial_1^{i}u)\mm{d}y
 \nonumber \\
 \lesssim  &  \|\partial_1 \eta\|_{\underline{1},1}\|  u\|_{\underline{1},2} \|   u\|_{i,1} +\|\nabla \eta\|_{2}\|   u\|_{i,1}^2.  \label{2022202181730}
\end{align}

Similarly, for $i$=1, $2$,
 \begin{align}
  I_{7,i,1}:=-\int  [\partial_1^{i},\mathcal{A}_{kl}]\partial_l q \partial_1^{i}u_k \mm{d}y    \lesssim  \| \partial_1\eta\|_{\underline{1},1}\|  u\|_{i,1}\| q\|_{\underline{1},1 }  \nonumber
\end{align}
and \begin{align}
 I_{7,i,2}=&-\int [\partial_1^{i}, \mathcal{A}_{kl} ]\partial_l u_k\partial_1^{i} q  \mm{d}y  \nonumber \\
 =& \int ([\partial_1^{i},\partial_1 \eta_2 ]\partial_2 u_1-[\partial_1^{i}, \partial_2\eta_2]\partial_1 u_1+[\partial_1^{i},\partial_2 \eta_1  ]\partial_1 u_2 -[\partial_1^{i},\partial_1 \eta_1  ]\partial_2 u_2)\partial_1^{i} q  \mm{d}y  \nonumber \\
  \lesssim   &(  \|\partial_1 \eta\|_{\underline{2},0}\| u\|_{\underline{1},2}+\|\partial_1 \eta\|_{\underline{1},1}\| \partial_1 u\|_{\underline{1},1})\| q\|_{i,0 } .   \nonumber
\end{align}
Making use of the  integral by parts, \eqref{202201152005}, \eqref{01dsaf16asdfasf03n}$_3$, \eqref{202012241002} and the above two estimates, we have
 \begin{align}
 I_{7,i}
=&\begin{cases}
0&\mbox{for }i=0;\\
  I_{7,i,1}+  I_{7,i,2}
&\mbox{for }i=1,\ 2\end{cases}\nonumber \\
\lesssim   &(    \|\partial_1 \eta\|_{\underline{2},0}\|  u\|_{\underline{1},2}+\| \partial_1\eta\|_{\underline{1},1}\| \partial_1 u\|_{\underline{1},1})\| q\|_{\underline{1},1 } .\label{20200n}
\end{align}

Consequently, putting \eqref{202008241624}--\eqref{20200n} into \eqref{201510n}, and then using \eqref{aprpiosesnew}, \eqref{omessetsim122n}, \eqref{2022111522311} and \eqref{202012241002}, we arrive at \eqref{202008241448} for sufficiently small $\delta$. This completes the proof. \hfill $\Box$
\end{pf}
\begin{lem}\label{2019100216355nnn}
For sufficiently small $\delta$, we have
\begin{align}
&\frac{\mm{d}}{\mm{d}t}\|\nabla_{\mathcal{A}} u \|^2_0 +
c\|   u_t\|_0^2\lesssim \| \eta \|_{2,0}^2+ \|u\|_2^3+\|u\|_2^2\|q\|_1
\label{202221235}
\end{align}
and
\begin{align}
&\frac{\mm{d}}{\mm{d}t}\left(\|\sqrt{\bar{\rho}}\psi \|^2_0- E(  u) \right)+
c\|   u_t\|_1^2\lesssim(\|  \eta \|_{2,0} +\|u \|_{2 }      ) \|u\|_2^2 ,
\label{20222saf201121235}
\end{align}
where $\psi :=u_t- {u} \cdot \nabla_{\mathcal{A}} u $.
\end{lem}
\begin{pf}
(1)
By \eqref{01dsaf16asdfasf}$_3$, we see that
$$\div_{\mathcal{A}} u_t=-\div_{\mathcal{A}_t} u. $$
  Multiplying \eqref{01dsaf16asdfasf}$_2$ by $u_t$ in $L^2$, and then using the integral by parts,  \eqref{20safd45}, \eqref{202201152005} and the above relation, we obtain
  \begin{align}\label{01sdfasf}
\frac{\mu}{2} \frac{\mm{d}}{\mm{d}t}\|\nabla_{\mathcal{\mathcal{A}}} u\|_0^2 +\|\sqrt{\bar{\rho}}u_t\|_0^2  =I_8,
\end{align}
where we have defined that
\begin{align*}
& I_8:= \int (\lambda m^2\partial_1^2\eta+ (g\bar{\rho}'\eta_2+\mathcal{G})\mathbf{e}_2 )\cdot u_t\mm{d}y+\int ( \mu \nabla_{\mathcal{\mathcal{A}}} u :\nabla_{\mathcal{\mathcal{A}}_t} u +   \nabla q \cdot  ({\mathcal{A}_t^{\mm{T}} } u ))\mm{d}y.
   \end{align*}

Exploiting \eqref{2022011130957} and \eqref{fgestims},
we get
\begin{align*}
& I_8 \lesssim \|(\partial_1^2\eta, {\eta}_2)\|_0\| u_t\|_0+ \|u\|_2^3+\|u\|_2^2\|q\|_1.
   \end{align*}
   Putting the above estimate into \eqref{01sdfasf}, and then using \eqref{omessetsim122n} and   \eqref{202012241002}, we get \eqref{202221235}.

(2)
Let
\begin{align*}
   & I_9:=\int
(( \mu \mathcal{A}_{il} \partial_l( \mathcal{A}_{ik} \partial_k  u )   + \lambda m^2\partial_1^2\eta +(\mathcal{G}+g  \bar{\rho}'\eta_2 )\mathbf{e}_2
 \\
&\qquad-  \bar{\rho} u\cdot\nabla_{\mathcal{A}}    u   -\bar{\rho}\psi) \cdot ( u\cdot\nabla_{\mathcal{A}} \psi)-   \partial_t (  \bar{\rho} u\cdot\nabla_{\mathcal{A}}    u ) \cdot \psi ) \mm{d}y,
 \\
&I_{10}:=   \int ((\lambda m^2\partial_1^2 u+g\bar{\rho}'(y_2+\eta_2)u_2\mathbf{e}_2)\cdot \psi-   \mu
  \partial_t ( \mathcal{A}_{il} \mathcal{A}_{ik} \partial_k  u ) \cdot \partial_l \psi )\mm{d}y .
   \end{align*}
Recalling the derivation of \eqref{eq0510} in Section \ref{202202021613} and the relation
$$\frac{1}{2} \int |\psi|^2 u\cdot \nabla_{\mathcal{A}}  \bar{\rho}\mm{d}y=  -\int  \bar{\rho}u\cdot\nabla_{\mathcal{A}} \psi\cdot \psi\mm{d}y ,$$ we can get the following identity  (i.e. taking  $J=1$, $w=u$ and $f=\lambda m^2\partial_1^2\eta +(\mathcal{G}+g  \bar{\rho}'\eta_2 )\mathbf{e}_2$  in \eqref{eq0510}) from \eqref{01dsaf16asdfasf}:
\begin{align}
 \frac{1}{2}\frac{\mm{d}}{\mm{d}t}\|\sqrt{\bar{\rho}}\psi \|^2_0
=I_{9} + I_{10}.  \label{62053}
  \end{align}

The integral term $I_{10}$ can be further rewritten as follows:
\begin{align}
 I_{10}=\frac{1}{2}\frac{\mm{d}}{\mm{d}t}  E(  u)-
 \mu \|  \nabla u_t\|_0^2+\tilde{I}_{10},  \label{2022101262053}
  \end{align}
  where
 \begin{align*}
 \tilde{I}_{10}:= & \int ( g(\bar{\rho}'(y_2+\eta_2) - \bar{\rho}'(y_2))u_2\partial_t u_2-  (\lambda m^2\partial_1^2 u   + g \bar{\rho}'(y_2+\eta_2)u_2 \mathbf{e}_2   )\cdot({u} \cdot \nabla_{\mathcal{A}} u)  \\
 &+ \mu ( \partial_t ( \mathcal{A}_{il} \mathcal{A}_{ik} \partial_k  u) \cdot \partial_l(u \cdot \nabla_{\mathcal{A}} u)-\partial_t(\mathcal{A}_{il}\tilde{\mathcal{A}}_{ik} \partial_k  u +  \tilde{\mathcal{A}}_{il}\partial_i  u  )\cdot \partial_l \partial_t u ) )\mm{d}y.
  \end{align*}

Making use of \eqref{aprpiosesnew}, \eqref{01dsaf16asdfasf}$_1$, \eqref{20200830asdfa2114}, \eqref{2022201201821} and \eqref{fgestims},
we have
$$
I_{9}+\tilde{I}_{10}\lesssim ( \|\eta_2\|_0+\|  \eta \|_{2,0}+\|u\|_2 +\|u_t\|_1)\|u\|_2 \|u_t\|_1 + \|u\|_2^3 +\|\nabla \eta\|_2\|u_t\|_1^2.
$$
Putting \eqref{2022101262053} to \eqref{62053}, and then using  \eqref{omessetsim122n}, \eqref{2022220111522311}, \eqref{202012241002},  the above estimate   and  Young's inequality, we immediately get \eqref{20222saf201121235} for sufficiently small $\delta$. This completes the proof. \hfill $\Box$
 \end{pf}

\subsection{Curl estimates for $\eta$}\label{subec:Vor}
This  section is devoted to establishing the curl estimates of $\eta$ for the normal estimates of $\eta$.
\begin{lem}\label{2055nnn}
Let the multiindex $\alpha$ satisfy $|\alpha|\leqslant 1$. We have
\begin{align}
&\label{202005021600}
 \frac{\mm{d}}{\mm{d}t}  \left(\frac{\mu }{2}
 \|\nabla_{\mathcal{A}} \partial^{\alpha}\mm{curl}_{\mathcal{A}}\eta \|_0^2+
 \int\partial^{\alpha}\mm{curl}_{\mathcal{A}}   \eta
 \partial^{\alpha}\mm{curl}_{\mathcal{A}} \left(\bar{\rho}u\right)
 \mm{d}y\right)
 + c\| \partial^{\alpha}\mm{curl}_{\mathcal{A}} \partial_1\eta \|^2_{0}\nonumber\\
&\lesssim
 \|\eta\|_{1,0}\|\eta\|_{1+|\alpha|,1} +\|  u\|_{1 }\|u\|_3 + \|\nabla \eta\|_2(\|\eta\|_{1+|\alpha|,1}^2\nonumber \\
 &\quad+ \|\eta\|_{1,2}\|  u\|_{\underline{1},2}+\|u\|_2^2)
   \mbox{ for }\alpha_2\neq 1
           \end{align}
 and
    \begin{align}
 &\frac{\mm{d}}{\mm{d}t}
\left(
 \frac{\mu }{2}\|   \nabla \partial_2 \mm{curl}_{\mathcal{A}}\eta\|^2_0+\mathcal{I}_2\right)
+c\|  \partial_2\mm{curl}_{\mathcal{A}} \partial_1 \eta \|^2_{0}
\nonumber \\
& \lesssim  \|  \eta\|_{1,0}\|\eta\|_{1,2}+\|u\|_1\|  u\|_3+\sqrt{\mathcal{E}}\mathcal{D}, \label{202201211416}  \end{align}
           where
\begin{align*}
&\mathcal{I}_2:=\int ( \mu   \partial_2^2  ( \partial_1\eta_1\partial_2 \eta_1) \partial_2^2\mm{curl} \eta-\partial_2^2\mm{curl}_{\mathcal{A}} \eta
 \mm{curl}_{\mathcal{A}}(\bar{\rho}  u))  \mm{d}y -\mathcal{I}_1,\\
&\mathcal{I}_1:=\int   \partial_2^2  \mm{curl}_{\mathcal{A}}\eta (   \partial_2  (\partial_2\eta_1\mm{curl}_{\mathcal{A}}\partial_1 \eta  )
+\partial_2\eta_1 (  \partial_2\mm{curl}_{\mathcal{A}} \partial_1 \eta-\partial_2\eta_1 \mm{curl}_{\mathcal{A}} \partial_1^2 \eta) ) \mm{d}y.
 \end{align*}
\end{lem}
\begin{pf}In view of \eqref{dstist01}, we have
$$\mm{curl}_{\mathcal{A}}\left(gG_{\eta} \mathbf{e}_2\right)=\mm{curl}_{\mathcal{A}}\left(-g\bar{\rho} \mathbf{e}_2\right)
=-g\mathcal{A}_{1j}\partial_j \bar{\rho} = g\bar{\rho}'\partial_1\eta_2 .
$$
Therefore, applying $\mm{curl}_{\mathcal{A}}$ to \eqref{01dsaf16asdfasf}$_2$  and then using the fact
$$\mm{curl}_{\mathcal{A}} \Delta_{\mathcal{A}}u=\Delta_{\mathcal{A}}\mm{curl}_{\mathcal{A}} u , $$
we get
\begin{align}\label{201910072117}
&\partial_t\mm{curl}_{\mathcal{A}}( \bar{\rho} u) -\mu \Delta_{\mathcal{A}} \mm{curl}_{\mathcal{A}}u-\lambda m^2 \partial_1\mm{curl}_{ \mathcal{A}}  \partial_1   \eta\nonumber \\
&= g\bar{\rho}'\partial_1 \eta_2-\lambda m^2\mm{curl}_{\partial_1\mathcal{A}}  \partial_1 \eta
+  \mm{curl}_{\mathcal{A}_t}( \bar{\rho} u) .
\end{align}

Let $\alpha$ be a multi-index. Applying $\partial^{\alpha}$ to \eqref{201910072117} yields
\begin{align}
& \partial_t\partial^{\alpha}\mm{curl}_{\mathcal{A}}( \bar{\rho} u)
-\mu \partial^{\alpha}\Delta_{\mathcal{A}}\mm{curl}_{\mathcal{A}} u- \lambda m^2\partial_1 \partial^{\alpha}\mm{curl}_{\mathcal{A}}  \partial_1  \eta\nonumber \\
& =g \bar{\rho}'\partial^{\alpha}\partial_1\eta_2  - \lambda m^2 \partial^{\alpha}\mm{curl}_{\partial_1\mathcal{A}}  \partial_1  \eta+ \partial^{\alpha}\mm{curl}_{\mathcal{A}_t}( \bar{\rho} u).\label{202005021542}
\end{align}

(1) We take $\alpha=(i,0)$ and $0\leqslant i\leqslant 1$. Multiply \eqref{202005021542} by $ \partial_1^i\mm{curl}_{\mathcal{A}} \eta $ in $L^2$ and then using the integral by parts,
 we have
 \begin{align}
&\frac{\mm{d}}{\mm{d}t}
\int
  \partial_1^i \mm{curl}_{\mathcal{A}} (\bar{\rho} u)\partial_1^i\mm{curl}_{\mathcal{A}} \eta
 \mm{d}y
+{\lambda}m^2\|  \partial_1^i \mm{curl}_{\mathcal{A}} \partial_1 \eta \|^2_{0}  =\sum_{j=1}^4J_{j,i}
,\label{f202005031720}
\end{align}
where
\begin{align}
J_{1,i}:=&\int \left(\partial_1^i
  \mm{curl}_{\mathcal{A}} (\bar{\rho} u )\partial_1^i\mm{curl}_{\mathcal{A}} u  - g  \bar{\rho}'\partial_1^{i}\eta_2
   \partial_1^{1+i} \mm{curl}_{\mathcal{A}} \eta\right)\mm{d}y,\nonumber \\
    J_{2,i}:=&
\int (
  \partial_1^i \mm{curl}_{\mathcal{A}}( \bar{\rho}   u  ) \partial_1^i\mm{curl}_{\mathcal{A}_t} \eta + \partial_1^i \mm{curl}_{\mathcal{A}_t}( \bar{\rho} u) \partial_1^i \mm{curl}_{\mathcal{A}} \eta )\mm{d}y,\nonumber\\
J_{3,i} :=& - \lambda m^2 \int (
  \partial^i\mm{curl}_{\partial_1\mathcal{A}}  \partial_1  \eta \partial^{i}_1 \mm{curl}_{\mathcal{A}}    \eta+ \partial_1^i \mm{curl}_{\mathcal{A}} \partial_1\eta \partial_1^i \mm{curl}_{\partial_1\mathcal{A}}  \eta ) \mm{d}y,
\nonumber\\
J_{4,i}:=&  \mu \int  \partial^{i}_1\Delta_{\mathcal{A}}\mm{curl}_{\mathcal{A}} u  \partial_1^i\mm{curl}_{\mathcal{A}} \eta \mm{d}y. \nonumber
\end{align}

Making use of \eqref{aprpiosesnew}, \eqref{202220202011843}, \eqref{fgestims} and the integral by parts, we have
\begin{align}
J_{1,i}
   \lesssim&\|\partial_1^i
  \mm{curl}_{\mathcal{A}} (\bar{\rho} u )\|_0\|\partial_1^i\mm{curl}_{\mathcal{A}} u \|_0+\|\partial_1^{i}\eta_2
   \|_0\|\partial_1^{1+i} \mm{curl}_{\mathcal{A}} \eta\|_0\nonumber \\
        \lesssim & \|u\|_2^2+ \| \eta_2\|_{i,0} \| \eta \|_{1+i,1} \label{20221011131856},\\
 \label{f202008231718}  J_{2,i} \lesssim &\|\nabla \eta\|_2\|  u\|_{2}^2 .
\end{align}
and
\begin{align}   J_{3,i}\leqslant & c\|\nabla \eta\|_2\|\eta\|_{1+i,1}^2 +
\begin{cases}
0 &\mbox{for }i=0;\\
   \lambda m^2 \int
   \mm{curl}_{\partial_1\mathcal{A}}  \partial_1   \eta \partial^{2}_1 \mm{curl}_{\mathcal{A}}    \eta    \mm{d}y &\mbox{for }i=1\end{cases}
  \nonumber \\
 \leqslant  &c\|\nabla \eta\|_2\|\eta\|_{1+i,1}^2.\label{202022010181634}
\end{align}
Next we turn to the estimate of $J_{4,i}$.

Using the integral by parts, the boundary condition \eqref{2022011091553}  and \eqref{202201152005}, we have
\begin{align}
 J_{4,i} =
\sum_{j=1}^3  J_{4,i,j}    -\frac{\mu }{2}\frac{\mm{d}}{\mm{d}t} \| \nabla_{\mathcal{A}}\partial_1^i \mm{curl}_{\mathcal{A}}\eta \|_0^2 ,  \label{2022202041810}
  \end{align}
 where
 \begin{align*}
 J_{4,i,1} :=&\mu \int \nabla_{\mathcal{A}} \partial_1^i  \mm{curl}_{\mathcal{A}}\eta  \cdot \nabla_{\mathcal{A}}\partial_1^i \mm{curl}_{\mathcal{A}_t} \eta \mm{d}y,\\
 J_{4,i,2}:=&
\mu \int  \nabla_{\mathcal{A}} \partial_1^i \mm{curl}_{\mathcal{A}}\eta \cdot  \nabla_{\mathcal{A}_t}\partial_1^i \mm{curl}_{\mathcal{A}} \eta \mm{d}y,\\
 J_{4,i,3}:=&\begin{cases}
0&\mbox{for }i=0;\\
-\mu \int  \partial_1 ( \mathcal{A}_{kl}  \mathcal{A}_{kn}) \partial_n \mm{curl}_{\mathcal{A}}u\cdot   \partial_l \partial_1\mm{curl}_{\mathcal{A}} \eta \mm{d}y&\mbox{for }i=1.
\end{cases}\end{align*}

Exploiting \eqref{fgestims} and \eqref{fgessfdims}, we easily get
\begin{align}
 \sum_{j=1}^3  J_{4,1,j}
\lesssim  \|\eta\|_{1,2} \|\nabla \eta\|_2\| u\|_{\underline{1},2}.  \nonumber
\end{align}
In addition,
 \begin{align}  J_{4,0,1} =  &\mu \int  \nabla_{\mathcal{A}} \mm{curl}_{\mathcal{A}}\eta \cdot  \nabla_{\mathcal{A}}
  (\partial_1 \eta_1\partial_2 u_1+\partial_1 \eta_2\partial_2u_2 -\partial_2 \eta_1\partial_1u_1-\partial_2 \eta_2\partial_1u_2
)\mm{d}y\nonumber \\
\leqslant & c\|\nabla \eta\|_1\|(\partial_1\eta, \partial_2\eta_2)\|_{2}\|\nabla u\|_1+  \mu \int
  (\nabla_{\mathcal{A}} \mm{curl}_{\mathcal{A}}\eta \cdot  (\nabla_{\mathcal{A}}
  (  \partial_1\partial_2 \eta_1u_1)  \nonumber \\
  &+     \nabla_{\partial_1\mathcal{A}}
  (\partial_2 \eta_1u_1))+ \partial_1 \nabla_{\mathcal{A}} \mm{curl}_{\mathcal{A}}\eta \cdot  \nabla_{\mathcal{A}}
(\partial_2 \eta_1 u_1  ) ) \mm{d}y\nonumber \\
\lesssim & \|\nabla \eta\|_1\|(\partial_1\eta,\partial_2\eta_2)\|_{2}\|u\|_2  \nonumber
\end{align}
and
 \begin{align} J_{4,0,2} =  & \mu \int  ( {\mathcal{A}_{1j}} \partial_j \mm{curl}_{\mathcal{A}}\eta (  \partial_1 \mm{curl}_{\mathcal{A}} \eta\partial_2 u_2- \partial_2 \mm{curl}_{\mathcal{A}} \eta\partial_1 u_2)   \nonumber \\
&+    {\mathcal{A}_{2j}} \partial_j \mm{curl}_{\mathcal{A}}\eta  ( \partial_2 \mm{curl}_{\mathcal{A}} \eta\partial_1 u_1-\partial_1 \mm{curl}_{\mathcal{A}} \eta\partial_2 u_1 ) ) \mm{d}y\nonumber \\
 \lesssim& \|\nabla \eta\|_1\|\eta\|_{1,2}\| u\|_2.
  \nonumber
\end{align}
Thanks to the above three estimates and \eqref{omessetsim122n}, we can infer from \eqref{2022202041810} that
\begin{align}
J_{4,i}\leqslant c \|\eta\|_{1,2} \|\nabla \eta\|_2\| u\|_{\underline{1},2}-\frac{\mu }{2}\frac{\mm{d}}{\mm{d}t} \|\nabla_{\mathcal{A}}\partial_1^i \mm{curl}_{\mathcal{A}}\eta \|_0^2.
\label{2022202081350}
\end{align}

Inserting the estimates \eqref{20221011131856}--\eqref{2022202081350} into \eqref{f202005031720}, and then using  the interpolation inequality \eqref{201807291850}, we immediately obtain
   \eqref{202005021600}.

(2) Now we turn to the derivation of \eqref{202201211416}.
Multiplying \eqref{201910072117}   by  $- \partial_2^2 \mm{curl}_{\mathcal{A}} \eta $, and then using the integral by parts, and the boundary conditions \eqref{202220118000}, \eqref{2022011091553} and $\partial_1\mm{curl}_{\mathcal{A}}u|_{\partial\Omega}=0$,
 we have
 \begin{align}
 &\frac{\mm{d}}{\mm{d}t}
\left(
 \frac{\mu }{2}\| \nabla \partial_2\mm{curl}_{\mathcal{A}}\eta\|^2-\int \partial_2^2\mm{curl}_{\mathcal{A}} \eta \mm{curl}_{\mathcal{A}}(\bar{\rho}  u)
\mm{d}y\right)
+{\lambda}m^2\|  \partial_2\mm{curl}_{\mathcal{A}} \partial_1 \eta \|^2_{0}  =\sum_{ j=5}^{9}J_{j}
,\label{f202005720}
\end{align}
where
\begin{align}
& J_5  := - \int (g \bar{\rho}' \partial_2\eta_2
 \partial_1  \partial_2 \mm{curl}_{\mathcal{A}} \eta +
\mm{curl}_{\mathcal{A}} (\bar{\rho} u  )\partial_2^2 \mm{curl}_{\mathcal{A}}  u) \mm{d}y,\nonumber \\
  &   J_6 :=-
\int   (\mm{curl}_{\mathcal{A}_t} (\bar{\rho} u   )  \partial_2^2\mm{curl}_{\mathcal{A} } \eta  +  \mm{curl}_{\mathcal{A}}( \bar{\rho} u)  \partial_2^2 \mm{curl}_{\mathcal{A}_t} \eta)   \mm{d}y,\nonumber\\
&  J_7 :=    \lambda m^2 \int ( \mm{curl}_{\partial_1\mathcal{A}}  \partial_1  \eta \partial_2^2\mm{curl}_{ \mathcal{A}}  \eta - \partial_2\mm{curl}_{\mathcal{A}} \partial_1\eta\partial_2 \mm{curl}_{\partial_1 \mathcal{A}}  \eta
   )\mm{d}y,
\nonumber\\
& J_{8}:= \mu \int  \nabla \partial_2   \mm{curl}_{\mathcal{A}}\eta \cdot \nabla \partial_2\mm{curl}_{\mathcal{A}_t}\eta   \mm{d}y\mbox{ and }J_{9}:= \mu \int   ( \Delta-\Delta_{\mathcal{A}}) \mm{curl}_{\mathcal{A}} u   \partial_2^2  \mm{curl}_{\mathcal{A}}\eta  \mm{d}y.
\nonumber
\end{align}

It is easy to see that
\begin{align}\label{f2031718}
 \begin{cases}
 J_5   \lesssim  \|\partial_2 \eta_2\|_{0}\|\eta\|_{1,2}+\|u\|_1\|\nabla u\|_2,\\
   J_6 \lesssim    \|\nabla \eta\|_2\|\nabla u\|_2\|u\|_1\lesssim \sqrt{\mathcal{E}}\mathcal{D},\\
    J_7 \lesssim  \|\nabla \eta\|_2\|  \eta\|_{1,2}^2\lesssim \sqrt{\mathcal{E}}\mathcal{D}.
    \end{cases}
\end{align}

Obviously
\begin{align} J_{8}=& -  \mu \frac{\mm{d}}{\mm{d}t}\int  \partial_2^2  ( \partial_1\eta_1\partial_2 \eta_1) \partial_2^2\mm{curl} \eta  \mm{d}y+\tilde{J}_8,
\label{2022201191657}
\end{align}
where we have defined that
\begin{align*}
\tilde{J}_{8}:=  &\mu \int ( \partial_2 \partial_1   \mm{curl}_{ {\mathcal{A}}}\eta  \partial_2\partial_1 \mm{curl}_{\mathcal{A} }\eta+ \partial_2^2  \mm{curl}_{\tilde{\mathcal{A}}}\eta  \partial_2^2\mm{curl}_{\mathcal{A}_t}\eta ) \mm{d}y  \\
& + \mu \int  \partial_2^2\mm{curl} \eta\partial_2^2  (\partial_1 \eta_1 \partial_2 u_1+\partial_1 \eta_2\partial_2u_2
-\partial_2 \eta_2\partial_1u_2
)  \mm{d}y   \\
& + \mu \int \partial_2^2\mm{curl} \eta \partial_2^2  ( \partial_1\eta_1\partial_2 u_1) \mm{d}y+ \mu \int  \partial_2^2  ( \partial_1\eta_1\partial_2 \eta_1) \partial_2^2\mm{curl} u  \mm{d}y.
\end{align*}
Exploiting \eqref{201907291432} and \eqref{fgestims}, we have
\begin{align} &\tilde{J}_8\lesssim
(\|\eta\|_{1,2}+\|\partial_2\eta_2\|_2)\|\nabla \eta\|_{2}\|\nabla u\|_2\lesssim \sqrt{\mathcal{E}}\mathcal{D}.  \label{2022201191653}
\end{align}

The integral $J_9$ can be rewritten as follows:
\begin{align} J_9:= &\mu \int  \partial_2^2  \mm{curl}_{\mathcal{A}}\eta  ( (1+\partial_2\eta_2)\partial_1((\partial_1\eta_2\partial_2- \partial_2\eta_2 \partial_1 )\mm{curl}_{\mathcal{A}} u  )
 \nonumber \\ &+(1+\partial_1\eta_1)\partial_2((\partial_2\eta_1\partial_1 -\partial_1\eta_1 \partial_2 )\mm{curl}_{\mathcal{A}} u  ) +
 \partial_1\eta_2\partial_2(((1+\partial_2\eta_2)\partial_1\nonumber \\
 &-\partial_1\eta_2 \partial_2 ) \mm{curl}_{\mathcal{A}} u  ) +\partial_2\eta_1\partial_1(((1+\partial_1\eta_1)\partial_2\nonumber \\
 &-\partial_2\eta_1\partial_1)\mm{curl}_{\mathcal{A}} u) - (\partial_1\eta_1 \partial_2^2  +\partial_2\eta_2 \partial_1^2 ) \mm{curl}_{\mathcal{A}} u    ) \mm{d}y  =  {J}_{9,1} + {J}_{9,2} ,
 \label{2022201191423}
\end{align}
where
\begin{align*}
 {J}_{9,1}:=  & \mu \int \partial_2^2  \mm{curl}_{\mathcal{A}}\eta     ( (1+\partial_2\eta_2)( \partial_1(\partial_1\eta_2\partial_2\mm{curl}_{\mathcal{A}} u  )-\partial_2\partial_1\eta_2 \partial_1\mm{curl}_{\mathcal{A}} u  )
 \nonumber \\
 &-  \partial_2\eta_2   (1+\partial_2\eta_2)  \partial_1^2 \mm{curl}_{\mathcal{A}} u  +  \partial_2( \partial_2\eta_1\mm{curl}_{\partial_1\mathcal{A}} u  ) \nonumber \\
 &
+ \partial_1\eta_1 \partial_2( \partial_2\eta_1\partial_1 \mm{curl}_{\mathcal{A}} u  )-(1+\partial_1\eta_1)\partial_2( \partial_1\eta_1 \partial_2  \mm{curl}_{\mathcal{A}} u  )
 \nonumber \\
 &+
 \partial_1\eta_2\partial_2(((1+\partial_2\eta_2)\partial_1-\partial_1\eta_2 \partial_2 ) \mm{curl}_{\mathcal{A}} u  ) +
 \partial_2\eta_1\partial_1^2\eta_1 \partial_2\mm{curl}_{ \mathcal{A}}   u \nonumber \\
& +  \partial_1\eta_1\partial_2\eta_1\partial_2\partial_1 \mm{curl}_{\mathcal{A}} u+\partial_2\eta_1\partial_2\mm{curl}_{\partial_1 \mathcal{A}} u-\partial_2\eta_1 \partial_2\partial_1 \eta_1\partial_1 \mm{curl}_{ \mathcal{A}} u
\nonumber \\
 &
  -(\partial_2\eta_1)^2 (  \mm{curl}_{\partial_1 ^2\mathcal{A}}  u + \mm{curl}_{\partial_1  \mathcal{A}} \partial_1 u) -(\partial_1\eta_1 \partial_2^2  +\partial_2\eta_2 \partial_1^2 ) \mm{curl}_{\mathcal{A}} u   ) \mm{d}y,\\
 {J}_{9,2}:=  & \mu \int  \partial_2^2  \mm{curl}_{\mathcal{A}}\eta (\partial_2  (\partial_2\eta_1\mm{curl}_{\mathcal{A}}\partial_1 u  )- \partial_2\eta_2(1+ \partial_2\eta_2 ) \mm{curl}_{\mathcal{A}} \partial_1^2  u
 \nonumber \\
 & +\partial_2\eta_1 ( \partial_2\mm{curl}_{\mathcal{A}} \partial_1 u -\partial_2\eta_1 \mm{curl}_{\mathcal{A}} \partial_1^2 u)  ) \mm{d}y .\end{align*}
 We further rewrite the integral $ {J}_{9,2}$ as follows:
 \begin{align}
 {J}_{9,2}:=  &  \tilde{J}_{9,2}-\frac{\mm{d}}{\mm{d}t}\mathcal{I}_1 , \label{2022201191633}
 \end{align}
 where
  \begin{align*}
 \tilde{J}_{9,2}:=  \mu  \int &((\partial_2\eta_2 (1 +\partial_2\eta_2 )
\mm{curl}_{\mathcal{A}} \partial_1^2 \eta-    \partial_2  (\partial_2\eta_1\mm{curl}_{\mathcal{A}}\partial_1 \eta  )-  \partial_2\eta_1 (  \partial_2\mm{curl}_{\mathcal{A}} \partial_1 \eta
 \nonumber \\
 &-\partial_2\eta_1 \mm{curl}_{\mathcal{A}} \partial_1^2 \eta))  \partial_t\partial_2^2  \mm{curl}_{\mathcal{A}}\eta -\partial_2^2  \mm{curl}_{\mathcal{A}}\eta(  \partial_2  (\mm{curl}_{\mathcal{A}}\partial_1  \eta  \partial_2u_1 \nonumber \\
 &+\partial_2\eta_1\mm{curl}_{\mathcal{A}_t}\partial_1 \eta    )  +( \partial_2\mm{curl}_{\mathcal{A}} \partial_1 \eta -\partial_2\eta_1 \mm{curl}_{\mathcal{A}} \partial_1^2\eta )\partial_2u_1
 \nonumber \\
    &+\partial_2\eta_1 ( \partial_2\mm{curl}_{\mathcal{A}_t} \partial_1 \eta-\partial_2u_1 \mm{curl}_{\mathcal{A}} \partial_1^2 \eta-\partial_2\eta_1 \mm{curl}_{\mathcal{A}_t} \partial_1^2 \eta) ) \mm{d}y.
 \end{align*}

 It is easy to estimate that
 \begin{align}
{J}_{9,1} +  \tilde{J}_{9,2}\lesssim (\|\eta\|_{1,2}+\|\partial_2\eta_2\|_2)\|\nabla \eta\|_2\|\nabla u\|_2. \nonumber
  \end{align}
Inserting \eqref{2022201191633} into \eqref{2022201191423} and then using  the above  estimate, we get
\begin{align} J_9\leqslant  c  \sqrt{\mathcal{E}}\mathcal{D} - \frac{\mm{d}}{\mm{d}t}\mathcal{I}_1.
 \label{291423}
\end{align}

Finally, putting \eqref{2022201191657} and  \eqref{291423} into \eqref{f202005720}, and then using the estimates \eqref{omessetsim122n}, \eqref{f2031718}  and \eqref{2022201191653}, we arrive at \eqref{202201211416}. This completes the proof.
\hfill$\Box$
\end{pf}

\subsection{Stabilizing estimates}

Now  we further establish the stabilizing estimates for $E(\partial_1^{i}\eta)$ and $E(u)$ appearing in Lemmas \ref{lem:082sdaf41545}--\ref{2019100216355nnn}.
\begin{lem}\label{lem:08250749}
We have
\begin{align}
&\label{202008250745}
\|\eta\|_{{1+i},0}^2
\lesssim-E(\partial_1^{i}\eta)+\begin{cases}
0 &\mbox{for }i=0,\ 1;  \\
\|\eta\|_{2,1}^4&\mbox{for   }i= 2,
\end{cases}
\\
&
\label{20250745}
\| u\|_{1, 0}^2
\lesssim-E(u)+\|\nabla \eta\|_2\| u\|_1^2.
\end{align}
\end{lem}
\begin{pf}
By virtue of the definition of $m_{\mm{C}}$, it is easy to see that
\begin{align*}
-g\int \bar{\rho}'w_2^2\mm{d}y\geqslant-\lambda \| m_{\mm{C}}\partial_1w\|_0^2\mbox{ for any }w\in H_{\sigma}^1,
\end{align*}
 which, together with the stability condition $|m|>m_{\mm{C}}$, implies
\begin{align}\label{202008250814}
\| w\|_{1,0}^2\lesssim\lambda (m^2-m_{\mm{C}}^2)\|\partial_1w\|_0^2\lesssim-E(w)\;\;\mbox{ for any } w\in H_{\sigma}^1.
\end{align}

For applying the above estimate to $\eta$,  we shall modify $\eta$.
 Let $0\leqslant i\leqslant 2$ be given. By Lemma \ref{pro:12x21}, there exists a Bogovskii's operator $\mathcal{B} : \partial_1^{i}\mm{div}\eta\in \underline{L}^2 \to H^1_0$ such that
\begin{align}
 \mathrm{div}\mathcal{B}(\partial_1^{i}\mm{div} \eta)=\partial_1^{i}\mm{div}\eta\mbox{ and }\|{\mathcal{B}}(\partial_1^{i}\mm{div}\eta)\|_1\lesssim \|\partial_1^{i}\mm{div}\eta\|_0.
 \label{2022202191518}
\end{align}

Now we use $\partial_1^i \eta- \mathcal{B}(\partial_1^{i}\mm{div} \eta)$ to rewrite $ {E}(\partial_1^{i}\eta)$ as follows.
\begin{align}\label{201912191626}
 {E}(\partial_1^{i}\eta) = {E}( \partial_1^{i} \eta-\mathcal{B}(\partial_1^{i}\mm{div} \eta)) +{E}(\mathcal{B}(\partial_1^{i}\mm{div} \eta) ) -K_{1+i},
\end{align}
where
\begin{align*}
K_{1+i}: = 2\lambda m^2\int \partial_1^{1+i}\eta\cdot \partial_1 \mathcal{B}(\partial_1^{i}\mm{div} \eta)\mm{d}y
-2g\int \bar{\rho}'\partial_1^{i}\eta_2\mathcal{B}(\partial_1^{i}\mm{div} \eta)\cdot \mathbf{e}_2\mm{d}y.
\end{align*}
Recalling $\partial_1^{i} \eta- \mathcal{B}(\partial_1^{i}\mm{div} \eta) \in H_{\sigma}^1$, we use \eqref{202008250814} to get
\begin{align}\nonumber
\| \partial_1^{i}\eta- \mathcal{B}(\partial_1^{i}\mm{div} \eta)\|_{ 1,0}^2 \lesssim-E(\partial_1^{i} \eta-\mathcal{B}(\partial_1^{i}\mm{div} \eta)),
\end{align}
which, together with \eqref{201912191626} and Young's inequality, gives
\begin{align}\label{202008250814nn}
\| \eta\|_{i+1,0}^2 \lesssim{E}(\mathcal{B}(\partial_1^{i}\mm{div} \eta) ) -E(\partial_1^{i}\eta) -K_{1+i}+\| \mathcal{B}(\partial_1^{i}\mm{div} \eta)\|_{ 1,0}^2.
\end{align}

Exploiting \eqref{2022202191518},  we find that
\begin{align}
  & {E}(\mathcal{B}(\partial_1^{i}\mm{div} \eta) )
-K_{1+i}
+\|\mathcal{B}(\partial_1^{i}\mm{div} \eta)\|_{1, 0}^2
 \nonumber \\
 &\lesssim \|\eta\|_{1+i,0}\|\mathcal{B}(\partial_1^{i}\mm{div} \eta)\|_{ 1,0}+ \|\eta_2\|_{i,0} \|\mathcal{B}(\partial_1^{i}\mm{div} \eta)\|_{ 0}+\|\mathcal{B}(\partial_1^{i}\mm{div} \eta)\|_{\underline{1}, 0}^2\nonumber \\
 & \lesssim
(\|\eta\|_{1+i,0}+ \|\eta_2\|_{i,0} )\| \mm{div} \eta \|_{ i,0}+\|  \mm{div} \eta \|_{i, 0}^2.\nonumber
\end{align}
Finally, putting the above estimate into \eqref{202008250814nn}, and then using \eqref{improtian1}, \eqref{omessetsim122n}, \eqref{202012241002} and Young's inequality, we get \eqref{202008250745}.

Similarly we can easily follow the argument of \eqref{202008250745} with $i=0$ to establish \eqref{20250745} by further using the relation
\begin{align}
\label{2022202121921}
\mm{div}u=-\mm{div}_{\tilde{\mathcal{A}}}u
\end{align} and the estimate
$$  \|\mm{div}_{\tilde{\mathcal{A}}}u\|_0\lesssim \|\nabla \eta\|_2\|u\|_1.$$
This completes the proof.
\hfill $\Box$
\end{pf}

\subsection{Stokes estimates}\label{subsec:Vor}

In this  section, we use the regularity theory of the Stokes problem  (with Navier boundary condition)  to
derive the estimates of normal estimates of $u$ and the equivalence estimates for $\mathcal{E}$.
\begin{lem}
For sufficiently small $\delta$, we have
\begin{align}
\label{omessetsim}
&\|  u\|_{\underline{i},2+j} +\|    q\|_{\underline{i},1+j} \lesssim \|    ( \partial_1^{ 2}\eta, u_t)\|_{\underline{i},j
}\mbox{ for }0\leqslant i+j\leqslant  1,  \\
&\label{202012252005}\mathcal{E} \mbox{ and } \| (\nabla \eta,u)\|_2^2\mbox{ are equivalent to each other} ,
\end{align}
where the  equivalent coefficients in \eqref{202012252005}   are independent of $\delta$.
\end{lem}
\begin{pf}
To begin with, we derive \eqref{omessetsim}. By \eqref{202201152005}, \eqref{01dsaf16asdfasf03n}$_2$ and \eqref{2022202121921}, we have the following Stokes problem
\begin{equation}
\label{Stokesequson}
\begin{cases}
\partial_1^i ( \nabla q-\mu\Delta  u)
=   \partial_1^i(\lambda m^2 \partial_1^{ 2}\eta+g\bar{\rho}'\eta_2\mathbf{e}_2- \bar{\rho}u_t
+ \mathcal{G}\mathbf{e}_2+\mathcal{N}^\mu-  \nabla_{\tilde{\mathcal{A}}} q ) ,\\
 \partial_1^i \div u=-  \partial_1^i\div({\tilde{\mathcal{A}}}^{\mm{T}} u) ,\\
 \partial_1^i (u_2,
 \partial_2   u_1) |_{\partial\Omega} =0,
\end{cases}
\end{equation}
where $i=0$, $1$.
By Remark \ref{pro:12x21}, we can apply  the regularity estimate  \eqref{202201122130} to above Stokes problem  to get
\begin{equation}   \label{ometsim}
  \|  u \|_{i,2+j} +\|   q\|_{i,1+j} \lesssim \|    ( \partial_1^{ 2}\eta,\eta_2,u_t)\|_{i,j} +K_4, \end{equation}
where $0\leqslant i+j\leqslant 1$ and
$$K_4:=\|  ( \mathcal{G}, \mathcal{N}^\mu,\nabla_{\tilde{\mathcal{A}}} q,{\tilde{\mathcal{A}}}^{\mm{T}} u)  \|_{i,j}  +\| \div({\tilde{\mathcal{A}}}^{\mm{T}} u)\|_{i,1+j}.$$

It is easy to estimate that
\begin{align}
 & \|    \mathcal{G} \|_{i,j}\lesssim \|    \mathcal{G} \|_{1}\lesssim \| \eta_2\|_1\lesssim \| \eta\|_{2,0},\label{2022202081423} \\
&  \|  \nabla_{\tilde{\mathcal{A}}} q  \|_{i,j}\lesssim \|\nabla \eta\|_2\|\nabla q\|_{\underline{i},j},\nonumber\\
& \|   {\tilde{\mathcal{A}}}^{\mm{T}} u  \|_{i,j}\lesssim  \|\nabla \eta\|_2   \|u\|_{\underline{i},j} ,\nonumber\\
  &\|    \mathcal{N}^\mu\|_{i,j}+\|\div({\tilde{\mathcal{A}}}^{\mm{T}} u) \|_{i,1+j}\nonumber \\
  &=\|    \mathcal{N}^\mu \|_{i,j}+\|\div_{\tilde{\mathcal{A}}}u  \|_{i,1+j}\lesssim  \|\nabla \eta\|_{2}\|\nabla u\|_{\underline{i},1+j},\nonumber
 \end{align}
 where we shall use \eqref{omessetsim122n} and \eqref{2022201201821} in \eqref{2022202081423}.
Putting the above estimates into \eqref{ometsim} yields the desired estimate \eqref{omessetsim}.

Next we turn to the derivation of  \eqref{202012252005}. By \eqref{202211}, the Stokes estimate \eqref{omessetsim} with $i=j=0$ and the definition  of $\mathcal{E}$, we easily see that
\begin{align}
& \|(\nabla \eta,u)\|_2^2 \lesssim \mathcal{E}\lesssim \|(\nabla \eta,u)\|_2^2+\|u_t\|_0^2.
\end{align}
 Obviously, to complete the proof, it suffices to prove that
 \begin{align}
\|u_t\|_0  \lesssim   \|   \eta\|_{2,0}+\| u\|_2 . \label{20222012120226}
\end{align}

 To  this end, we multiply \eqref{01dsaf16asdfasf}$_2$ by $u_t$ in $L^2$ to obtain
\begin{align}
\|\sqrt{ \bar{\rho} }u_{t}\|_{0}^2   = & \int  ((\lambda m^2 \partial_1^2\eta+(\mathcal{G}+g  \bar{\rho}'\eta_2)\mathbf{e}_2+\mu\Delta_{\mathcal{A}}  u)\cdot u_t+ \nabla q \cdot({\mathcal{A}}_t^{\mm{T}}u ))\mathrm{d}y =:K_5. \label{20222012120232}
\end{align}
Making use of  \eqref{2022011130957}, \eqref{omessetsim122n}, \eqref{omessetsim} and  \eqref{202012241002}, we obtain
\begin{align*}
 K_5   \lesssim & (\|\eta_2\|_0+\|\eta\|_{2,0}+ \|u\|_2)\|u_t \|_0 +\|\nabla u\|_0\| u\|_2\|\nabla q\|_0 \\
 \lesssim & ( \|\eta\|_{2,0} +\|u\|_2+\|\nabla u\|_0\| u\|_2 )\|u_t \|_0 +\|\eta\|_{2,0} \| u\|_2^2 . \end{align*}
    Putting the above estimate into \eqref{20222012120232}, and then using
 \eqref{aprpiosesnew} and Young's inequality, we arrive at \eqref{20222012120226}. This completes the proof.
  \hfill $\Box$
\end{pf}

\subsection{A priori stability estimates}\label{202220201221244}

Now we are in a position to establish the \emph{a priori} stability estimate  \eqref{1.200}  under the assumptions \eqref{apresnew}--\eqref{aprpiosesnew}.

We can derive from  Lemmas \ref{lem:082sdaf41545}--\ref{lem:08241445} and the estimate \eqref{20222saf201121235} that there exist two different constants $ {c}$, such that for any sufficiently large $\chi\geqslant 1 $ (depending on $\mu$ and $\bar{\rho}$), and for any sufficiently small
$\delta\in (0,1]$ (independent of $\chi$),
  the following  tangential energy inequality holds:
  \begin{align}&\label{202008250856n0}
\frac{\mm{d}}{\mm{d}t} {\mathcal{E}}_{\mm{tan}}
+c  {\mathcal{D}}_{\mm{tan}}  \lesssim \chi\sqrt{\mathcal{E}}\mathcal{D} ,
\end{align}
where
$$
\begin{aligned}
&\begin{aligned} {\mathcal{E}}_{\mm{tan}} :=& \sum_{0\leqslant i \leqslant 2}\left(\int\bar{\rho}\partial_1^{i}\eta\cdot\partial_1^{i} u\mm{d}y
+\frac{\mu}{2}\| \nabla \partial_1^{i} \eta\|_{0}^2\right) \\
&+\chi   \left(\|\sqrt{ \bar{\rho} }  u\|^2_{\underline{2},0}
-\sum_{0\leqslant i \leqslant 2}E(\partial_1^{i} \eta)\right)+\|\sqrt{\bar{\rho}}\psi\|_0^2- E(  u) ,\end{aligned}\\
  &  {\mathcal{D}}_{\mm{tan}} :=\chi\|  u\|_{\underline{2},1}^2 -\sum_{0\leqslant i\leqslant 2 }E(\partial_1^{i}\eta) + \|   u_t\|_1^2  .
\end{aligned}
$$

Making use of Lemma \ref{lem:08250749}, \eqref{202211}, \eqref{omessetsim}  and Young's inequality, we get, for any sufficiently large $\chi$ and for any sufficiently small $\delta\in (0,\chi^{-1})$,
\begin{align}
&\| \eta\|_{\underline{2},1}^2 +\|u\|_2^2+ \|u_t \|_0^2+\|q\|_1^2+ \chi \|( \partial_1 \eta, u)\|^2_{\underline{2},0}  \lesssim  {\mathcal{E}}_{\mm{tan}} \lesssim \chi \mathcal{E}.
\label{20221057}
\end{align}
and
\begin{align}
   \|\partial_1\eta\|_{\underline{2},0}^2+ \chi\|  u\|_{\underline{2},1}^2  +\|   u_t\|_1^2 \lesssim  {\mathcal{D}}_{\mm{tan}}+\|\eta\|_{2,1}^4 . \label{2022202012047}
\end{align}

Obviously,
\begin{align}
&\|  \mm{curl}  \eta \|^2_{1,0}
\lesssim\|  \mm{curl}_{\mathcal{A}} \partial_1\eta \|^2_{0}+\|    \eta \|^2_{1,1} \|    \nabla \eta \|^2_2 ,\nonumber \\
&\|  \mm{curl}  \eta \|^2_{2,0}
\lesssim\| \partial_1 \mm{curl}_{\mathcal{A}} \partial_1\eta \|^2_{0}+ \|\eta\|_{2,1}^2\|\nabla \eta\|_{2}^2, \label{2022202181239}\\
& \|  \partial_2\mm{curl}  \eta \|^2_{1,0}
\lesssim\|\partial_2\mm{curl}_{\mathcal{A}} \partial_1 \eta \|^2_{0}+ \|\eta\|_{1,2}^2\|\nabla \eta\|_{2}^2.\nonumber
\end{align}
Thanks to   the above three estimates,
 we can further derive
the total energy inequality from Lemma  \ref{2055nnn} and \eqref{202008250856n0}:
  \begin{align}
\frac{\mm{d}}{\mm{d}t} \tilde{\mathcal{E}} +c \tilde{\mathcal{D}} \lesssim
 \| \eta\|_{ {1},0}  \|\eta\|_{1,2}   +\| u\|_1\| u\|_3 + \chi^2 \sqrt{\mathcal{E}}\mathcal{D} ,\label{202008250856}
\end{align}
where
\begin{align*}
&
\tilde{\mathcal{E}}:=
{\mathcal{I}_3 + \frac{\mu}{2} \| ( \nabla_{\mathcal{A}} \mm{curl}_{\mathcal{A}} \eta, \nabla_{\mathcal{A}}\partial_1 \mm{curl}_{\mathcal{A}} \eta,\nabla  \partial_2 \mm{curl}_{\mathcal{A}}\eta)\|^2_0+\chi  {\mathcal{E}}_{\mm{tan}} } ,
   \\
& \tilde{\mathcal{D}} := \|  \mm{curl}\eta \|^2_{1,1} +  \chi  {\mathcal{D}}_{\mm{tan}}  ,
 \\
&  {\mathcal{I}_3} := \mathcal{I}_2
+ \int(   \mm{curl}_{\mathcal{A}} \eta
 \mm{curl}_{\mathcal{A}}(\bar{\rho} u ) -   \partial_1^2\mm{curl}_{\mathcal{A}} \eta  \mm{curl}_{\mathcal{A}}(\bar{\rho}  u)  )\mm{d}y   .
\end{align*}
Moreover, by \eqref{omessetsim122n}, \eqref{2022202011749} and \eqref{2022202012047}, it holds that, for any sufficiently small $\delta\in (0,\chi^{-1})$,
\begin{align}
 {\mathcal{D}} \lesssim &\|    \eta\|_{1,2}^2+ \|\eta_2\|_3^2+\|  u\|_{3}^2+\|q\|_2^2+  \chi(\|\partial_1\eta\|_{\underline{2},0}^2\nonumber \\
 &+ \chi\|  u\|_{\underline{2},1}^2+\|   u_t\|_1^2) \lesssim  \tilde{\mathcal{D}} . \label{20x012047}
\end{align}

Noting that
 \begin{align}
 \|   \mm{curl}  \eta\|^2_2
 \lesssim&\|\nabla \eta\|_0^2 + \|  \nabla \mm{curl}  \eta\|^2_{1}\nonumber \\
 \lesssim &\|(\nabla \eta, \nabla_{\mathcal{A}} \mm{curl}_{\mathcal{A}} \eta,
  \nabla_{\mathcal{A}}\partial_1 \mm{curl}_{\mathcal{A}}\eta,
 \nabla \partial_2 \mm{curl}_{\mathcal{A}}\eta, \nabla_{\tilde{\mathcal{A}}}  \mm{curl}_{\mathcal{A}} \eta, \nabla  \mm{curl}_{\tilde{\mathcal{A}} } \eta  ,\nonumber \\
 & \nabla_{\tilde{\mathcal{A}}}\partial_1\mm{curl}_{ {\mathcal{A}}}\eta,   \nabla \partial_1\mm{curl}_{\tilde{\mathcal{A}}}\eta ,\nabla \partial_2\mm{curl}_{\tilde{\mathcal{A}}}\eta  )\|^2_0  \nonumber \\
\lesssim& \|(\nabla \eta,\nabla_{\mathcal{A}} \mm{curl}_{\mathcal{A}} \eta, \nabla_{\mathcal{A}} \partial_1 \mm{curl}_{\mathcal{A}} \eta , \nabla \partial_2 \mm{curl}_{\mathcal{A}}\eta)\|^2_0 + \| \nabla \eta\|^4_2  ,
\label{202202172007}
\end{align}
thus, we get from \eqref{2022202011749} and the above  estimate  that
\begin{align}
 \|  \eta \|_3^2 \lesssim \|(\nabla \eta,\nabla_{\mathcal{A}} \mm{curl}_{\mathcal{A}} \eta, \nabla_{\mathcal{A}} \partial_1 \mm{curl}_{\mathcal{A}} \eta , \nabla \partial_2 \mm{curl}_{\mathcal{A}}\eta)\|^2_0 . \nonumber
\end{align}
 In addition, it is easy to check that
\begin{align}
&| {\mathcal{I}_3}| \lesssim
\| \nabla \eta\|_2\|u\|_1+\|\nabla \eta\|_2^3  . \nonumber
\end{align}
Exploiting \eqref{202012252005}, \eqref{20221057}  and the above two estimates, we obtain that, for any sufficiently large $\chi $ and for any sufficiently small $\delta$,
\begin{align}
&\mathcal{E}\lesssim \|\eta\|_3^2+ \chi  ( \|u\|_2^2+ \|u_t \|_0^2+\|q\|_1^2 )
\lesssim   \tilde{\mathcal{E}}
\lesssim \chi^2\mathcal{E}\lesssim \chi^2\|(\nabla \eta,u)\|_2^2.
\label{2027}
\end{align}

Using  Young's inequality, \eqref{20x012047} and the last inequality  \eqref{2027}, we further derive from \eqref{202008250856} with some sufficiently large $\chi $ that
\begin{align}\label{20191005asdf0940}
\frac{\mm{d}}{\mm{d}t}\tilde{\mathcal{E}}+c  \mathcal{D} \leqslant 0\mbox{  for any sufficiently small }\delta,
\end{align}
where $\tilde{\mathcal{E}}$ satisfies \eqref{2027}. Integrating the above inequality over $(0,t)$, and then using  \eqref{2027}, we arrive at, for some $c_1\geqslant 1$,
\begin{align}
\label{estemalas}
\mathcal{E}(t)+\int_0^t\mathcal{D}(\tau)\mm{d}\tau\leqslant c_1 \|(\nabla \eta^0,u^0)\|_2^2.
\end{align}

\subsection{Decay-in-time estimates}\label{202244}

This secretion is devoted to the derivation of decay-in-time estimate \eqref{1.200n0} under the additional  \emph{a priori} assumption \eqref{202220201915}.

Exploiting the integral by parts, we can derive  from Lemmas \ref{lem:082sdaf41545} and \ref{lem:08241445} that, for $i=1$,  $2$,
\begin{align}
&
\frac{\mm{d}}{\mm{d}t}\left(\langle t\rangle^i\left( \int\bar{\rho}\partial_1^{i}\eta\cdot\partial_1^{i} u\mm{d}y
+\frac{\mu}{2}\| \nabla \partial_1^{i} \eta\|_{0}^2 + \gamma  \left(\|\sqrt{ \bar{\rho} }  u\|^2_{i,0}
-E(\partial_1^i \eta) \right)\right)\right)
\nonumber \\
& +c \langle t\rangle^i(\gamma \|   u\|_{i,1}^2 -  E(\partial_1^{i}\eta) )
 \lesssim
 \langle t\rangle^{i-1}  (\| ( \nabla \partial_1^{i} \eta,\partial_1^{i-1} u  ) \|_{0}^2+ \gamma (\|   u\|^2_{i,0} -E(\partial_1^{i} \eta)))+ K_6,
\label{2020046}
\end{align}
where $\gamma \geqslant 1$ is a sufficiently large  constant (may depending on $\mu$, $g$,   $\lambda$, $m$, $\bar{\rho}$ and $\Omega$) and
\begin{align*}  K_6:=&  \langle t\rangle^2(\|\nabla \eta\|_{2}\|\eta\|_{3,0} ( \|\eta_2\|_{0}+ \|u\|_{\underline{1},2}+\|q\|_{\underline{1},1}
 ) +\|\eta\|_{2,1}^2 \| q\|_{\underline{1},1 } )\nonumber \\
&+ \gamma \langle t\rangle^3( \|\eta\|_{2,1} ( \|\eta_2\|_{0}  + \|  u\|_{\underline{1},2}  ) \| u\|_{\underline{2},1} +  \|\eta\|_{2,1} \| u\|_{\underline{1},2} \| q\|_{\underline{1},1 } )   .   \end{align*}
Making use of  \eqref{202008250745}, \eqref{2022202181239} and \eqref{202012241002}, we further derive from \eqref{202008241448} with $i=2$,
\eqref{202005021600} with $\alpha_1=1$ and \eqref{2020046} that
\begin{align}
& \frac{\mm{d}}{\mm{d}t}\left(\mathcal{E}_{\mm{D}}+\gamma^2 \langle t\rangle
 \int\partial_1\mm{curl}_{\mathcal{A}}   \eta
 \partial_1\mm{curl}_{\mathcal{A}} \left(\bar{\rho}u\right)
 \mm{d}y\right)  +c \mathcal{D}_{\mm{D}} \nonumber\\
&\lesssim  \gamma   (\gamma ( \|    \eta \|_{1,2}^2+ \|u\|_2^2) + \langle t\rangle( \| ( \nabla \partial_1^2 \eta,\partial_1 u  ) \|_{0}^2 +     \gamma (\|   u\|^2_{2,0} +\|\eta\|_{3,0}^2  \nonumber\\
&\quad
+ \|\eta \|_{1,0}\|\eta\|_{2,1} +\|  u\|_1\|u\|_3)  ))+ \langle t\rangle^2 (\|   u\|^2_{2,0} +\|\eta\|_{3,0}^2  ) +K_7,
\label{2020046xx}
\end{align}
where
\begin{align*}
&\mathcal{E}_{\mm{D}}:= \gamma \sum_{i=1}^2 \left(\langle t\rangle^i\left( \int\bar{\rho}\partial_1^{i}\eta\cdot\partial_1^{i} u\mm{d}y
+\frac{\mu}{2}\| \nabla \partial_1^{i} \eta\|_{0}^2 + \gamma  \left(\|\sqrt{ \bar{\rho} }  u\|^2_{i,0}
-E(\partial_1^i \eta) \right)\right)\right) \nonumber \\
&\qquad  + \langle t\rangle^{3}\left(\|\sqrt{ \bar{\rho} }  u\|^2_{2,0}
-E(\partial_1^{2} \eta) \right)+ \gamma^2 \langle t\rangle  { \mu }
 \|\nabla_{\mathcal{A}} \partial_1\mm{curl}_{\mathcal{A}}\eta \|_0^2 /{2},\nonumber \\
&\mathcal{D}_{\mm{D}}:= \gamma\sum_{i=1}^2 \langle t\rangle^i(\gamma \|   u\|_{i,1}^2 + \| \eta\|_{ {1+i},0}^2)+ \langle t\rangle^3 \|   u\|_{2,1}^2 +\gamma^2 \langle t\rangle\| \mm{curl} \eta \|^2_{2,0} ,\nonumber \\
&K_7:=(1+\gamma )K_6+\gamma^2 \langle t\rangle\|\nabla \eta\|_2( \|\eta\|_{2,1}^2+\|\eta\|_{1,2}\|u\|_{\underline{1},2}
+\|u\|_2^2 )+\gamma\langle t\rangle^2 \|\eta\|_{2,1}^4.
\end{align*}

Following the argument of \eqref{202202172007}, we have
\begin{align}
 \|   \mm{curl}  \eta\|^2_{1,1}
 \lesssim&\|\nabla \eta\|_{1,0}^2 + \|  \nabla \mm{curl}  \eta\|^2_{1,0}\nonumber \\
 \lesssim &\|(\nabla \partial_1\eta,   \nabla_{\mathcal{A}}\partial_1 \mm{curl}_{\mathcal{A}}\eta,  \nabla_{\tilde{\mathcal{A}}}\partial_1\mm{curl}_{ {\mathcal{A}}}\eta , \nabla \partial_1\mm{curl}_{\tilde{\mathcal{A}}}\eta   )\|^2_0\nonumber \\
\lesssim& \|(\nabla \partial_1\eta, \nabla_{\mathcal{A}} \partial_1 \mm{curl}_{\mathcal{A}} \eta  )\|^2_0 + \| \eta\|^2_{1,2}   \| \nabla \eta\|^2_2, \nonumber
\end{align}
which, together with \eqref{2022202011749}, implies that
\begin{align}
 \|  \eta \|_{1,2}^2 \lesssim\|(\nabla \partial_1\eta, \nabla_{\mathcal{A}} \partial_1 \mm{curl}_{\mathcal{A}} \eta  )\|^2_0. \nonumber
\end{align}
Thanks to \eqref{2022202011749}, \eqref{202008250745}, \eqref{202012241002} and the above estimate, we easily further  see   that,
 for any given sufficiently large $\gamma $,
\begin{align}
\begin{cases}
\langle t\rangle   \| \partial_2^2 \eta\|_{1,0}^2  + \langle t\rangle^2 \| \partial_2\partial_1\eta\|_{\underline{1},0}^2
+ \langle t\rangle^3 \| \partial_1\eta \|_{\underline{2},0}^2\\
-c(1+\gamma^2)\langle t\rangle^3 \|\eta\|_{2,1}^4 \lesssim  \mathcal{E}_{\mm{D}}\lesssim \gamma^2\langle t\rangle^3\|(\nabla\eta,u)\|_0^2 , \\
 \gamma^2\langle t\rangle  \|\partial_1\eta\|_{\underline{1},1}^2 +\gamma \langle t\rangle^2    \|\partial_1\eta\|_{\underline{2},0}^2 + (\gamma^2\langle t\rangle^2+ \langle t\rangle^3) \|\partial_1u\|_{ {1},1}^2\lesssim  \mathcal{D}_{\mm{D}}.
 \end{cases}
\label{20222020171604}
\end{align}
In addition,
$$
 \left|\int\partial_1\mm{curl}_{\mathcal{A}}   \eta
 \partial_1\mm{curl}_{\mathcal{A}} \left(\bar{\rho}u\right)
 \mm{d}y\right|\lesssim  \|\eta\|_{1,1}\|\nabla u\|_{1} .
$$
Thus integrating \eqref{2020046xx} with some sufficiently large $\gamma$ over $(0,t)$ and then using   \eqref{estemalas}, \eqref{20222020171604}, the above estimate and Young's inequality, we arrive at, for any sufficiently small $\delta$,
\begin{align}
&\langle t\rangle   \| \partial_2^2 \eta\|_{1,0}^2  + \langle t\rangle^2 \| \partial_2\partial_1\eta\|_{\underline{1},0}^2
+ \langle t\rangle^3 \| \partial_1\eta \|_{\underline{2},0}^2\nonumber \\
& +\int_0^t( \langle \tau\rangle  \|\partial_2\partial_1\eta\|_{\underline{1},0}^2 + \langle \tau\rangle^2    \|\partial_1 \eta\|_{\underline{2},0}^2 + \langle \tau\rangle^3 \|\partial_1 u\|_{\underline{1},1}^2)\mm{d}\tau \nonumber \\
&\lesssim  \|(\nabla\eta^0,u^0)\|_2^2+\langle t\rangle^3 \|\eta\|_{2,1}^4 +\int_0^t( \langle \tau\rangle \|  u\|_1\|u\|_3 + K_7)\mm{d}\tau. \label{2022235}
\end{align}

It is easy see  from \eqref{202221235} and \eqref{20222saf201121235} that, for any sufficiently large $\alpha $ (may depending on $\mu$, $g$,   $\lambda$, $m$, $\bar{\rho}$ and $\Omega$),
\begin{align}
&\frac{\mm{d}}{\mm{d}t}(\alpha\langle t\rangle^2\|\nabla_{\mathcal{A}} u \|^2_0+\langle t\rangle^3 (\|\psi\|^2_0- E(  u)) ) +
c( \alpha\langle t\rangle^2\|   u_t\|_0^2+\langle t\rangle^3 \|   u_{t}\|_1^2)\nonumber \\
&\lesssim \alpha \langle t\rangle  \| \nabla_{\mathcal{A}} u \|^2_0  +\langle t\rangle^2(\alpha (\| \eta \|_{2,0}^2 + \|u\|_2^3 +\|u\|_2^2\|q\|_1)\nonumber \\
&\quad +E(u)+\|u\cdot\nabla_{\mathcal{A}}u\|_0^2 )+  \langle t\rangle^3(\|  \eta \|_{2,0}+\|u \|_{2 }     ) \|u\|_2^2.
\label{20235}
\end{align}
Integrating the resulting inequality over $(0,t)$ yields
 \begin{align}
&   \langle t\rangle^2\|\nabla_{\mathcal{A}} u \|^2_0 +\langle t\rangle^3 (\|u_t\|^2_0- E(  u)) +
c\int_0^t(\alpha\langle \tau\rangle^2\|   u_{\tau}\|_0^2+   \langle    \tau \rangle^3 \|   u_{\tau}\|_1^2 )\mm{d}\tau\nonumber \\
&\lesssim \alpha \| \nabla_{\mathcal{A}^0} u^0 \|^2_1+E(u^0)+\|\psi^0\|^2_0  +\langle t\rangle^3 \|u\cdot\nabla_{\mathcal{A}}u  \|^2_0+ \int_0^t ( \alpha \langle \tau\rangle \|  \nabla_{\mathcal{A}}u \|^2_0+\langle \tau\rangle^{2}
  ( \alpha( \| {\eta} \|_{2,0}^2\nonumber \\
 &\quad+\|u\|_2^3+\|u\|_2^2\|q\|_1)+ E(u)+\|u\cdot\nabla_{\mathcal{A}}u  \|^2_0)   +    \langle \tau\rangle^3(\|  \eta \|_{2,0}^2   +\|u \|_{2 } )\|   u\|_{2}^2 ) \mm{d}\tau,
\nonumber \\
&\lesssim  \alpha\| (\nabla\eta^0, u^0) \|^2_2+\langle t\rangle^3 \|u\|^4_2+ \int_0^t (  \langle \tau\rangle \|  u \|^2_1+\langle \tau\rangle^{2}
  ( \alpha( \| {\eta} \|_{2,0}^2\nonumber \\
 &\quad+\|u\|_2^3+\|u\|_2^2\|q\|_1)+\|u\|_{1}^2)   +    \langle \tau\rangle^3(\|  \eta \|_{2,0} +\|u \|_{2 })\|   u\|_{2}^2 ) \mm{d}\tau,
\label{202225}
\end{align}
where we have used \eqref{aprpiosesnew}, \eqref{202012252005} and \eqref{fgestims}  in the last inequality above.
Exploiting \eqref{aprpiosesnew}, \eqref{20250745}, \eqref{omessetsim}   and Young's inequality, we further derive from \eqref{202225} with some sufficiently large $\alpha $ that
 \begin{align}
&\langle t\rangle^3(\|u\|_2^2+ \|q\|_1^2+ \|u_t\|^2_0 ) +
c\int_0^t(\langle \tau \rangle\|u\|_3^2+\langle \tau \rangle^2( \|u\|_{\underline{1},2}^2+\|q\|_{\underline{1},1}^2)+\langle \tau \rangle^3\|   u_{\tau}\|_1^2)\mm{d}\tau\nonumber \\
&\lesssim \|(\nabla\eta^0,u^0)\|_2^2+\langle t\rangle^3 \|\eta\|_{2,0}^2\nonumber \\
 &\quad + \int_0^t\left(\langle \tau\rangle \| {\eta} \|_{ {2},1}^2+  \langle \tau\rangle^{2} \| {\eta} \|_{3,0}^2 +\langle \tau\rangle^3(\|  \eta \|_{2,0}^2 +\|u \|_{2 })\|   u\|_{2}^2 \right) \mm{d}\tau.
\nonumber
\end{align}

Finally, using  \eqref{omessetsim122n}   and Young's inequality,   we further derive from \eqref{estemalas}, \eqref{2022235}  and the above inequality that
\begin{align}
 \mathfrak{E} (t)+c\int_0^t\mathfrak{D}(\tau)\mm{d}\tau
  \lesssim  &\|(\nabla\eta^0,u^0)\|_2^2+\langle t\rangle^3 \|\eta\|_{2,1}^4 +\int_0^t (K_7 +\langle \tau\rangle^3 (\|  \eta \|_{2,0}^2 +\|u \|_{2 } )\|   u\|_{2}^2)   ) \mm{d}\tau,\nonumber
\end{align}
which together with   \eqref{aprpiosesnew}, \eqref{202220201915}  and Young's inequality,  yields \begin{align}
 \mathfrak{E} (t)+ \int_0^t\mathfrak{D}(\tau)\mm{d}\tau
\leqslant c_2  \|(\nabla\eta^0,u^0)\|_2^2, \label{estemalasn0}
\end{align}
see \eqref{2022202180904} and  \eqref{2022202180904x}  for the definitions of $ \mathfrak{E} $ and $\mathfrak{D}$.

Now we sum up the  \emph{a priori} estimates \eqref{estemalas} and \eqref{estemalasn0} as follows.
\begin{pro}Let $(\eta,u,q)$ be a solution to the transformed MRT problem \eqref{01dsaf16asdfasf} and \eqref{20safd45},  and satisfy \eqref{apresnew}--\eqref{aprpiosesnew}. If  $m$ and $\bar{\rho}$ satisfy
the assumptions in Theorem \ref{thm2}, there exists a constant $\delta_1$, depending on $\mu $, $\lambda$, $m$, $g$, $\bar{\rho}$ and $\Omega$, such that the solution  $(\eta,u,q)$ enjoys the \emph{a prior} estimates \eqref{estemalas} and \eqref{estemalasn0}  for any $\delta\leqslant  \delta_1$.
\end{pro}

\subsection{Proof  of Theorem \ref{thm2}}\label{subsec:08}

Now we state the local well-posedness result for the transformed MRT problem.
\begin{pro}\label{202102182115}
Let $b>0$  be a constant  and $\gamma >0$   the same constant in Lemma \ref{pro:1221}. Assume that $\bar{\rho}$ satisfies \eqref{0102},
$(  \eta^0,u^0)\in  \mathcal{H}^{3}_{\mm{s}} \times {\mathcal{H}^2_{\mm{s}}}$, $\|(\nabla \eta^0,u^0)\|_2\leqslant b $  and $\mm{div}_{\mathcal{A}^0}u^0=0$,
where $\mathcal{A}^0:=(\nabla\zeta^0)^{-\mm{T}}$ and $\zeta^0 = \eta^0+y$.
Then, there is a sufficiently small constant $\delta_2\leqslant \gamma /2$, such that if $\eta^0$ satisfies
\begin{align}
\|\nabla \eta^0\|_2\leqslant \delta_2, \nonumber
\end{align}
 the transformed MRT  problem \eqref{01dsaf16asdfasf}  and \eqref{20safd45} admits a unique local-in-time classical  solution
$( \eta, u,q)\in {C}^0(\overline{I_{T}}, \mathcal{H}^{3}_{\mm{s}} )\times {\mathcal{U}}_T \times  (C^0(\overline{I_T} ,\underline{H}^1)\cap L^2_T{H}^2)$ for some $T>0$.
Moreover, $(\eta,u) $ satisfies\footnote{ {
Here the uniqueness means that if there is another solution
$(\tilde{u}, \tilde{\eta},\tilde{q})\in {C}^0 (\overline{I_{T }},\mathcal{H}^{3}_{\mm{s}} )\times \mathcal{U}_{T } \times  (C^0(\overline{I_T} ,\underline{H}^1)\cap L^2_T {{H}^2})$
 satisfying $0<\inf_{(y,t)\in \Omega_T} \det(\nabla \tilde{\eta}+I)$, then
 $(\tilde{\eta},\tilde{u},\tilde{q})=(\eta,u,q)$ by virtue of the smallness condition
 ``$\sup_{t\in \overline{I_T}}\|\nabla \eta\|_2\leqslant 2\delta_2$''. In addition,
  we have, by the fact ``$\sup\nolimits_{t\in \overline{I_T}}\|\nabla \eta\|_2\leqslant  \gamma $" and  Lemma \ref{pro:1221},
$$\inf\nolimits_{(y,t)\in \Omega_T} \det(\nabla \eta+I)\geqslant 1/4.$$}}
$$ \sup\nolimits_{t\in  {I_T}} \| \nabla \eta\|_2\leqslant 2\delta_2.
$$
  Here $\delta_2$ and $T$ may depend on $\mu $,  $g$, $\lambda$, $m$, $\bar{\rho}$ and $\Omega$, while $T$ further depends on $b$.
\end{pro}
\begin{pf} The proof of Proposition \ref{202102182115} will be provided in Section \ref{202202021613}.
\hfill $\Box$
\end{pf}
\begin{rem}\label{2022202021617}
If the initial data $(\eta^0,u^0)$ in Proposition \ref{202102182115} additionally satisfies $\det (\nabla \eta^0+I)=1$ and $( \bar{\rho}\eta_1^0)_\Omega= ( \bar{\rho}u_1^0)_\Omega= 0$, then $(\eta,u)\in \widetilde{\mathfrak{H}}^{1,3}_{\gamma,T}\times  {^0\mathcal{U}_T}$.
\end{rem}

Thanks to the \emph{priori} estimate \eqref{estemalas}  and Proposition  \ref{202102182115}, we can easily establish the global solvability in Theorem \ref{thm2}. Next, we briefly describe the proof.

Let $(\eta^0,u^0)$ satisfy the assumptions in Theorem \ref{thm2},
\begin{align}
\|(\nabla \eta^0,u^0)\|_2 \leqslant\delta /\sqrt{2},\
\delta=\min\left\{ \delta_1, \delta_2  \right\}/ \sqrt{{c}_3}\mbox{ and } c_3= c_1+c_2 \geqslant 1 ,\nonumber
\end{align}
where $c_1$ and $c_2$ are the same constants in \eqref{estemalas} and \eqref{estemalasn0}, resp..
By virtue of Proposition \ref{202102182115} and Remark \ref{2022202021617}, there exists a unique local solution $(\eta,u,q)$ to the transformed MRT problem
\eqref{01dsaf16asdfasf} and   \eqref{20safd45} with a maximal existence time $T^{\max}$, which satisfies
\begin{itemize}
  \item for any $a\in I_{T^{\max}}$,
the solution $(\eta,u,q)$ belongs to $\widetilde{\mathfrak{H}}^{1,3}_{\gamma,a}\times {^0\mathcal{U}}_{a} \times ( C^0(\overline{I_a }, \underline{H}^1)\cap L^2_a {{H}^2})$,
$ \sup\nolimits_{t\in  {I_a}} \|\nabla \eta\|_2\leqslant 2\delta_2$;
  \item $\limsup_{t\to T^{\max} }\|\nabla \eta (t)\|_2 >\delta_2$ or $\limsup_{t\to T^{\max} }\|(\nabla \eta,u)( t)\|_2=\infty$, if $T^{\max}<\infty$.
\end{itemize}

Let
\begin{align}  \nonumber
T^{*}:=&\sup \{\tau\in I_{T^{\max}}~ |~  \|(\nabla \eta,u)(t)\|_2^2 +\langle t\rangle^2( \|\eta(t)\|_{2,1}^2+\|u(t)\|_2^2) \leqslant c_3 \delta^2 \mbox{ for any }t\leqslant\tau \}.\nonumber
\end{align}
It is easy to see that the definition of $T^*$ makes sense.
 Thus, to show the existence of a global solution, it suffices to verify $T^*=\infty$.
We shall prove this fact by contradiction below.

Assume $T^*<\infty$, then, by Proposition \ref{202102182115},
\begin{align}T^*\in (0,T^{\max})
\label{20222022519850}
 \end{align}and
$$ (\|(\nabla \eta,u)(t)\|_2^2 +\langle t\rangle^2( \|\eta\|_{2,1}^2+\|u\|_2^2))|_{t=T^*} = c_3 \delta^2 . $$

 Noting that
\begin{equation}
\label{201911262sadf202}
\sup\nolimits_{\overline{I_{T^{*}}}}(\|(\nabla \eta,u)(t) \|_2  +\langle t\rangle^2( \|\eta\|_{2,1}^2+\|u\|_2^2))   \leqslant c_3 \delta^2 \leqslant \delta_1,
\end{equation}
thus, using \eqref{201911262sadf202} and a standard regularization method,
we can follow the same arguments as in the derivation of \eqref{estemalas} and \eqref{estemalasn0} to verify  that
\begin{align*}
 \mathcal{E}(t)+\mathfrak{E} (t)+ \int_0^{T^{*}} ( \mathcal{D}(\tau) + \mathfrak{D}(\tau))\mm{d}\tau \leqslant {c}_3 \|(\nabla \eta^0,u^0)\|_2^2 \leqslant {{c}_3} \delta^2/2.
\end{align*}
In particular,
\begin{align}
{ \sup\nolimits_{t\in \overline{I_{T^{*}}}}(\|(\nabla \eta,u)(t) \|_2  +\langle t\rangle^2( \|\eta\|_{2,1}^2+\|u\|_2^2))  \leqslant  {{c}_3} \delta^2/2.}  \label{2020103261534}
\end{align}

By \eqref{20222022519850}, \eqref{2020103261534} and the strong continuity $(\nabla\eta,u)\in C^0([0,T^{\max}), H^2)$,   there exists $\tilde{T}\in (T^*,T^{\max})$  such that
\begin{align}
 \sup\nolimits_{t\in \overline{I_{\tilde{T} }}} (\|(\nabla \eta,u)(t) \|_2^2+\langle t\rangle  \| \eta\|_{1,2}^2 +\langle t\rangle^2( \|\eta\|_{2,1}^2+\|u\|_2^2)) \leqslant  {{c}_3}\delta^2  , \nonumber
\end{align}
which contradicts with the definition of $T^*$. Hence, $T^*=\infty$ and thus $T^{\max}=\infty$.
This completes the proof of the existence of a global solution. The uniqueness of the global solution is obvious due to
the uniqueness result of the local solutions in Proposition \ref{202102182115}
and the fact $\sup_{t\geqslant 0}\|\nabla \eta\|_2\leqslant 2 \delta_2$.

To complete the proof of Theorem \ref{thm2}, we have to show that the solution $(\eta,u,q)$ satisfies the properties \eqref{1.200}--\eqref{1.200xx}.
Recalling the derivation of  \eqref{estemalas} and \eqref{estemalasn0}, we easily verify that
 the global solution $(\eta,u)$ enjoys \eqref{1.200} and \eqref{1.200n0} by a standard regularization method. Hence next it suffices to show \eqref{1.200xx}.

 From \eqref{1.200n0} we get
\begin{align}
\label{202101242112}
\partial_1\eta(t)\to 0\mbox{ in }H^2 \mbox{ as }t\to \infty
\end{align}
and
\begin{align}
\label{20safda2101242112}
\left\|\int_0^t u\mm{d}\tau\right\|_2\lesssim\int_0^t \|u  \|_2\mm{d}\tau\lesssim \sqrt{\| \nabla \eta^0\|_{2}^2+ \|u^0\|_2^2 }\end{align} for any $t>0$.
Due to \eqref{20safda2101242112}, there are a subsequence $\{t_n\}_{n=1}^{\infty}$ and some function $\eta^\infty_1\in H^2$, such that
$$ \int_0^{t_n} u_1\mm{d}\tau \to\eta^\infty_1-\eta^0_1\mbox{ weakly in }H^2.$$
Utilizing $\eqref{01dsaf16asdfasf}_1$, \eqref{1.200n0}, and the weakly lower semi-continuity, we conclude
$$
\begin{aligned}
\|\eta_1(t)-\eta^\infty_1\|_2\leqslant &\liminf_{t_n\to \infty} \left\|\int_t^{t_n } u_1\mm{d}\tau\right\|_2\\
\lesssim &\sqrt{\|  \eta^0\|_{2,1}^2+ \|u^0\|_2^2 } \liminf_{t_n \to \infty}\int_t^{t_n} \langle \tau \rangle^{-3/2}\mm{d}\tau
 \nonumber \\
 \lesssim &\sqrt{\|\nabla \eta^0\|_{2}^2+ \|u^0\|_2^2 }\langle t \rangle^{-1/2},
\end{aligned}$$
which, combined with \eqref{202101242112}, yields \eqref{1.200xx} holds and that $\eta_1^\infty $ depends on ${y_2}$ only.
This completes the proof of Theorem \ref{thm2}.

\section{Proof of Theorem \ref{thm1}}\label{sec:instable}
The existence of RT instability solutions had been widely investigated, see \cite{JFJSO2014,JFJSZWC} for examples. In particular, Jiang et.al. proved the
 the existence of RT instability solutions for the stratified incompressible viscous resistive MHD fluids with Dirichlet boundary conditions on the both of upper and lower boundaries of a slab domain \cite{JFJSZWC}.  We can easily establish the instability result of the transformed MRT problem  in Theorem \ref{thm1} by following the proof frame in \cite{JFJSZWC}. Next we briefly sketch the proof for the completeness.  In what follows, the fixed positive constant $c_i^I$ for $i\geqslant 1$
may depend on $\mu $, $g$,  $\lambda$, $m$, $\bar{\rho}$ and the domain $\Omega$.

To being with, we introduce the instability result for the  linearized MRT problem under the instability condition $|m|\in[0,m_{\mm{C}})$.
\begin{pro}\label{pro:08252100}
Let $\mu >0$ and $\bar{\rho}$ satisfy \eqref{0102} and \eqref{0102n}. If  $|m|\in[0,m_{\mm{C}})$,
 then the zero solution
 is unstable to the linearized MRT problem
\begin{equation}\label{01dsaf16asdfasf0101}
                              \begin{cases}
\eta_t=u,   \\[1mm]
\bar{\rho}u_t+\nabla q-\mu  \Delta u=\lambda m^2\partial_1^2\eta+g\bar{\rho}'\eta_2\mathbf{e}_2   ,\\[1mm]
\div u=0 , \\[1mm]
(\eta_2,\partial_2\eta_1,u_2,\partial_2u_1)|_{\partial\Omega}=0.
\end{cases}
\end{equation} That is, there is an unstable solution
$(\eta, u,  q):=e^{ \Lambda   t}(w/ \Lambda ,w, \beta )$
 to the above problem \eqref{01dsaf16asdfasf0101}, where
 \begin{equation}\nonumber
 (w, \beta )\in {^0\mathcal{H}^3_{\mm{s}}}\times\underline{H}^2
 \end{equation}
 solves  the boundary-value problem  \begin{equation*}
                              \begin{cases}
\Lambda ^2\bar{\rho}w+ \Lambda  \nabla  \beta
- \Lambda\mu  \Delta w=  m^2\partial_1^2w+g\bar{\rho}'w_2\mathbf{e}_2
 ,\\[1mm]
\div w=0  , \\[1mm]
(w_2,\ \partial_2 w_1)|_{\partial\Omega}=0 .
\end{cases}
\end{equation*}
  with some growth rate $ \Lambda>0 $, where $\Lambda$ satisfies
\begin{equation}
E(v)\leqslant  {\Lambda^2}\|\sqrt{\bar{\rho}}v\|_0^2+ { \Lambda }\mu  \|\nabla v\|_0^2 \mbox{ for any }v\in H_{\sigma}^{1} . \label{Lambdard} \end{equation}
In addition,
\begin{align}  \label{201602081445MH}
 \int \bar{\rho}'|w_2|^2\mm{d}y\|  w_i\|_{0}\|\partial_1w_i\|_{0}\|\partial_2w_i\|_0\neq 0  \mbox{ for }i=1,\ 2. \end{align}
\end{pro}
\begin{pf} We can use the modified variational method  as in \cite[Proposition 3.1]{JFJSZWC} and Lemma \ref{20222011161567} to easily verify Proposition \ref{pro:08252100}, and thus omit the trivial proof.
\hfill$\Box$
\end{pf}

Then  we follow the derivation of \emph{a priori} stability estimate \eqref{estemalas} in Section \ref{202220201221244} with some slight modifications to establish the following Gronwall-type energy inequality for the solutions of
the transformed MRT problem.
\begin{pro}  \label{pro:0301n0845}
Let $\Lambda>0$ be the same as in Proposition \ref{01dsaf16asdfasf0101} and $(\eta,u,q)$ be the local solution constructed
by Proposition \ref{202102182115} with initial condition $(\eta^0,u^0)\in {^0{\mathcal{H}}^{3,\mm{s}}_{1,\gamma}} \times {^0\mathcal{H}^2_{\mm{s}}}$. There are two constants $\delta_1^I$ and $ {c}_1^I>0$, such that if $\|(\nabla \eta,u)\|_2\leqslant\delta_1^I$
in some time interval $ I_{\tilde{T}}\subset I_T$ where $I_T$ is the existence time interval of $(\eta,u,q)$, then  $(\eta,u,q)$  satisfies the Gronwall-type energy inequality:   for any $t\in I_{\tilde{T}}$,
\begin{align}
{\mathcal{E}}(t) +c\int_0^t\mathcal{D}(\tau)\mm{d}\tau\leqslant &   {c}_1^I \left(\|(\nabla \eta^0, u^0)\|_2^2
+\int_0^t\|(\eta_2,u_2)(\tau)\|_{0}^2 \mm{d}\tau\right),\label{2019120521430850}
\end{align}
where the constants $\delta_1^I$ may depend on $\mu$, $g$,   $\lambda$,  $m$, $\bar{\rho}$ and $\Omega$.
\end{pro}
\begin{pf}
Let  $(\eta,u,q)$ be the local solution constructed by Proposition  \ref{202102182115}. Then $(\eta,u)\in \widetilde{\mathfrak{H}}^{1,3}_{\gamma,T}\times  {^0\mathcal{U}_T}$. We further assume
\begin{align*}
\|(\nabla \eta,u)\|^2_2\leqslant \delta\in (0,1]\;\;\mbox{ for any }t\in I_{\tilde{T}}\subset I_T.
\end{align*}

Recalling the derivation of \eqref{20191005asdf0940} and using the regularity of $(\eta,u,q)$, we easily verify that, for sufficiently small $\delta$,
\begin{align}\label{201sdf0940}
\frac{\mm{d}}{\mm{d}t}\Xi+c  \mathcal{D} \lesssim \|(\eta_2,u_2)\|_{\underline{2},0}^2,
\end{align}
 where $\Xi$ has been defined by $\tilde{\mathcal{E}}$ with $g=0$ and satisfies
\begin{align}&\label{11111648}
\Xi,\ \mathcal{E}  \mbox{ and }\|(\nabla \eta,u)\|^2_2\mbox{ are equivalent to each other}     .\end{align}
It should be  noted that the equivalent coefficients in \eqref{11111648}  are independent of $\delta$.

By the interpolation inequality \eqref{201807291850}, we have, for any $\varepsilon\in (0,1]$,
\begin{align}
 \| \chi_2\|_{k,0}  \lesssim
                                         \begin{cases}
                       \varepsilon^{-1}  \|\chi_2\|_0 + \varepsilon \| \chi_2\|_2   & \hbox{for }k=1; \\
     \varepsilon^{-1}\| \chi_2\|_{1,0} + \varepsilon \| \chi _2\|_{1,2} & \hbox{for }  k=2,
                                         \end{cases}
  \label{202102181813}
\end{align}
where $\chi=\eta$ or $u$.
Therefore, with the help of \eqref{11111648}--\eqref{202102181813}, we easily deduce \eqref{2019120521430850} from \eqref{201sdf0940}  for sufficiently small $\delta$. The proof is complete.
\hfill$\Box$
\end{pf}

For any given $\delta>0$, let
\begin{equation}\label{0501}
\left(\eta^\mm{a}, u^\mm{a}, q^\mm{a}\right)=\delta e^{ \Lambda  t } (\tilde{\eta}^0, \tilde{u}^0, \tilde{q}^0),
\end{equation}
where $ (\tilde{\eta}^0, \tilde{u}^0, \tilde{q}^0):=(w/\Lambda ,w,\beta )$, and $(w,\beta ,\Lambda)$ is given by Proposition \ref{pro:08252100}.
Then $\left(\eta^\mm{a},u^\mm{a},q^\mm{a}\right)$ is also a  solution to the linearized MRT problem \eqref{01dsaf16asdfasf0101}, and enjoys the estimate: for $j\geqslant 0$,
\begin{equation}
\label{appesimtsofu1857}
\|\partial_{t}^j(\eta^\mm{a}, u^\mm{a})\|_3+\|\partial_{t}^j q^\mm{a} \|_2=\Lambda^j\delta e^{\Lambda t}(\|(\tilde{\eta}^0,\tilde{u}^0)\|_3
+\|\tilde{q}^0\|_2)\lesssim \Lambda^j \delta e^{\Lambda t}.
\end{equation}
In addition, we have by \eqref{201602081445MH} that
\begin{eqnarray}\label{n05022052}
\|\bar{\rho}'\chi_2\|_0\|\chi _i\|_{0}\|\partial_1\chi _i\|_{0}\|\partial_2\chi _i\|_{0} >0,
\end{eqnarray}
where $\chi  =\tilde{\eta}^0 $ or $\tilde{u}^0 $, and $i=1$, $2$.

 Since the initial data of the solution $( {\eta}^{\mm{a}},{u}^{\mm{a}},{q}^{\mm{a}})$ to the linearized MRT problem
may not satisfy the necessary compatibility conditions required by initial data of the corresponding nonlinear transformed MRT problem.
So, we shall modify the initial data of the linearized problem
as in \cite[Proposition 5.1]{JFJSZWC}, such that the modified initial data approximates the original initial data
of the linearized problem, and satisfies the compatibility conditions for the corresponding nonlinear problem.
\begin{pro}\label{pro:0101}
Let $(\tilde{\eta}^0, \tilde{u}^0 ):=(w/\Lambda , w)$  be the same as in \eqref{0501}.
 There is a constant $\delta_2^I \in (0,1]$, such that for any $\delta\in(0, \delta_2^I ]$, there exists
$(\eta^{\mm{r}}, u^{\mm{r}})\in {^0\mathcal{H}^3_{\mm{s}}}$ enjoying the following properties:
\begin{enumerate}
\item[(1)] The modified initial data
\begin{align}\nonumber
({\eta}^{\delta}_{0},{u}^{\delta}_{0 } )
:=\delta(\tilde{\eta}^0,\tilde{u}^0 )+\delta^2(\eta^{\mm{r}},u^{\mm{r}} )
\end{align}
belongs to
${^0{\mathcal{H}}^{3,\mm{s}}_{1,\gamma}} \times {^0\mathcal{H}^3_{\mm{s}}}$ and satisfies
the compatibility condition
\begin{align}\nonumber
&\mm{div}_{\mathcal{A}_{0 }^{\delta}}u_{0 }^{\delta}=0  \mbox{ in } \Omega,
\end{align}  where  $\mathcal{A}^{\delta}_0$ is defined as $\mathcal{A}$ with $\eta^{\delta}_{0 }$ in place of $\eta$.
\item[(2)]
Uniform estimate:
\begin{align}
\label{202103281107}
 \| (\eta^{\mm{r}},u^{\mm{r}}  )\|_3 \leqslant {c}_2^I,
\end{align}
where the positive constant $ {c}_2^I $ is independent of $\delta$.
\end{enumerate}
\end{pro}
\begin{pf}
Thanks to Lemma \ref{20222011161567} and Remark \ref{202202202212011}, we can easily establish Proposition \ref{pro:0101}   by following the argument of \cite[Proposition 5.1]{JFJSZWC}, and thus omit the trivial proof.
\hfill$\Box$
\end{pf}

Let $({\eta}^{\delta}_{0},{u}^{\delta}_0 )\in {^0{\mathcal{H}}^{3,\mm{s}}_{1,\gamma}} \times {^0\mathcal{H}^3_{\mm{s}}} $ be constructed by  Proposition \ref{pro:0101},
\begin{align}
&\label{201912041727}
 {c}_3^I= {\|(\tilde{\eta}^0,\tilde{u}^0)\|_3}  + {c}_2^I>0
 \end{align}and
 \begin{align}
&\label{201912041705}
 {\delta}_0= \frac{1}{2c_3^I } \min\left\{\gamma,{\delta_1^I}, {{\delta}_2} , 2c_3^I {\delta}_2^I \right\}\leqslant 1 .
\end{align}
From now on, we assume that $\delta\leqslant {\delta}_0$. Since $\delta\leqslant\delta_2^I$, we can use Proposition \ref{pro:0101} to construct
$(\eta_0^\delta, u_0^\delta)$ that satisfies
$$  {\|(\eta_0^\delta,u_0^\delta)\|_3}
\leqslant c_3^I\delta\leqslant  {\delta}_2 .$$
By Proposition  \ref{202102182115}, there exists a unique solution $(\eta, u, q)$ of the transformed MRT problem \eqref{01dsaf16asdfasf} and \eqref{20safd45}
with initial value $({\eta}_0^\delta, {u}_0^\delta)$ in place of $(\eta^0, u^0)$, where $(\eta, u, q)\in \widetilde{\mathfrak{H}}^{1,3}_{\gamma,\tau}\times  {^0\mathcal{U}_{\tau}} \times ( C^0(\overline{I_\tau}, \underline{H}^1)\cap L^2_{\tau} {H}^2) $ for any $\tau\in I_{T^{\mm{\max}}}$ and $T^{\mm{max}}$ denotes the maximal time of existence.

Let $\epsilon_0\in (0,1]$ be a constant, which will be given in \eqref{201907111842}.
We define
\begin{align}\label{times}
& T^\delta:=\Lambda^{-1}\mm{ln}({\epsilon_0}/{\delta})>0, \mbox{ i.e. }
 \delta e^{ \Lambda  T^\delta }=\epsilon_0,\\[1mm]
&T^*:=\sup\left\{t\in I_{T^{\max}}\left|~
\sup\nolimits_{\tau\in [0,t)}\sqrt{\|\eta(\tau)\|_3^2+\| u (\tau)\|_2^2}\right.\leqslant 2   {c}_3^I \delta_0 \right\}, \label{xfdssdafatimes}\\[1mm]
&  T^{**}:=\sup\left\{t\in I_{T^{\max}} \left|~\sup\nolimits_{\tau\in [0,t)}\left\|\eta(\tau)\right\|_{0}\leqslant 2 {c}_3^I\delta  e^{\Lambda \tau}
 \right\}.\right.  \label{xfdstimes}
\end{align}

Since
\begin{align}
 \left.\sqrt{\|\eta(t)\|_3^2+\| u(t)\|_2^2} \right|_{t=0}
=  \sqrt{\|\eta_0^\delta\|_3^2+\|u_0^\delta\|_2^2}
\leqslant c_3^I \delta  <  2c_3^I \delta ,
\label{201809121553}
\end{align} we have $T^{*}>0$, $T^{**}>0$. Obviously,
\begin{align}
\label{0n111} & T^{*}=T^{\max}=\infty \mbox{ or }T^{*}<T^{\max},\\
\label{0502n111} & \left\|\eta (T^{**}) \right\|_0 =2{c}_3^I\delta e^{\Lambda T^{**}},\mbox{ if }T^{**}<T^{\max},\\
\label{050211} &\sqrt{\|\eta (T^{*})\|_3^2+\| u  (T^{*})\|_2^2} =2   {c}_3^I \delta_0,\mbox{ if }T^{*}<T^{\max}.
\end{align}

From now on, we define
$$T^{\min}:=\min\{T^\delta,T^*,T^{**}\} .$$
Noting that $\sup\nolimits_{t\in [0,T^{\min} )}\sqrt{\|\eta(t)\|_3^2+\| u (t)\|_2^2} \leqslant     \delta_1^I$,
thus, by Proposition \ref{pro:0301n0845}, $(\eta,u,q)$ enjoys  the  Gronwall-type energy inequality \eqref{2019120521430850}  for any $t\in I_{T^{\min}}$.
Making use of this fact, \eqref{xfdssdafatimes}--\eqref{201809121553}, Lemma \ref{pro:1221} and the condition $\|\eta\|_3\leqslant \gamma$,  we see that, for any $t\in [0 , T^{\min})$,
\begin{align}
  \mathcal{E} (t) + c\int_0^t \mathcal{D}(\tau) \mm{d}\tau  \leqslant c_4^I\delta^2e^{2\Lambda  t}
\label{20191204safda2114}\end{align}
and
\begin{align}  \left\|\int_{0}^{\eta_2 }\left(\eta_2 -
z\right)\bar{\rho} {''}(y_2+z)\mm{d}z \right\|_{{L^1}}\leqslant (c_4^I \delta  e^{ \Lambda t})^2 . \label{2022202101204}\end{align}

 Next we further establish the estimates for the errors between $(\eta,u)$ and $(\eta^{\mm{a}}, u^{\mm{a}})$  as in \cite[Proposition 6.1]{JFJSZWC}.
\begin{pro}\label{2022202101315}
Let
$(\eta^{\mathrm{d}}, u^{\mathrm{d}},q^{\mathrm{d}})=(\eta, u,q)-(\eta^{\mm{a}}, u^{\mm{a}}, q^{\mm{a}})$, then there is a constant
${c}_5^I$, such that for any $\delta\in(0,\delta_0]$ and for any $t\in I_{T^{\min}}$,
\begin{align}
\label{ereroe}
&\|(\bar{\rho}'\chi^{\mm{d}},\chi^{\mm{d}} ) \|_{\mathfrak{X}} +  \|u_t^{\mm{d}}\|_{0}   \leqslant {c}_5^I\sqrt{\delta^3e^{3\Lambda t}}, \\
&\|(\mathcal{A}_{ik}\partial_k\chi_j
-\partial_i \chi^{\mm{a}}_j)(t)
\|_{{L^1}}\leqslant  {c}_5^I\sqrt{\delta^3 e^{3\Lambda t}},\label{2018090119021}
\end{align}
where $i$, $j=1$, $2$, $\chi=\eta$ or $u$, $ \mathfrak{X}=W^{1,1}$ or $H^1$, and ${c}_5^I$ is independent  of  ${T^{\min}}$.
\end{pro}
\begin{pf}
Let $\mathcal{K}^1=\mathcal{K}^1_{\mm{L}}+\mathcal{K}^1_{\mm{N}}$,
\begin{align*}
\mathcal{K}^1_{\mm{L}} = \lambda m^2\partial_1^2\eta^{\mathrm{a}}+g\bar{\rho}'\eta_2^{\mathrm{a}}\mathbf{e}_2   +\mu   \Delta u^{\mathrm{a}}
\end{align*}
and \begin{align*}
\mathcal{K}^1_{\mm{N}} =\nabla \eta(\lambda m^2\partial_1^2\eta^{\mathrm{a}}+g\bar{\rho}'\eta_2^{\mathrm{a}}\mathbf{e}_2 ) +\bar{\rho}\nabla u u^{\mathrm{a}}+\mu  \nabla \eta \Delta u^{\mathrm{a}}  +\nabla_{\tilde{\mathcal{A}}}q^{\mathrm{a}}  -\nabla \eta\nabla q^{\mathrm{a}}.
\end{align*}
It is easy to see that $(\eta^{\mathrm{a}},u^{\mathrm{a}},q^{\mathrm{a}})$ satisfies that
\begin{equation}\label{01dasf0101}
                              \begin{cases}
\eta_t^{\mathrm{a}}=u^{\mathrm{a}},   \\[1mm]
\bar{\rho}(\nabla \zeta u^{\mathrm{a}})_t+\nabla_{\mathcal{A}}   q^{\mathrm{a}}
= \mathcal{K}^1,\\[1mm]
\div_{\mathcal{A}}(\nabla \zeta u^{\mathrm{a}} )= \div u^{\mathrm{a}}=0 , \\[1mm]
(X_2,\partial_2X_1   )|_{\partial\Omega}=0,
\end{cases}
\end{equation}
where $\zeta:=\eta+y$, and $X=\eta^{\mathrm{a}}$, $u^{\mathrm{a}}$ or $  \nabla \zeta u^{\mathrm{a}} $.
Subtracting the both of the transformed MRT problem and the above problem, we get
\begin{equation}\label{201702052209} \begin{cases}
\eta_t^{\mathrm{d}}=u^{\mathrm{d}} ,\\[1mm]
\bar{\rho} \bar{u}_{t}^{\mathrm{d}}+\nabla_{\mathcal{A}}q^{\mm{d}}- \mu\Delta_{\mathcal{A}} u   =:\mathcal{K}^2 , \\
\div_{\mathcal{A}}  \bar{u}^{\mathrm{d}}=\div_{\mathcal{A}}  u=0 ,
 \\(Y_2,\partial_2Y_1 )|_{\partial\Omega}=0 ,\\(\eta^{\mathrm{d}},u^{\mathrm{d}})|_{t=0}=\delta^2 (\eta^{\mm{r}},u^{\mm{r}}), \end{cases}\end{equation}
 where $\bar{u}^{\mathrm{d}}:=u-\nabla \zeta u^{\mathrm{a}} $, $\mathcal{K}^2=\lambda m^2\partial_1^2\eta +G_\eta \mathbf{e}_2-\mathcal{K}^1 $, and $Y=\eta^{\mathrm{d}} $, $u^{\mathrm{d}} $ or $\bar{u}^{\mathrm{d}}  $.

Let $\Phi= \bar{u} ^{\mathrm{d}} _t- u \cdot \nabla_{\mathcal{A}} \bar{u}^{\mathrm{d}}  $,
\begin{align*}
I_1^I =  &  \int
 ((\mu   \mathcal{A}_{il} \partial_l(\mathcal{A}_{ik} \partial_k  u) +    \bar{\rho} u\cdot\nabla_{\mathcal{A}}    \bar{u}^{\mathrm{d}}  + \mathcal{K}^2     +  \bar{\rho}\Phi )\cdot (u \cdot\nabla_{\mathcal{A}}  \Phi )-   \partial_t (  \bar{\rho} u\cdot\nabla_{\mathcal{A}}   \bar{u}^{\mathrm{d}}) \cdot \Phi  )\mm{d}y
\end{align*}
and
\begin{align*}
I_{2}^I =  &  \int
(\mathcal{K}^2_t \cdot \Phi  -  \mu\partial_t (    \mathcal{A}_{il}\mathcal{A}_{ik} \partial_k  u  )\cdot \partial_l \Phi  )\mm{d}y.
\end{align*}
Following the argument of \eqref{eq0510} in Section \ref{202202021613},
we can deduce from \eqref{201702052209}$_2$--\eqref{201702052209}$_4$ with $X= \bar{u}^{\mathrm{d}}$ and the boundary conditions of $(\eta,u)$ that
\begin{align}
 \frac{1}{2}\frac{\mm{d}}{\mm{d}t}\|\sqrt{ \bar{\rho}  } \Phi\|^2_0
   = I_1^{{I}}+I_2^{{I}}. \label{eq0dsdfa510xx}
\end{align}

Obviously the integral $I_2^{{I}} $ can be rewritten as follows:
\begin{align}
I_2^{{I}}:=  \frac{1}{2}\frac{\mm{d}}{\mm{d}t}E(u^{\mm{d}})  -\mu \|\nabla u^{\mm{d}} _t\|_0^2+ \tilde{I}_2^{{I}},
\label{2022202101434}
\end{align}
where
\begin{align*}
\tilde{I}_2^{{I}}= &\int
(\mu\partial_t (    \mathcal{A}_{il}\mathcal{A}_{ik} \partial_k  u )\cdot \partial_l (  u \cdot \nabla_{\mathcal{A}} \bar{u}^{\mathrm{d}}  + \partial_t (\nabla \eta u^{\mathrm{a}} )) \nonumber \\
&-(\lambda m^2\partial_1^2u+g\bar{\rho}'(\eta_2+y_2)u_2\mathbf{e}_2 - \mathcal{K}^1_t) \cdot (u \cdot \nabla_{\mathcal{A}} \bar{u}^{\mathrm{d}}   \\
& +\partial_t(\nabla \eta u^{\mathrm{a}} )) -  \mu\partial_t (    \tilde{\mathcal{A}}_{il}\mathcal{A}_{ik} \partial_k  u + \tilde{\mathcal{A}}_{lk} \partial_k  u  )\cdot \partial_l  {u} ^{\mathrm{d}}_t
\\
& + (g(\bar{\rho}'(\eta_2+y_2)-\bar{\rho}'( y_2) )u_2\mathbf{e}_2  -\partial_t\mathcal{K}^1_{\mm{N}})\cdot   {u} ^{\mathrm{d}} _t) \mm{d}y,\end{align*}
Putting \eqref{2022202101434} into \eqref{eq0dsdfa510xx} yields
\begin{align}
 \frac{1}{2}\frac{\mm{d}}{\mm{d}t}(\|\sqrt{ \bar{\rho}  } \Phi\|^2_0-E(u^{\mm{d}})  )
 +\mu \int_0^t\|\nabla u^{\mm{d}} _{\tau}\|_0^2\mm{d}\tau
   = I_1^{{I}}+\tilde{I}_2^{{I}}  . \label{eq0dsdfa510}
\end{align}
Integrating the above identity in time from $0$ to $t$, we get
\begin{align}
& \|\sqrt{ \bar{\rho}  } \Phi\|^2_0-E(u^{\mm{d}})+2 \mu\int_0^t\|\nabla_{\mathcal{A} } u^{\mathrm{d}}_{\tau}\|_0^2\mm{d}\tau\nonumber \\
  &= \|\sqrt{ \bar{\rho}  } \Phi|_{t=0}\|^2_0 - {E}(u^{\mm{d}}|_{t=0})+ 2 \int_0^t( I_1^{{I}}+\tilde{I}_2^{{I}}) (\tau)\mm{d}\tau . \label{0314}
\end{align}

Making use of \eqref{appesimtsofu1857}, \eqref{20191204safda2114}, the fact $\delta e^{\Lambda t}\leqslant 1$ and the initial condition $u^\mm{d}(0)=\delta^2 u^{\mm{r}}$ in \eqref{201702052209}$_5$,  we  easily estimate that
\begin{align}
& \|\sqrt{ \bar{\rho}  } u^\mm{d}_t \|^2_0 = \|\sqrt{ \bar{\rho}  } \Phi\|^2_0-
\|\sqrt{\bar{\rho}} (u \cdot \nabla_{\mathcal{A}} \bar{u}^{\mathrm{d}}+ \partial_t( \nabla \eta u^{\mathrm{a}} ) )\|^2_0\nonumber \\
& \qquad \qquad  \quad  + 2\int\bar{\rho}(  u \cdot \nabla_{\mathcal{A}} \bar{u}^{\mathrm{d}}+\partial_t(\nabla \eta u^{\mathrm{a}} ))\cdot u^\mm{d}_t\mm{d}y
  \leqslant \|\sqrt{ \bar{\rho}  } \Phi\|^2_0+    \delta^3 e^{3\Lambda t},    \label{201811081020} \\
&\|\sqrt{ \bar{\rho}  } \Phi|_{t=0}\|^2_0\lesssim\|  u_t^{\mathrm{d}}|_{t=0}\|_0^2 + \delta^3 e^{3\Lambda t},\\
&
E(u^{\mm{d}}|_{t=0})\lesssim \delta^4\|u^{\mm{r}}\|_1^2\lesssim  \delta^3 e^{3\Lambda t},
\label{2018110702038}\\
&\int_0^t (I_1^{{I}}+\tilde{I}_2^{{I}})(\tau)\mm{d}\tau  \lesssim   \delta^3 e^{3\Lambda t}  . \label{2011020}
\end{align}

Next we shall estimate for $\|  u_t^{\mathrm{d}}|_{t=0}\|^2_0  $.
Noting that
$$\div  u_t^{\mm{d}}=  -\mm{div}\partial_t( \tilde{\mathcal{A}}^{\mm{T }}  u)$$
and
$$\bar{\rho} u_{t}^{\mathrm{d}} =
   \mu \Delta u^{\mm{d}}+\lambda m^2\partial_1^2\eta^{\mm{d}}+
g\bar{\rho}'\eta_2^{\mm{d}}\mathbf{e}_2 +\mathcal{N}^\mu  +\mathcal{G}\mathbf{e}_2 -\nabla  q-\nabla_{\tilde{\mathcal{A}}} q
 $$(see \eqref{01dsaf16asdfasf03n} for the definition of $\mathcal{N}^\mu$),
thus, multiplying the above identity  by $u_t^{\mm{d}}$ in $L^2$, and then using the integration by parts,  we have
\begin{align}
\int \bar{\rho} |u_{t}^{\mathrm{d}}|^2\mm{d}y=&
  \int\bigg( \mu \Delta u^{\mm{d}}+\lambda  m^2\partial_1^2\eta^{\mm{d}}+
g\bar{\rho}'\eta_2^{\mm{d}}\mathbf{e}_2 +\mathcal{N}^\mu +\mathcal{G}\mathbf{e}_2 -\nabla_{\tilde{\mathcal{A}}} q\bigg)\cdot u_t^{\mm{d}}   \mm{d}y \nonumber \\
&+
\int  \nabla q^{\mathrm{d}}\partial_t({\tilde{\mathcal{A}}}^{\mm{T}} u)  \mm{d}y  .\nonumber
  \end{align}
Thus we immediately  derive from the above estimate that
$$  \|  u_t^{\mathrm{d}}\|_0^2 \lesssim\| (\eta^{\mm{d}},u^{\mm{d}})\|_2^2+ \delta^3 e^{3\Lambda t},$$
which, together with the initial data \eqref{201702052209}$_5$, implies that
\begin{equation}
\label{201807112050}
\|  u_t^{\mathrm{d}}|_{t=0}\|_0^2\lesssim  \delta^3 e^{3\Lambda t} .
\end{equation}

Consequently, putting \eqref{201811081020}--\eqref{201807112050} into \eqref{0314}  yields
\begin{align}
 \|\sqrt{\bar{\rho}} u_t^\mm{d}\|^2_{0}+ 2\mu\int_0^t\| \nabla_{\mathcal{A} } u^{\mathrm{d}}_{\tau}\|_0^2\mm{d}\tau\leqslant   {E}( u^{\mm{d}})+  c\delta^3 e^{3\Lambda t}.\label{0314xxdfafdss}
\end{align}

Since $u^{\mm{d}}$ does not satisfies the divergence-free condition (i.e.,
 $\mm{div}u^{\mm{d}}=0$), thus we can not use \eqref{Lambdard} to deal with $ E(u^{\mm{d}})$.
To overcome this trouble, we consider the following  Stokes problem
\begin{equation*}
 \begin{cases}
\nabla \varpi-\Delta \tilde{u}=0,\quad
 \mm{div}\tilde{u} = \mm{div}{u} =-  \div_{\tilde{\mathcal{A}}}
  u  ,  \\
(\tilde{u}_2, \partial_2\tilde{u}_1)|_{\partial\Omega}    =0 .
 \end{cases}\end{equation*}
Then there exists a solution $(\tilde{u},\varpi)\in \mathcal{H}_{\mm{s}}^2\times  \underline{H}^1$ to the above Stokes problem for given $(\eta,u)$ by Lemma \ref{20222011161567}. Moreover,
\begin{equation}
\label{201705012219xx}
\|\tilde{u}\|_2\lesssim \|\mm{div}_{\tilde{\mathcal{A}}}u\|_1 \lesssim \delta^2 e^{2\Lambda t}.
\end{equation}
It is easy to check that $v^{\mm{d}}:=u^{\mm{d}}-
\tilde{u}\in H_\sigma^2$.

Now, we can apply \eqref{Lambdard} to $ E(v^{\mathrm{d}})$, and get
  \begin{equation*}\begin{aligned}
 E(v^{\mathrm{d}})\leqslant  \Lambda^2{\|v^{\mathrm{d}}\|^2_0}+ \Lambda\mu\| \nabla  v^{\mathrm{d}}\|_0^2  ,
\end{aligned}\end{equation*}
which, together with \eqref{201705012219xx}, immediately implies
\begin{equation*}\begin{aligned}
 E(u^{\mathrm{d}})\leqslant  \Lambda^2{\|\sqrt{\bar{\rho}}u^{\mathrm{d}}\|^2_0} +  \Lambda{\mu}\|\nabla   u^{\mathrm{d}}\|_0^2 +c\delta^3 e^{3\Lambda t}.
\end{aligned}\end{equation*}
Inserting it into \eqref{0314xxdfafdss}, we arrive at
\begin{equation}
\label{201702232221}
\begin{aligned}  \| \sqrt{\bar{\rho}}u_t^\mm{d}\|^2_{0}+2\mu\int_0^t\|\nabla_{\mathcal{A}}
u^{\mathrm{d}}_{\tau}\|_0^2\mm{d}\tau \leqslant {\Lambda^2} \|  \sqrt{\bar{\rho}}u^\mm{d}\|_{0}^2 +{\Lambda} \mu\|\nabla u^{\mathrm{d}}\|_0^2 +  c \delta^3 e^{3\Lambda t}.
\end{aligned}\end{equation}
In addition,
$$
 \int_0^t\|\nabla u_{\tau}^{\mathrm{d}}\|_0^2\mm{d}\tau
\leqslant   \int_0^t(\| \nabla _{\mathcal{A}}u_{\tau}^{\mathrm{d}}\|_0^2+
\|\tilde{\mathcal{A}}\|_2\|u_{\tau}^{\mathrm{d}}\|_1^2)\mm{d}\tau
\leqslant \int_0^t\|\nabla_{\mathcal{A}}u_{\tau}^{\mathrm{d}}\|_0^2\mm{d}\tau+  c \delta^3 e^{3\Lambda t}. $$
Putting the above estimate into \eqref{201702232221} yields that
  \begin{equation} \label{new0311}
\begin{aligned}  \| \sqrt{\bar{\rho}} u_t^\mm{d}\|^2_{0}+2 \mu\int_0^t\| \nabla u^{\mathrm{d}}_{\tau}\|_0^2\mm{d}\tau \leqslant {\Lambda^2} \|\sqrt{\bar{\rho}} u^\mm{d}\|_{0}^2 +  \Lambda\mu  \|\nabla  u^{\mathrm{d}}\|_0^2+  c \delta^3 e^{3\Lambda t}.
\end{aligned}\end{equation}
Then we can further deduce from the above estimate that
\begin{eqnarray}\label{uestimate1n}
\|u^{\mathrm{d}}\|_{1 }^2+\| u_t^{\mathrm{d}}\|^2_0  \lesssim \delta^3e^{3\Lambda t}.
\end{eqnarray}

We turn to derive the error estimate for $\eta^{\mathrm{d}}$.
It follows from \eqref{201702052209}$_1$ that
\begin{equation*}\begin{aligned}
 \frac{\mm{d}}{\mm{d}t}\|\eta^{\mathrm{d}}\|_{1}^2
\lesssim   \| u^{\mathrm{d}}\|_1 \|\eta^{\mathrm{d}}\|_{1}.
\end{aligned}\end{equation*}
Therefore, using \eqref{uestimate1n} and  the initial condition ``$\eta^{\mm{d}}|_{t=0}=\delta^2 \eta^{\mm{r}}$" in  \eqref{201702052209}$_5$,  it follows that
\begin{equation}\begin{aligned}\label{erroresimts}
 \|\eta^{\mathrm{d}}\|_{1 }\lesssim  &  \int_0^t \|u^{\mathrm{d}}\|_{1} \mm{d}\tau
+\delta^2\|\eta^{\mm{r}}\|_1\lesssim  \sqrt{ \delta^3e^{3\Lambda t}}.
\end{aligned}\end{equation}
Noting that $H^1\hookrightarrow W^{1,1}$, then we can derive
 \eqref{ereroe}  from \eqref{uestimate1n} and  \eqref{erroresimts}. Finally,  it is easy to see that
 $$\|(\mathcal{A}_{ik}\partial_k\chi_j
-\partial_i \chi^{\mm{a}}_j)(t)
\|_{{L^1}}\lesssim\|\tilde{\mathcal{A}}_{ik}\partial_k\chi_j
 \|_{{L^1}}+ \| \partial_i\chi_j^{\mm{d}} (t)
\|_{{L^1}} \lesssim  \sqrt{ \delta^3e^{3\Lambda t}},$$
which yields \eqref{2018090119021}.
This completes the proof of Proposition   \ref{2022202101315}. \hfill$\Box$
\end{pf}

Let
 $$   \begin{aligned}
\varpi :=& \min_{ {\chi} =\tilde{\eta}^0, \tilde{u}^0} \left\{\|{\chi}_1\|_{L^1}, \|\partial_1{\chi}_1\|_{L^1}, \|\partial_2{\chi}_1\|_{L^1},\|\chi_2\|_{L^1},\|\bar{\rho}'\chi_2^0\|_{L^1} ,\|\partial_1\chi_2\|_{L^1},\|\partial_2\chi_2\|_{L^1} \right\}.
\end{aligned}$$
Then $\varpi >0$ by \eqref{n05022052}.

Now defining
\begin{equation}\label{201907111842}
\epsilon_{0}=\min \left\{ \left(\frac{ {c}_3^I}{2 c_{5}^I }\right)^2, \frac{c_3^I \delta_0}{c_4^I} ,
\frac{\varpi ^2}{4 ({c}_{5}^I+ |{c}_{4}^I|^{2})^2},1 \right\}>0,
\end{equation}
we claim that
\begin{equation}\label{201907111840}
T^{\delta}=T^{\mm{min}}=
\min \left\{T^{\delta}, T^{*}, T^{* *}\right\}
\neq T^{*}\;\;\mbox{or}\;\;T^{**},
\end{equation}
which can be shown by contradiction as follows.
\begin{enumerate}[(1)]
  \item If $T^{\mm{min}}=T^{**} $, then $T^{**} < T^{\max} $ by \eqref{0n111}.
Noting that $\sqrt{\epsilon_0}\leqslant  {c}_3^I/2{c}_5^I$, we see  that by
\eqref{0501}, \eqref{201912041727}, \eqref{times} and \eqref{ereroe},
\begin{equation*}\begin{aligned}
\|  \eta  (T^{**})\|_{0}&\leqslant \| \eta^\mm{a} (T^{**})\|_{0}+\| \eta^\mm{d} (T^{**})\|_{0}\\
&\leqslant  \delta e^{{\Lambda T^{**}}}(c_3^I+ c_5^I\sqrt{\delta e^{\Lambda  T^{**}}})
\leqslant  \delta e^{{\Lambda  T^{**}}}(c_3^I+ c_5^I\sqrt{\epsilon_0})\\
&\leqslant 3c_3^I \delta e^{\Lambda  T^{**}}/2< 2c_3^I\delta e^{\Lambda T^{**}},
\end{aligned} \end{equation*}
which contradicts to \eqref{0502n111}. Hence $T^{\mm{min}}\neq T^{**}$.
  \item If $T^{\mm{min}}=T^{*}$, then $T^{*}<T^{**}  $. Recalling
$  \epsilon_0\leqslant {c}_3^I  \delta_0/c_4^{I}$, we deduce from \eqref{20191204safda2114}
that for any $t\in I_{T^{\min}}$,
\begin{equation}\nonumber
 \sqrt{\|\eta(t)\|_3^2+\| u(t)\|_2^2  }  \leqslant c_4^I \delta  e^{ \Lambda T^\delta} \leqslant  {c}_3^I  \delta_0<2 {c}_3^I \delta_0,
\end{equation}
which contradicts \eqref{050211}. Hence $T^{\mm{min}}\neq T^{*}$.
\end{enumerate}

Since $T^{\delta} $ satisfies \eqref{201907111840}, the inequalities \eqref{ereroe}--\eqref{2018090119021}   hold to $t=T^\delta$.
Using this fact, \eqref{0501},  \eqref{times}, \eqref{2022202101204}  and the condition $\epsilon_0\leqslant {\varpi ^2}/ 4 ({c}_{5}^I+|{c}_4^I |^{2})^2$, we find  the following instability relations:
for $i$, $j=1$, $2$,
 \begin{align}
\|\mathcal{A}_{ik}\partial_k\chi_j(T^\delta) \|_{{L^1}}\geqslant & \| \partial_i \chi^{\mm{a}}_j(T^\delta) \|_{{L^1}}
- \|\mathcal{A}_{ik}\partial_k\chi_j(T^\delta) -\partial_i\chi^{\mm{a}}_j(T^\delta) \|_{{L^1}}\nonumber  \\
\geqslant  & \delta e^{\Lambda T^\delta }( \|\partial_i\tilde{\chi}^{0}_j\|_{L^1}- {c}_{5}^I\sqrt{\delta e^{\Lambda T^\delta }}) \geqslant \epsilon:= \varpi \epsilon_0 /2 \nonumber
\end{align}
and
\begin{align}
&\|\bar{\rho}-\bar{\rho}(\chi_2(y,T^\delta) +y_2) \|_{{L^1}}\nonumber \\
&\geqslant  \| \bar{\rho}' \chi^{\mm{a}}_2(T^\delta) \|_{{L^1}}-\| \bar{\rho}' \chi^{\mm{d}}_2(T^\delta) \|_{{L^1}}
\nonumber \\
&\quad - \left\|\int_{0}^{\chi_2(y,T^\delta)}\left(\chi_2(y,(T^\delta)) -
z\right)\bar{\rho}''(y_2+z)\mm{d}z \right\|_{{L^1}}\nonumber  \\
&\geqslant   \delta e^{\Lambda T^\delta }( \|\bar{\rho}'\tilde{\chi}^{0}_2\|_{L^1}- {c}_{5}^I\sqrt{\delta e^{\Lambda T^\delta }}-(c_4^I)^2 \delta  e^{ \Lambda T^\delta} ) \geqslant  \varpi \epsilon_0 /2 ,\nonumber
\end{align}
where $\chi=\eta$ or $u$. Similarly, we can also verify that $(\eta,u)$ satisfies the rest instability relations in \eqref{201806012326} by using \eqref{ereroe}.
This completes the proof of  Theorem \ref{thm1}.

\section{Local well-posedness}\label{202202021613}

This section is devoted to the proof of the local well-posedness results in Propositions \ref{202102182115}. To begin with we shall establish the existence of  strong solutions to the following linear initial value problem:
 \begin{equation}
\label{201912060857}
  \begin{cases}
 \bar{\rho}u_t+\nabla_{\ml{A}}q -\mu  \Delta_{\ml{A}}u = f,    \\
\mm{div}_{\mathcal{A}}u=0 ,\\
 u|_{t=0}=u^0 ,\\
 (u_2,\partial_2 u_1)|_{\partial\Omega  }=0 , \end{cases}
  \end{equation}
where $\mu >0$,  $(\eta^0,u^0,w)$ are given,
\begin{align}
\ml{A}=(\nabla \zeta)^{-\mm{T}} \mbox{ and } \zeta=y+\eta^0+\int_0^t w\mm{d}y,
\label{202009151926}
 \end{align}
In what follows, we shall use the following notations.
\begin{align}
&  ^*\!X=\mbox{the dual space of }X , \    <\cdot,\cdot>_{^{*}\!X,X} \mbox{ denotes  {the} dual product},   \nonumber \\
 &  { \mathcal{G}_T:=\{f\in C^0(\overline{I_T}, L^2)~|~ (f,f_t)\in L^2_TH^1 \times L^2_T{^*\!H^{1}}\}, } \nonumber \\
& { \|v\|_{ {\mathcal{U}}_{ T}}:=  }
                            \sqrt{ \| v\|_{C^0(\overline{I_T}, H^2)}^2+\|v_t\|_{C^0(\overline{I_T} , L^2)}^2+ \sum_{0\leqslant j\leqslant 1} \|\partial_t^jv\|_{L^2_TH^{2(1-j)+1} }^2  }, \nonumber  \\
                          & a\lesssim_0 b \mbox{ means } a\leqslant c_0b,\  A\lesssim_{\mm{L}}B\mbox{ means }A\leqslant c^{\mm{L}} B, \nonumber   \end{align}
 where $X$ denotes a Banach space, $c_0$ a generic positive constant at most depending on the domain $\Omega$,   and $c^{\mm{L}}$ a generic positive constant depending on $\mu $, $\lambda$, $m$, $\bar{\rho}$ and $\Omega$, and may vary  {from line to line} (if not stated explicitly).
\begin{pro}
\label{qwepro:0sadfa401nxdxx}
Let $B >0 $, $\delta>0$, $(\eta^0,u^0)\in   \mathcal{H}^3_{\mm{s}}  \times {\mathcal{H}^2_{\mm{s}}}$,  $\mathcal{A}^0=(I+\nabla\eta^0)^{-\mm{T}}$, $w\in  {\mathcal{U}}_{T}$,
$f\in \mathcal{G}_T$, $\mathcal{A}$ and $\zeta$ be defined by \eqref{202009151926}, $\eta=\zeta -y$ and
\begin{align}  {T:=\min\{1,(\delta /B )^4\} }.  \label{201912061028}
\end{align}
Assume that
\begin{align}
& \|\nabla \eta^0\|_{2 } \leqslant \delta ,\ \mm{div}_{\mathcal{A}^0}u^0=0,\   w|_{t=0}=u^0,  \nonumber \\
& \sqrt{\|\nabla w\|_{C^0(\overline{I_T},H^1)}^2 +\| \nabla w\|_{L^2_TH^2}^2
+\|\nabla w_t\|_{L^2(\Omega_T)}}\leqslant B,\label{2019122821231}
\end{align}
then there is a sufficiently small constant $\delta^{\mm{L}}  \in (0,1]$ independent of  $\mu $, such that for any
$\delta\leqslant\delta^{\mm{L}} $, there exists a unique local strong solution
 $(u,q)\in {\mathcal{U}}_{T}\times (C^0(\overline{I_T},\underline{H})\cap L^2_T{H}^2)$ to the initial boundary value  problem \eqref{201912060857}.
Moreover, the solution enjoys the following {{estimates}}:
\begin{align}
& 1\leqslant 2\det\zeta\leqslant {3},\quad \| \nabla \eta \|_{2} \leqslant  2\delta, \label{2022202121142} \\
& \|u\|_{\mathcal{U}_{T}}+  \|q\|_{C^0(\overline{I_T} ,H^1)}+\|q\|_{L^2_TH^2}  \lesssim_{\mm{L}} \sqrt{\mathfrak{B}(u^0,f )},
\label{202011102145}
\end{align}
where
\begin{align}
&\mathfrak{B}(u^0,f) :=   \|u^0\|_2^2+ \| u^0\|_2^4 +  \|f\|_{C^0(\overline{I_T},L^2)}^2+ \| f\|_{L^2_TH^1}^2 \nonumber \\
&\qquad \qquad \quad +\|f_t\|_{L^2_{T}{^*\!H^{1}}}^2+\|\nabla w\|_{C^0(\overline{I_T},H^1)}(1+ \|  u^0 \|_1 )( \|  u^0 \|_2^2
+ \|f\|_{L^2(\Omega_T)}^2),\nonumber \end{align}
Moreover,  for a.e. $t\in I_T$,
\begin{align}
&\frac{1}{2}\frac{\mm{d}}{\mm{d}t}\|\sqrt{ \bar{\rho}J } \Psi\|^2_0\nonumber \\
 & =    \int
( ( J(  \mu    \mathcal{A}_{il}\partial_l( \mathcal{A}_{ik} \partial_k  u) - \bar{\rho} w\cdot\nabla_{\mathcal{A}}    u   + f) \cdot w\cdot\nabla_{\mathcal{A}} \Psi + |\Psi|^2 w\cdot\nabla_{\mathcal{A}}\bar{\rho}/2)
\nonumber \\
&\quad- \mu  \partial_t (  J \mathcal{A}_{il}\mathcal{A}_{ik} \partial_k  u)\cdot \partial_l \Psi -   \partial_t ( J\bar{\rho} w\cdot\nabla_{\mathcal{A}}    u ) \cdot \Psi+f \Psi J _t)\mm{d}y+ <f_t, \Psi J>_{^{*}\!H^1,H^1}   , \label{eq0510}
\end{align}
where $\Psi:=u_t- {w} \cdot \nabla_{\mathcal{A}} u $.
\end{pro}
\begin{pf}
We shall break up the proof into three steps.

(1) \emph{Existence of local strong solutions}

Recalling that $\|\nabla \eta^0\|_{2}\leqslant \delta$, {the definition \eqref{201912061028} } and the relation
\begin{align} \label{202010011513}
\eta=\eta^0+\int_0^tw\mm{d}\tau, \end{align}
we make use of the regularity of $w$, \eqref{201912061028} and \eqref{2019122821231} to find that $\eta\in C^0(\overline{I_T}, \mathcal{H}^3_{\mm{s}})$, $\eta_t=w$ and
\begin{equation}
\label{201912061030}
\| \nabla \eta(t)\|_{2 } \leqslant\delta +\sqrt{t} \|\nabla w\|_{L^2_TH^2 }\leqslant 2\delta \mbox{ for all }t\in \overline{I_T}.
\nonumber
\end{equation}
Obviously
\begin{align}
\label{201sdfa912061030}
\| \nabla \eta(t)\|_{L^\infty} \lesssim_0 \| \nabla \eta(t)\|_{2} \lesssim_0 \delta\;\;\;\mbox{ for any }t\in \overline{I_T}.
\end{align}
Thanks to the estimates \eqref{201sdfa912061030} and \eqref{esmmdforinfty}, we have for sufficiently small $\delta$ that $1\leqslant 2J\leqslant {3}$,
where and in what follows, $J:=\det\nabla \zeta$. Therefore, $\mathcal{A}$ makes sense and is given by the following formula:
\begin{align}
 {\mathcal{A}}=J^{-1} \begin{pmatrix}
          1+\partial_2\eta_2&  - \partial_1\eta_2\\
                - \partial_2\eta_1&   1+\partial_1\eta_1
                 \end{pmatrix} . \nonumber
\end{align}
\emph{We remark that the smallness of $\delta$ (independent of $\mu $) will be often used in the derivation of some estimates
and conclusions later, and we shall omit to mention it for the sake of simplicity.}

Inspired by the proof in \cite[Theorem 4.3]{GYTILW1}, we next solve the linear problem \eqref{201912060857} by applying the Galerkin method.
Let $\{\varphi^i\}_{i=1}^\infty\subset  {\mathcal{H}^\infty_{\sigma}}$ be a countable orthogonal basis in $ {{H}^1_{\sigma}}$ by Lemma \ref{20222011231648}.
For each $i\geqslant 1$ we define $\psi^i=\psi^i(t):=\nabla\zeta\varphi^i$. Let $\mathfrak{H}(t)= \{v\in \mathcal{H}^2_{\mm{s}}~|~\mm{div}_{\mathcal{A}}v=0 \}$ and $\mathfrak{M}(t)= \{v\in  {H}^1_{\mm{s}}~|~\mm{div}_{\mathcal{A}}v=0  \}$.
Then $\psi^i(t)\in \mathfrak{H}(t)$ and $\{\psi^i(t)\}_{i=1}^\infty$ is a basis of $\mathfrak{M}(t)$ for each $t\in \overline{I_T}$. Moreover,
\begin{equation}
\label{201912102051}
\psi_t^i =R \psi^i,
\end{equation}
where $R:=\nabla w\mathcal{A}^{\mm{T}}\in C^0(\overline{I_T}, H^1 )\cap L^2_TH^2$. Obviously, by \eqref{2012fsa2450}, $R_t=\nabla w_t\mathcal{A}^{\mm{T}}+ \nabla w \mathcal{A}^{\mm{T}}_t\in L^2_TL^2$.

For any integer $n\geqslant 1$, we define the finite-dimensional space
$\mathfrak{H}^n(t):=\mm{span}\{\psi^1,\ldots, \psi ^n\}\subset \mathfrak{H}(t)$,
and write $\mathcal{P}^n(t):\mathfrak{H}(t)\to \mathfrak{H}^n(t) $ for the $\mathfrak{H}$ orthogonal projection
onto $ \mathfrak{H}^n(t)$. Clearly, for each $v\in \mathfrak{H}(t)$,
$\mathcal{P}^n(t) v \rightarrow v $ in $\mathcal{H}^2_{\mm{s}}$ as $n\rightarrow\infty$ and $\|\mathcal{P}^n(t) v \|_2\leqslant \|v\|_2$.

Now, we define an approximate solution
$$u^n(t)=a_j^n(t) \psi^j \mbox{ with }a_j^n:\ \overline{I_T}\rightarrow\mathbb{R}\;\mbox{ for }j=1,\ldots, n,$$
where $n\geqslant 1$ is given.
We want to choose the coefficients $a_j^n$, so that for any $1\leqslant i\leqslant n$,
\begin{equation}\label{appweaksolux}
 \int \bar{\rho} u ^n_t\cdot \psi^i \mm{d} y +\mu \int
\nabla_{\mathcal{A}}  u ^n:\nabla_{\mathcal{A}} \psi^i \mm{d} y   =    \int  f  \cdot \psi^i \mm{d} y
\end{equation}
with initial data
$ u^n(0)=\mathcal{P}^n u _0\in \mathfrak{H}^n$.

Let
$$  \begin{aligned}
& X=(a_i^n)_{n\times 1},\ \mathfrak{N}=  \left(\int f \cdot \psi^i \mm{d} y\right)_{n\times 1},\
\mathfrak{C}^1=\left(  \int \bar{\rho}\psi^i\cdot \psi^j \mm{d} y\right)_{n\times n}, \\
 & \mathfrak{C}^2=\left( \int R \psi^i\cdot \psi^j \mm{d} y+\mu   \int
\nabla_{\mathcal{A}}\psi^i:\nabla_{\mathcal{A}}\psi^j\mm{d} y \right)_{n\times n}.
\end{aligned}$$
Recalling the regularity of $(\eta, \psi^i,R,f)$,  Lemma \ref{psdf221} and  \eqref{20122450}, we easily verify that
\begin{align}
\label{202001200222}
\mathfrak{C}^1 \in C^{1,1/2}(\overline{I_T}),\ \mathfrak{C}^2\in C^{0,1/2}(\overline{I_T}),\ \mathfrak{N} \in C^{0,1/2}(\overline{I_T})\mbox{ and } (\mathfrak{C}^2_t, \mathfrak{N}_t)\in L^2(I_T).
\end{align}

 Noting that $\mathfrak{C}^1$ is invertible, we can rewrite \eqref{appweaksolux} as follows.
 \begin{equation}  \label{appweaksolu}
 X_{t}+(\mathfrak{C}^1)^{-1}( \mathfrak{C}^2 X- \mathfrak{N})=0
\end{equation}
with initial data
$$ X|_{t=0}=\left(\int \mathcal{P}^n u^0 \cdot \psi^i\mm{d}y \right)_{n\times 1}  ,$$
where $(\mathfrak{C}^1)^{-1}$ denotes the inverse matrix of $\mathfrak{C}^1$.
By virtue of the well-posedness theory of ODEs (see \cite[Section 6 in Chapter II]{WWODE148}), the equation \eqref{appweaksolu} has
exactly one solution $X\in C^{1}(\overline{I_T})$. Moreover,  by \eqref{202001200222} and \eqref{appweaksolu}, we have
\begin{align}\ddot{a}_j^n(t)\in L^2({I_T}).
\label{20222021101801}
  \end{align}
  Thus, we have established the existence of the approximate solution
$u^n(t)=a_j^n(t)\psi^j\in C^0(\overline{I_T}, \mathcal{H}^3_{\mm{s}})$.  Next, we derive uniform-in-$n$ estimates for $u^n$.

Due to \eqref{201912061030}, we easily get from \eqref{appweaksolux} with $u^n$ in place of $\psi$ that
for sufficiently small $\delta$,
\begin{equation}   \label{201912112038}
\frac{\mm{d}}{\mm{d}t} \| \sqrt{\bar{\rho}}u ^n \|_0^2   + c^{\mm{L}} \|  u^n\|_1^2\lesssim_{\mm{L}} \|f\|_0^2.
\end{equation}

By \eqref{201912102051},
\begin{equation}
\label{201912122222}
u^n_t- R  u^n= \dot{a}_j^n  \psi^j\in C^0(\overline{I_T},  {H}^1_{\mm{s}}) \cap L^2_TH^2.
\end{equation}
Obviously, we can replace $\psi$ by $\dot{a}_j^n\psi^j$ in \eqref{appweaksolux} and use \eqref{201912122222} to deduce that
\begin{align}
& \| \sqrt{\bar{\rho}} u ^n_t\|_0^2 +\mu \int
\nabla_{\mathcal{A}}  u ^n:\nabla_{\mathcal{A}} u ^n_t \mm{d}y  \nonumber  \\
&   =  \int   {\bar{\rho}}u_t^n \cdot(R u^n)\mm{d}y +\mu  \int \nabla_{\mathcal{{A}}}u^n:\nabla_{\mathcal{A}} (R u^n)\mm{d}y
 + \int  f   \cdot   (u_t^n-R u^n) \mm{d}y. \label{appweakssafolux}
\end{align}

Thanks to  \eqref{201912061030}, \eqref{fgessfdims}, \eqref{20122450} and   \eqref{2012fsa2450}, we can further obtain from \eqref{appweakssafolux} that
\begin{align}
& {\mu  } \frac{\mm{d}}{\mm{d}t}   \|\nabla_{\mathcal{A}}  u ^n\|_0^2  +  \| \sqrt{\bar{\rho}} u ^n_t\|_0^2\nonumber \\
 &\lesssim_{\mm{L}}\|\nabla u^n\|_0(\|R \|_{L^\infty}\|\nabla u^n\|_0
+\|\mathcal{A}_t\|_{L^\infty} \|\nabla u^n\|_0 )+\int | \nabla u^n| |\nabla R ||u^n|\mm{d}y + \|(f, R u^n) \|_{0}^2  \nonumber \\
&  \lesssim_{\mm{L}}   \|\nabla w\|_{2}  (\|u^n\|_0^2+\|\nabla_{\mathcal{A}}u^n\|_0^2)
+\|\nabla w\|_{1}\|\nabla w\|_{2} \|u^n\|_0^2+\|f\|_0^2,  \label{appwfolux}
\end{align}
where $ \|\nabla_{\mathcal{A}}  u ^n\|_0^2 \in AC(\overline{I_T})$.

With the help of Gronwall's lemma, \eqref{201912061028} and \eqref{201912061030}, we infer from \eqref{201912112038} and \eqref{appwfolux}
  that for any $t\in\overline{I_T}$,
\begin{align}
\| u ^n
\|_1^2+\int_0^t\|u_\tau\|_0^2\mm{d}\tau \lesssim_{\mm{L}}&  \left( \|\mathcal{P}^nu^0\|_1^2+\int_0^t\|f\|_0^2\mm{d}\tau\right)
 e^{\int_0^t c (\|\nabla w\|_2+\|\nabla w\|_{1}\|\nabla w\|_{ 2})\mm{d}\tau}\nonumber \\
\lesssim_{\mm{L}}& \|u^0\|_2^2+ \|f\|_{L^2(\Omega_t )}^2 .
\label{201912241359}
\end{align}

By the regularity of $(f, \psi^i)$  and \eqref{20122450}, we have
\begin{align}
\frac{\mm{d}}{\mm{d}t}\int f(t)\cdot \psi^i(t)\mm{d}y=<f_\tau,\psi^i>_{^{*}\!H^1,{H}^1}+ \int f\cdot \psi^i_t\mm{d}y \mbox{ for a.e. } t\in I_T.
\label{202011091820}
\end{align}
 Recalling \eqref{20222021101801} and \eqref{201912122222}, we see that $u_{tt}^n \in L^2(\Omega_T)$ makes sense. So, with the help of \eqref{201912102051} and \eqref{202011091820},
we get from \eqref{appweaksolux} that
 \begin{align}
&  \int  {\bar{\rho}} u ^n_{tt}\cdot \psi^i \mm{d} y+\mu \int \nabla_{\mathcal{A}}  u ^n_t:\nabla_{\mathcal{A}} \psi^i \mm{d} y   \nonumber \\
& = <  f_t \cdot \psi^i>_{{^*H^{1}},{H}^1  } + \int  (f- \bar{\rho}u ^n_{t} )\cdot (R \psi^i)  \mm{d} y \nonumber \\
&\quad -\mu \int (\nabla_{\mathcal{A}_t} u^n:\nabla_{\mathcal{A}} \psi^i
+ \nabla_{\mathcal{A}} u^n:(\nabla_{\mathcal{A}_t} \psi^i
+ \nabla_{\mathcal{A}}(R\psi^i ))\mm{d} y \mbox{  a.e. in }I_T. \label{201912102242}
\end{align}

Noting that, by \eqref{20222012312450},
$$ \begin{aligned}
& \frac{1}{2}\| \sqrt{\bar{\rho}} u ^n_t\|_0^2 - \int   {\bar{\rho}} u_{t}^n\cdot (R u^n)\mm{d}y- \left.\left(\frac{1}{2}\| \sqrt{\bar{\rho}} u ^n_t\|_0^2
 - \int    {\bar{\rho}} u_{t}^n\cdot ( R u^n)\mm{d}y\right)\right|_{t=0} \\
 & = \int_0^t \left(\int   {\bar{\rho}} u ^n_{\tau\tau}\cdot  ( u^n_{\tau}-Ru^n) \mm{d} y
 -\int   {\bar{\rho}}  u ^n_{\tau}\cdot  (Ru^n)_{\tau} \mm{d} y\right)\mm{d}\tau
 \end{aligned} $$
and
$$  \begin{aligned}
&\int f(\tau)\cdot (Ru^n)(\tau) \mm{d}y\bigg|_{\tau=0}^{\tau=t}  =\int_0^t\left( <f_\tau,  R u^n>_{^*\!H^1,{H}^1  }
+  \int f\cdot( R u^n )_\tau\mm{d} y\right)\mm{d}\tau ,
 \end{aligned} $$
we utilize \eqref{201912122222} and the above two identities to infer from \eqref{201912102242} with $\psi^i$ replaced
by $(u^n_t-Ru^n)$ that
\begin{align}
&\frac{1}{2 }\| \sqrt{\bar{\rho}} u ^n_t\|_0^2  - \int    {\bar{\rho}} u_{t}^n\cdot (R u^n)\mm{d}y+ \int f\cdot (Ru^n) \mm{d}y
+\mu \int_0^t\|\nabla_{\mathcal{A}}  u ^n_\tau\|_0^2 \mm{d}\tau\nonumber \\
& = \left. \left(\frac{1}{2}\| \sqrt{\bar{\rho}} u ^n_t\|_0^2 - \int   {\bar{\rho}} u_{t}^n\cdot ( Ru^n)\mm{d}y
+\int f \cdot (Ru^n) \mm{d}y\right)\right|_{t=0}+I^L,
\label{201912241249}
\end{align}
where
$$    \begin{aligned}
I^L:=&\int_0^t\bigg( <f_{\tau},  u_{\tau}^n >_{^*\!H^{1},{H}^1  }
+\int \Big( f\cdot (2 Ru^n_{\tau} +R_{\tau}u^n -R^2 u^n)-\bar{\rho}u_\tau^n \cdot (R (u_\tau^n\nonumber \\
&\quad - R u^n))\Big)\mm{d}y -  \int   {\bar{\rho}} u_{\tau}^n\cdot (R u^n)_{\tau} \mm{d}y
-\mu \int\big(  \nabla_{\mathcal{A}} u^n: (\nabla_{\mathcal{A}_{\tau}} (u_{\tau}^n-R u^n) \nonumber \\
&\quad +\nabla_{\mathcal{A}}(R (u_{\tau}^n -R u^n))\big) +\nabla_{\mathcal{A}_{\tau}} u^n:\nabla_{\mathcal{A}} (u_{\tau}^n-R u^n)
- \nabla_{\mathcal{A}}  u^n_\tau :\nabla_{\mathcal{A}} (R u^n)) \mm{d} y\bigg)\mm{d}\tau .
\end{aligned}
$$

Keeping in mind that $H^1\hookrightarrow L^6$ and
\begin{align}
  \|\nabla w\|_0\leqslant \int_0^t\|\nabla w_{\tau}\|_0\mm{d}\tau+\|\nabla u^0\|_0, \label{202009121026}
\end{align}
we get from \eqref{201912241249} that
 \begin{align}
& \| \sqrt{\bar{\rho}} u ^n_t\|_0^2  +\int_0^t   \mu \| \nabla u ^n_{\tau}\|_0^2 \mm{d}\tau\nonumber \\
& \lesssim_{\mm{L}}\|\nabla w\|_{0}\|\nabla w\|_1\|u^n\|_1^2+\|u^0\|_2^4 +\|f\|_{C^0(\overline{I_T},L^2)}^2+\|u^n_t|_{t=0}\|_0^2+ I^{\mm{L}}\nonumber\\
& \lesssim_{\mm{L}}  \|\nabla w\|_1(1+ \|u^0 \|_1 )\left(\|u^0 \|_2^2 +\|f\|_{L^2(\Omega_t )}^2\right)\nonumber \\
&\quad+ \|u^0\|_2^4+ \|f\|_{C^0(\overline{I_T},L^2)}^2
+\|u ^n_t|_{t=0}\|_0^2+ I^L,
\label{201912242022}
\end{align}
where we have used  \eqref{201912061028}, \eqref{201912241359} and \eqref{202009121026}
in the second inequality. Below, we shall bound the the last two terms in \eqref{201912242022}.

Replacing $\psi^i$ by $(u_t^n-Ru^n)$ in \eqref{appweaksolux}, we see that
\begin{align}
 \|  \sqrt{\bar{\rho}} u^n_t \|_0^2 = &  \int f \cdot (u ^n_t-R u^n) \mm{d} y
 +\mu \int \Delta_{\mathcal{A}}u^n\cdot(u_t^n-R u^n)\mm{d}y + \int  {\bar{\rho}} u_t^n\cdot (R u^n)\mm{d}y  ,
\end{align}
which implies
$$ \| \sqrt{\bar{\rho}} u^n_t \|_0^2\lesssim_{\mm{L}}\| f \|_0^2+\|  u^n\|_2^2 +\|\nabla w\|_1^2\|u^n\|_1^2 . $$
In particular, \begin{equation}
\label{201912112039}
\| \sqrt{\bar{\rho}} u ^n_t|_{t=0} \|_0^2\lesssim_{\mm{L}}  \|u^0\|_2^2+ \|u^0\|_2^4+\|f^0\|_0^2 .
\end{equation}

In addition, the last term on the right-hand side of \eqref{201912242022} can be estimated as follows:
\begin{align}
I^L\lesssim_{\mm{L}} & \int_0^t\big(\|f_{\tau}\|_{^*H^1} \|u^n_{\tau}\|_{1}
+\|u^n\|_1\big( \| f\|_0 \|\nabla w\|_1 \sqrt{\|\nabla  w\|_1\|\nabla  w\|_{2}}  + \sqrt{\| f\|_0\|f\|_1}\|\nabla w_{\tau}\|_0\big) \nonumber \\
&\quad+ \|f\|_1\| \nabla w\|_1\|u_{\tau}^n\|_0+ \| u^n  \|_1^2  {\|\nabla  w\|_1\|\nabla  w\|_{2}}   \nonumber \\
 &\quad+ \| u^n_\tau \|_0(\|\nabla w\|_2\| u^n_\tau \|_0 +   \|\nabla w\|_1^2 \|u^n \|_1)\nonumber \\
&\quad
+ \|\nabla w_\tau\|_0 \|u^n \|_1\sqrt{\| u^n_\tau \|_0 \| u^n_\tau \|_1} +\|u^n \|_1\| u^n_\tau \|_1  \sqrt{\|\nabla  w\|_1\|\nabla  w\|_{2}}
\big)\mm{d}\tau \nonumber \\
 \lesssim_{\mm{L}}  & \|u^0\|_2^2+\|f\|_{L^\infty_tL^2}^2 +\|  f\|_{L^2_tH^1}^2
+  \int_0^t \|f_{\tau}\|_{^*\!H^{1}}  \| u_{\tau}^n\|_1 \mm{d}\tau\nonumber \\
& + \left(  \|u^0\|_2   + \| f\|_{L^2_tH^1} \right)( \|u_{\tau}^n\|_{{L^\infty_tL^2} } + \| \nabla u_\tau^n\|_{L^2(\Omega_t)} )
 +\delta  \|u_{\tau}^n\|_{{L^\infty_tL^2} }^2,\label{202001221640}
\end{align}
where we have used \eqref{fgessfdims} and the embedding $H^1\hookrightarrow L^6$  in the first inequality,  and \eqref{201912061028} in the second inequality.

Substituting \eqref{201912112039} and \eqref{202001221640} into \eqref{201912242022}, and applying Young's inequality, we arrive at
\begin{align}
&  \| u ^n_t \|_{L^\infty_TL^2}^2 +   \| u ^n_{t}\|_{L^2_TH^1}^2  \lesssim_{\mm{L}}  {\mathfrak{B}} (u^0,f).\label{appwasdfasfsadfeaksolux}
\end{align}
Summing up \eqref{201912241359} and \eqref{appwasdfasfsadfeaksolux}, we conclude
\begin{align}
 \| u^n\|_{L^\infty_TH^1}^2+\|u^n_t \|_{L^\infty_TL^2}^2+ \| u _t^n \|_{L^2_TH^1}^2 \lesssim_{\mm{L}}   {\mathfrak{B}} (u^0,f). \label{sumestimeat}
\end{align}

In view of \eqref{sumestimeat}, the Banach--Alaoglu and Arzel\`{a}--Ascoli theorems, up to the extraction of a subsequence
(still labelled by $u^n$), we have, as $n\to\infty$, that
$$\begin{aligned}
&(u^n,u^n_t)\rightharpoonup(u,u_t)  \mbox{ weakly-$*$ in }L^\infty_T {H}_{\mm{s}}^1\times L^\infty_T L^2 ,\\
& u^n_t \rightharpoonup  u_t  \mbox{ weakly  in } L^2_TH_{\mm{s}}^1,\\
& u^n\rightarrow u  \mbox{ strongly in }C^0(\overline{I_T},L^2),\\
&  \mm{div}_{\mathcal{A}} u =0   \mbox{ a.e. in }\Omega_T  \mbox{ and } u(0)=u^0,
\end{aligned}$$
where $u$ and $u_t$ are measurable functions defined on $\Omega_T$.
Moreover,
\begin{align}
&  \|u\|_{L^\infty_TH^1} +\|u_t \|_{L^\infty_TL^2} + \|u _t\|_{L^2_TH^1} \lesssim_{\mm{L}}\sqrt{ {\mathfrak{B}} (u^0,f)} .  \label{2019121asdfa12157}
\end{align}
Therefore, we can take to the limit in \eqref{appweaksolux} as $n\to\infty$, and obtain that there exists a zero-measurable set $\mathfrak{Z}$ such that, for any $t\in I_T\setminus\mathfrak{Z}$,
\begin{equation}\label{0507n1}
  \int \bar{\rho} u _t\cdot\zeta \mm{d} y +\mu \int\nabla_{\mathcal{A}} u:\nabla_{\mathcal{A}} \zeta\mm{d} y =\int f\cdot\zeta\mm{d} y,\ \forall\,\zeta\in  \mathfrak{H}(t) .
\end{equation}

Now, we begin to show spatial regularity of $u$. Let us further assume that $\delta\in (0,\gamma) $ is so small that
$\eta$ satisfies \eqref{20210301715x} and \eqref{20210301715} by virtue of Lemma  \ref{pro:1221}.
Denoting $F:=f- \bar{\rho}u_t$, $\tilde{F}:=F(\zeta^{-1},t)$ and $\tilde{J}:=J(\zeta^{-1},t)$, we see
that $\tilde{F}$ has the same regularity as that of $F$ by \eqref{2022104101908}, i.e.,
\begin{align}
 \label{20218181407} \| \tilde{F} \|_{L^\infty_TL^2}+  \|  \tilde{F} \|_{L^2_TH^1}  <\infty.\end{align}
Moreover, \begin{align}
\int F(y,t)\mm{d}y= \int \tilde{F}\tilde{J}^{-1}\mm{d}x . \nonumber  \end{align}

Applying the regularity theory of the Stokes problem with Naiver boundary condition, we see that there is a unique strong solution
 $\alpha \in  L^\infty_T \mathcal{ H}^2_\sigma\cap  L^2_T {H}^3$ with a unique associated function
 $P \in L^\infty_T\underline{H}^1 \cap  L^2_TH^2$, such that
\begin{equation}
\nabla  P -\mu  \Delta \alpha =\tilde{F}\mbox{ and }  (\alpha_1)_\Omega=0 .
\end{equation}
 Let $\varpi=\alpha(\zeta,t)$ and $q=P(\zeta,t)- (P(\zeta,t))_{\Omega} $, then, by \eqref{2021sfa04031901} and the regularity $\eta\in C^0(\overline{I_T}, \mathcal{H}^3_{\mm{s}})$, $(\varpi,q)\in (L^\infty_TH^2_{\mm{s}} \cap  L^2_TH^3)\times (L^\infty_T\underline{H}^1\cap  L^2_T {H}^2)$  satisfies the following boundary value problem:   for a.e. $t\in I_T$,
\begin{equation}
\label{appsdfstokes}
\begin{cases}
\nabla_{\mathcal{A}}  {q} -\mu  \Delta_{\mathcal{A}}\varpi= {F},  \\
\mm{div}_{\mathcal{A}} \varpi =0,  \\
(\varpi_2,\partial_2 \varpi_1)=0.
\end{cases}
\end{equation}
Recalling that   $\{\psi^i(t)\}_{i=1}^\infty\subset \mathfrak{H}(t)$ is a basis of $\mathfrak{M}(t)$, thus the identity \eqref{0507n1} also holds for any $\zeta \in \mathfrak{M}(t)$.
This fact, together with \eqref{appsdfstokes}, implies
\begin{align}
\label{20222020062007}
\nabla u=\nabla\varpi.
\end{align}

Exploiting Lemma \ref{20222011161567}, we easily derive from \eqref{appsdfstokes}  that, for sufficiently small $\delta$,
\begin{align}
 \| \varpi\|_{2+i} +  \|  q\|_{1+i}\lesssim_0
  \|(  u_t, f)\|_{i}\mbox{ for }i=0,\ 1.\label{20201fas092203}
\end{align}
So, it follows from  \eqref{2019121asdfa12157}, \eqref{20222020062007}  and \eqref{20201fas092203} that
\begin{align}
 \|  (u,u_t,q) \|_{L^\infty(I_T,H^{2 }\times L^2\times H^1)}  + \|  (u, u_t,q) \|_{L^2(I_T,H^3\times H^1\times H^2)}
\lesssim_{\mm{L}}  \sqrt{\mathfrak{B}(u^0,f) }.\label{201912221611}
 \end{align}
This completes the existence of local strong solutions. Moreover, a strong solution, which enjoys the regularity of $(\eta,u)$
constructed above, is obviously unique.

(2) \emph{Strong continuity of $(u,u_t,q)$.}

 Since $u\in L^2_TH^3$ and $ u_t\in L^2_T{H}^1$, $u\in C^0(\overline{I_T}, H^2)$. By Lemma \ref{lem:08181945}, for each given $t\in \overline{I_T}$, there exists a unique weal solution $ \tilde{q}\in \underline{H^1} $ such that
\begin{align}
\int \bar{\rho}^{-1}  \nabla_{\mathcal{A}}   \tilde{q} \cdot \nabla_{\mathcal{A}} \vartheta\mm{d}y =\int  \bar{\rho}^{-1}(f+\nu \Delta_{\mathcal{A}}u)  \cdot \nabla_{\mathcal{A}}\vartheta\mm{d}y+\int \mathcal{A}_t^{\mm{T}} u\cdot \nabla \vartheta\mm{d}y
\label{2022202101551}
\end{align}
for any  $\vartheta\in H^1$,
and
\begin{align}
\sup_{t\in \overline{I_T}}\| \tilde{q}\|_1\lesssim  \|(\bar{\rho}^{-1}\mathcal{A}^{\mm{T}}(f +\nu \Delta_{\mathcal{A}}u)+\mathcal{A}_t^{\mm{T}} u)\|_{C^0({\overline{I_T}},L^2)}<\infty.\label{2001551}
\end{align} Moreover, it is easy to check that $q \in C^0(\overline{I}_T, \underline{H^1})$ by \eqref{2022202101551} and \eqref{2001551}.

Multiplying \eqref{appsdfstokes}$_1$ by $\bar{\rho}^{-1}    \nabla_A \vartheta$ in $L^2$ and using the integral by parts and \eqref{20222020062007}, we get, for a.e. $t\in I_T$,
\begin{align}
\int \bar{\rho}^{-1}  \nabla_{\mathcal{A}}  {q} \cdot \nabla_{\mathcal{A}} \vartheta\mm{d}y =\int  \bar{\rho}^{-1}(f+\mu \Delta_{\mathcal{A}}u)  \cdot \nabla_{\mathcal{A}} \vartheta\mm{d}y+\int \mathcal{A}_t^{\mm{T}} u\cdot \nabla\vartheta\mm{d}y .
\end{align}
We immediately see $q=\tilde{q}\in C^0(\overline{I}_T, \underline{H^1})$ from \eqref{2022202101551} and the above identity for sufficiently small $\delta$.

 Finally, we can further derive $u_t \in C^0(\overline{I_T}, L^2)$ from \eqref{appsdfstokes}$_1$. Hence, $u\in \mathcal{U}_{T}$.
Thanks to the strong continuity of $(u,u_t,q)$,  we immediately
get \eqref{202011102145} from \eqref{201912221611}.

(3) \emph{Verification of the identity \eqref{eq0510}.}

  For any given $\varphi\in {H}^1_{\mm{s}}$, let $\psi =\varphi(\zeta(y,t)) $. Noting $J_t=J\mm{div}_{\mathcal{A}}w$, $\partial_j(J\mathcal{A}_{ij} )=0$ and  $\nabla \varphi|_{x=\zeta(y,t)}=\nabla_{\mathcal{A}}\psi$,  {we deduce}
from \eqref{appsdfstokes} and \eqref{20222020062007} that for any $\phi\in C_0^\infty(I_T)$,
\begin{align} &-\int_0^t \phi_t\int  \bar{\rho} u \cdot \psi J\mm{d} y \mm{d}\tau
 \nonumber \\&= \int_0^t\phi \int ( f +\mu  \Delta_{\mathcal{A}}  u  - \nabla_{\mathcal{A}}  q   -  \bar{\rho}w\cdot \nabla_{\mathcal{A}}u-w\cdot\nabla_{\mathcal{A}}\bar{\rho}u)\cdot\psi J  \mm{d} y
 \label{20222020917104} \mm{d}\tau .\end{align}

Let ${\rho}=\bar{\rho}(\zeta^{-1}(x,t))$ and ${v}=u(\zeta^{-1}(x,t),t)$. Using  Lemma \ref{20021032019018},  we can check that
\begin{align}
&\rho \in C^0(\overline{I_T},H^2),   \ v\in C^0(\overline{I_T},\mathcal{H}^2_{\mm{s}})\cap L^2_T H^3_\sigma ,\label{2sad0131}\\
&\rho_t  =- ({w} \cdot \nabla_{\mathcal{A}}\bar{\rho} )|_{y=\zeta^{-1}(x,t)}\in L^\infty_TH^1,\label{20222020881255}\\
&v_t  =(u_t- {w} \cdot \nabla_{\mathcal{A}} u )|_{y=\zeta^{-1}(x,t)}\in L^\infty_TL^2 \cap L^2_T H^1_\sigma.\label{20131}
\end{align}

Thanks to the regularity \eqref{2sad0131}, we derive
 from \eqref{20222020917104} that
\begin{align}
&  -\int_0^t \phi_t\int  {\rho}  v \cdot \varphi   \mm{d} x \mm{d}\tau\nonumber \\
& =\int_0^t\phi \int  (  f+\mu  \Delta_{\mathcal{A}}  u -    \nabla_{\mathcal{A}} q
-\bar{\rho} w\cdot\nabla_{\mathcal{A}}u-w\cdot\nabla_{\mathcal{A}}\bar{\rho}u )\cdot  \psi   J  \mm{d} y \mm{d}\tau.
\label{appweakdfsafssfaolux}
\end{align}
In addition, by \eqref{20222020881255} and \eqref{20131}, we have
\begin{align}
-\int_0^t \phi_t\int  \rho v \cdot \varphi \mm{d} y \mm{d}\tau=  &\int_0^t \phi \int (\rho_t v+ \rho v_t ) \cdot \varphi \mm{d} y \mm{d}\tau  \nonumber \\
=& \int_0^t \phi \int  \rho v_t   \cdot \varphi \mm{d} y \mm{d}\tau- \int_0^t \phi \int w\cdot\nabla_{\mathcal{A}} \bar{\rho} u   \cdot \psi J \mm{d} y \mm{d}\tau. \nonumber
\end{align}
Inserting the above identity into \eqref{appweakdfsafssfaolux} yields
\begin{align}
&  \int_0^t \phi \int  {\rho}  v_t \cdot \varphi   \mm{d} x \mm{d}\tau\nonumber \\
& =\int_0^t\phi \int  (  f+\mu  \Delta_{\mathcal{A}}  u -    \nabla_{\mathcal{A}} q
-\bar{\rho} w\cdot\nabla_{\mathcal{A}}u  )\cdot  \psi   J  \mm{d} y \mm{d}\tau.\nonumber
\end{align}

Now let us further assume $\varphi\in H^1_\sigma$, then $\mm{div}_{\mathcal{A}}\psi=0$ and $\psi_t|_{x=\zeta}= w\cdot \nabla_{\mathcal{A}}\psi$. The above identity further  implies
\begin{align}
 \frac{\mm{d}}{\mm{d}t}\int  \rho v_t \cdot \varphi  \mm{d} y =& \int
(J((\mu\mathcal{A}_{il} \partial_l (\mathcal{A}_{ik} \partial_k  u)- \bar{\rho} w\cdot\nabla_{\mathcal{A}}    u+f   )w\cdot \nabla_{\mathcal{A}} \psi  ) \nonumber\\
&  -\mu \partial_t (  J\mathcal{A}_{il} \mathcal{A}_{ik} \partial_k  u)\cdot \partial_l \psi    -   \partial_t ( J\bar{\rho} w\cdot\nabla_{\mathcal{A}}    u ) \cdot \psi  \nonumber \\
  &  + f \psi J _t)\mm{d}y  +<f_t,\psi J>_{^{*}\!H^{1},H^1} =:<\chi,\varphi>_{^{*}\!H^1_\sigma,H_\sigma^1}. \label{appwasfdaeakssfaolux}
\end{align}
 Recalling the definition of $<\chi,\varphi>_{^{*}\!H^1_\sigma,H_\sigma^1}$, $\|\psi\|_1\lesssim_0 \|\varphi\|_1\lesssim_0 \|\psi\|_1$ for any $t\in I_T$ and   $H_\sigma^1$ is a reflexive Banach space,  we immediately see that (referring to \cite[Lemma 1.66]{NASII04})
\begin{align}\label{202231836}(\rho v_t)_t=\chi \in L^2_T{^{*}\!H_\sigma^1}.\end{align}

Exploiting the regularity $\rho$, $u_t$ and \eqref{202231836},  by means of a classical regularization method,  we have
\begin{align}
\frac{\mm{d}}{\mm{d}t}\int \rho|v_t|^2
\mm{d}x=2<(\rho v_t)_t,v_t>_{^{*}\!H^1_\sigma,H_\sigma^1}-\int  \rho_t|v_t|^2\mm{d} {x}\mbox{ for a.e. in }I_T.
\nonumber
\end{align}
Consequently, making use of \eqref{20222020881255}, \eqref{20131} and the second identity in \eqref{appwasfdaeakssfaolux}, we get \eqref{eq0510} from the above identity. This completes the proof of Proposition \ref{qwepro:0sadfa401nxdxx}. \hfill $\Box$
\end{pf}

Now we are in a position to show Proposition \ref{202102182115}. To start with, let $(\eta^0,u^0)$ satisfy all the assumptions
in Proposition \ref{202102182115} and $\|\nabla\eta^0\|_{2}^2\leqslant\delta\leqslant\delta^{\mm{L}}$.
We should remark here that the smallness of $\delta$ (independent of $\lambda$ and  $m$) will be frequently used
in the calculations that follow.

Let $f=\partial_1^2\eta+ {G}_{\eta}\mathbf{e}_2$ with $\eta_t=w$ in Proposition \ref{qwepro:0sadfa401nxdxx}. Then, by \eqref{201912061028}--\eqref{2022202121142}, we have
\begin{align}
 \|f\|_{C^0(\overline{I_T},L^2)}^2+ \| f\|_{L^2_TH^1}^2 +\|f_t\|_{L^2_{T}{^*\!H^{1}}}^2
+ \|f\|_{L^2(\Omega_T)}^2\lesssim 1, \nonumber
\end{align}
which implies that
\begin{align}
&  \sqrt{ {\mathfrak{B}}(u^0,\partial_1^2\eta+ {G}_{\eta}\mathbf{e}_2)} \lesssim_0 1+ \| u^0\|_2^2 +\sqrt{\|\nabla w\|_{C^0(\overline{I_T},H^1)}}(1+ \| u^0\|_2^{3/2}). \label{202010192124}
\end{align}Thus, from \eqref{202011102145} and \eqref{202010192124} we get
\begin{align}
& \|u\|_{\mathcal{U}_{T}}+\|q\|_{C^0(\overline{I_T},H^1)}+\|q\|_{L^2_TH^2} \leqslant c^{\mm{L}} (1+ \| u^0\|_2^3) + \|\nabla w\|_{L^\infty H^1}/2.  \label{20201101111005}  \end{align}

Denote
\begin{align}
 {B}:= 2c^{\mm{L}}( 1+\| u^0\|_2^3 ),
\label{202011111007}
\end{align}
 where  the constant $c^{\mm{L}}$ is the same as in  \eqref{20201101111005}.
  By \eqref{20201101111005}  and Proposition \ref{qwepro:0sadfa401nxdxx} with $ {B}$ defined by \eqref{202011111007}, we can
 construct a function sequence $\{u^k,{q}^{k}\}_{k=1}^\infty$ defined on $\Omega_T$ with $T$ satisfying \eqref{201912061028}.
 Moreover,
\begin{itemize}
  \item for $k\geqslant 1$, $({u}^{k+1},q^{k+1})\in  \mathcal{U}_{T}\times (C^0(\overline{I_T},\underline{H^1})\cap L^2_T{H}^2) $, $(\bar{\rho}u^{k+1}_1)_\Omega= (\bar{\rho}u^{k+1}_1)_\Omega  |_{t=0} $ and
 \begin{equation}\label{iteratiequat}\begin{cases}
\eta^k=\int_0^tu^k\mm{d}\tau+\eta^0,\\
 \bar{\rho }u_t^{k+1}+\nabla_{\mathcal{A}^k} {q}^{k+1}-\mu  \Delta_{\mathcal{A}^k} {u}^{k+1}=\lambda m^2\partial_1^2\eta^k+gG_{\eta^{k}}\mathbf{e}_2 ,\\
\mathrm{div}_{\mathcal{A}^k} {u}^{k+1}=0,\\
   {u}^{k+1}  |_{t=0}=u^0 ,\\
( {\eta}^{k}_2,\partial_2 {\eta}^{k}_1, {u}^{k+1}_2, \partial_2 {u}^{k+1}_1 )|_{\partial\Omega}=0 \end{cases}  \end{equation}
with initial condition ${u}^{k+1} |_{t=0}=u^0$,
where $\mathcal{A}^k:=(\nabla \eta^k+I)^{-\mm{T}}$;
  \item $({u}^1,q^1)$ is constructed by Proposition \ref{qwepro:0sadfa401nxdxx} with $w=0$ and  with $(\partial_1^2\eta^0+gG_{\eta^0}\mathbf{e}_2)$ in place of $f$;
\item
the solution sequence $\{  {u}^k,q^k\}_{k=1}^\infty$ satisfies the following uniform estimates: for all $k\geqslant 1$,
\begin{align}
\label{2020100sfas11506}
& 1\leqslant 2\det(\nabla \eta^k+I)\leqslant 3,\  \|\nabla \eta^k\|_{2}\leqslant 2\delta \mbox{ for all }t\in\overline{I_T}, \\
 & \|u^k \|_{\mathcal{U}_{T}}+\|q^k \|_{C^0(\overline{I_T},H^1)} +\|q^k \|_{L^2_TH^2}\leqslant  {B}  .\label{n053112}\end{align}
\end{itemize}

In order to take the limit in \eqref{iteratiequat} as $k\to\infty$, we have to show that $\{{u}^k,q^k\}_{k=1}^\infty$ is a Cauchy sequence.
To this end, we define for $k\geqslant 2$,
$$( \bar{\eta}^{k}, \bar{ {u}}^{k+1}, \bar{\mathcal{A}}^k, \bar{q}^{k+1}):= (\eta^{k}- {\eta}^{k-1},{u}^{k+1}- {u}^k,\tilde{\mathcal{A}}^k-\tilde{\mathcal{A}}^{k-1},{q}^{k+1} -q^{k}),$$
which satisfies $(\bar{\rho}\eta^{k}_1)_\Omega= (\bar{\rho}\bar{u}^{k}_1)_\Omega  =0 $ and
 \begin{equation}\label{difeequion} \begin{cases}
\bar{\eta}^k=\int_0^t\bar{u}^k\mm{d}\tau,\\
  \bar{\rho }\bar{ { u}}_t^{k+1}+\nabla \bar{q}^{k+1}-\mu  \Delta  \bar{ {u}}^{k+1}- \lambda m^2 \partial_1^2 \bar{\eta}^k=\mathcal{N}^k,\\[1mm]
\mathrm{div} \bar{u}^{k+1}= -(\mathrm{div}_{\bar{\mathcal{A}}^k } {u}^{k+1}+\mathrm{div}_{\tilde{\mathcal{A}}^{k-1}} \bar{u}^{k+1}), \\
\bar{ {u}}^{k+1} |_{t=0}=0,\\
(\bar{\eta}^{k}_2,\partial_2\bar{\eta}^{k}_1,\bar{u}^{k+1}_2, \partial_2\bar{u}^{k+1}_1 )|_{\partial\Omega}=0 \end{cases}  \end{equation}
Here and in what follows, $\tilde{\mathcal{A}}^k= {\mathcal{A}}^k-I$ and
$$
\begin{aligned}
\mathcal{N}^k:=& \mu  ( \mm{div}_{\bar{\mathcal{A}}^k}\nabla_{\ml{A}^k}u^{k+1}
+ \mm{div}_{ \tilde{\ml{A}}^{k-1}}\nabla_{\bar{\mathcal{A}}^k}u^{k+1}+\mm{div}_{ \tilde{\ml{A}}^{k-1}}\nabla_{\ml{A}^{k-1}}\bar{u}^{k+1}\\
&+ \mm{div}\nabla_{\bar{\mathcal{A}}^k}u^{k+1}+ \mm{div} \nabla_{\tilde{\ml{A}}^{k-1}}\bar{u}^{k+1})
-\nabla_{\bar{\mathcal{A}}^k}q^{k+1}-\nabla_{\tilde{\ml{A}}^{k-1}}\bar{q}^{k+1}\\
& +g(\bar{\rho}(\eta_2^k(y,t)+y_2)-\bar{\rho}(\eta_2^{k-1}(y,t)+y_2) )\mathbf{e}_2.
\end{aligned}$$

Thanks to \eqref{201912061028}, \eqref{2020100sfas11506} and \eqref{n053112}, it is easy to check that
\begin{align}
&\|{\mathcal{A}}^{k-1} \|_{2}\lesssim_01,\ \|\tilde{\mathcal{A}}^{k-1} \|_{2}\lesssim_0\delta, \ \|{\mathcal{A}}^{k-1}_t \|_{1}\lesssim_0 \|\nabla u^{k-1}\|_1, \nonumber \\
&\|\bar{\mathcal{A}}^k \|_{1}\lesssim_0  \|\nabla \bar{\eta}^k\|_1\lesssim_0 T^{1/2} \| \nabla \bar{u}^k\|_{L^2_tH^1},\ \|\bar{\mathcal{A}}^k _t\|_0\lesssim_0  \|\nabla \bar{u}^k\|_0+B\|\nabla \bar{\eta}^k\|_1\label{20222011261534} \\
&\|\mathcal{N}^k\|_{L^2(\Omega_T)}
\leqslant c {T}^{1/4} \| \nabla \bar{u}^{k}\|_{L^2_TH^1} +c_0\delta (\mu\|\nabla \bar{u}^{k+1}\|_{L^2_TH^1} +
\|\bar{q}^{k+1}\|_{L^2_TH^1}). \label{20222534}
\end{align}
By Lemmas \ref{lem:0102} and \ref{20222011161567}, we can derive from \eqref{difeequion}$_2$, \eqref{difeequion}$_3$ and \eqref{difeequion}$_5$ that
\begin{align}
\mu\| \bar{ {u}}^{k+1} \|_{2}+\|\bar{q}^{k+1}\|_{1}  \lesssim_0  \|( \bar{\rho}\bar{ { u}}_t^{k+1}, \lambda m^2\partial_1^2 \bar{\eta}^k, \mathcal{N}^k, \mathrm{div}_{\bar{\mathcal{A}}^k } {u}^{k+1},\mathrm{div}_{\tilde{\mathcal{A}}^{k-1}} \bar{u}^{k+1})\|_{0} . \label{202202121416}
\end{align}

Multiplying \eqref{difeequion}$_2$ by $\bar{u}^{k+1}$, resp. $\bar{u}^{k+1}_t$ in $L^2$, we get
\begin{align}
\frac{1}{2}\frac{\mm{d}}{\mm{d}t}    \|\sqrt{\bar{\rho}}\bar{u}^{k+1}\|_0^2 + \mu \|\nabla \bar{ {u}}^{k+1}\|_0^2 =
\int (\lambda m^2  \partial_1^2 \bar{\eta}^k+\mathcal{N}^k -\nabla \bar{q}^{k+1} )\bar{u}^{k+1} \mm{d}y
\label{2022202121415} \end{align}
and
\begin{align}
\frac{\mu }{2} \frac{\mm{d}}{\mm{d}t} \|\nabla \bar{u}^{k+1}\|_0^2 + \|\sqrt{\bar{\rho }}\bar{ { u}}_t^{k+1}\|_0^2=
\int ( \lambda m^2  \partial_1^2 \bar{\eta}^k+\mathcal{N}^k-\nabla \bar{q}^{k+1})\bar{u}^{k+1}_t\mm{d}y.
\label{2022202sfda121415} \end{align}

Noting that
\begin{align*}
 \int  \nabla \bar{q}^{k+1} \cdot \bar{u}^{k+1} \mm{d}y= & \int (\mathrm{div}_{\bar{\mathcal{A}}^k } {u}^{k+1}+\mathrm{div}_{\tilde{\mathcal{A}}^{k-1}} \bar{u}^{k+1}) \bar{q}^{k+1}  \mm{d}y\\
  \lesssim  &( B\|\bar{\mathcal{A}}^k \|_{1}  + \delta\| \bar{u}^{k+1} \|_2)\| \bar{q}^{k+1} \|_1
\end{align*}
and
\begin{align*}
& \int  \nabla \bar{q}^{k+1} \cdot \bar{u}^{k+1} _t\mm{d}y =  \int  \partial_t(\mathrm{div}_{\bar{\mathcal{A}}^k } {u}^{k+1}+\mathrm{div}_{\tilde{\mathcal{A}}^{k-1}} \bar{u}^{k+1}) \bar{q}^{k+1} \mm{d}y\\
& = -\int  \partial_t(  (\bar{\mathcal{A}}^k) ^{\mm{T}} {u}^{k+1}+  (\tilde{\mathcal{A}}^{k-1} ) ^{\mm{T}} \bar{u}^{k+1})   \cdot \nabla \bar{q}^{k+1} \mm{d}y\\
& \lesssim    ( \|\bar{\mathcal{A}}^k\|_{1} \|{u}^{k+1}_t\|_1+B(\|\bar{\mathcal{A}}^k_t\|_{0} +    \| \bar{u}^{k+1} \|_1)+  \delta  \| \bar{u}^{k+1}_t \|_0)\| \bar{q}^{k+1} \|_1,
\end{align*}
thus putting \eqref{2022202121415} and \eqref{2022202sfda121415} together, and then using \eqref{202202121416}, the above two estimates and Young's inequality, we get, for sufficiently small $\delta$,
\begin{align}
&\frac{1}{2}\frac{\mm{d}}{\mm{d}t}   \|( \sqrt{\bar{\rho}}\bar{u}^{k+1}, \sqrt{\mu}\nabla \bar{u}^{k+1})\|_0^2  + c(  \| \bar{ {u}}^{k+1}\|_2^2+\|\bar{ { u}}_t^{k+1}\|_0^2+\|\bar{q}^{k+1}\|_{1}^2 )\nonumber \\
&\lesssim_{\mm{L}}  \| ( \partial_1^2 \bar{\eta}^k, \mathcal{N}^k )\|_{0}^2  + \|\bar{\mathcal{A}}^k\|_{1} ^2 \|{u}^{k+1}_t\|_1^2+B^2( \|\bar{\mathcal{A}}^k_t\|_{0}^2 +    \|(\bar{\mathcal{A}}^k, \bar{u}^{k+1}) \|_1^2) +T \| \nabla \bar{u}^k\|_{L^2_TH^1}  .
 \label{202202121421}
 \end{align}

Integrating the above inequality over $(0,t)$, and then using \eqref{201912061028},  \eqref{difeequion}$_1$, \eqref{20222011261534} and \eqref{20222534}, we get
\begin{align}
 &\|  \bar{\eta}^k\|_{L^\infty_TH^2}+  \|  \bar{u}^{k+1}\|_{L^\infty_TH^1} +\| \bar{ {u}}^{k+1} \|_{L^2_TH^2} +    \| \bar{ { u}}_t^{k+1}\|_{L^2(\Omega_T)} \nonumber \\
 &+\|\bar{q}^{k+1}\|_{L^2_TH^1}\lesssim_{\mm{L}}   T^{1/4}( \|  \bar{u}^{k+1}\|_{L^\infty_TH^1}+ \|  \nabla \bar{u}^{k}\|_{L^\infty_TL^2}+ \| \nabla \bar{u}^k\|_{L^2_TH^1}) . \nonumber
\end{align}
Hence, for sufficiently small $T$ (depending possibly on $ {B}$,  $\mu $, $\lambda$, $m$, $\bar{\rho}$ and $\Omega$),
\begin{align}
&\{  \eta^k,   {u}^k,  {u}^k_t,q^k \}_{k=1}^\infty\mbox{ is a Cauchy sequence in }\nonumber \\
&L^\infty_T {H}^2  \times (L^\infty_TH^1\cap L^2_T {H}^2 )\times L^2_TH^1_{\mm{s}} \times L^2_TH^1 .  \label{2022201261641}
\end{align}

By \eqref{n053112}, up to the extraction of a subsequence, we have, as $k\to\infty$, that
\begin{align}
&(\eta^{n_k}, u^{n_k}_t,q^{n_k} )\rightharpoonup(\eta,  u_t,q)  \mbox{ weakly-$*$ in }L^\infty_T \mathcal{H}^3_{\mm{s}}\times L^\infty_T L^2\times L^\infty_T\underline{H^1} ,
\label{202201261659} \\
&( u^{n_k},u^{n_k}_t,q^{n_k})\rightharpoonup (u,u_t,q) \mbox{ weakly  in }L^2_T {H} ^3\times L^2_TH_{\mm{s}}^1\times L^2_TH^2,\\
& u^{n_k}\rightarrow u \mbox{ strongly in }C^0(\overline{I_T},\mathcal{H}^2_{\mm{s}}) \mbox{ with  } u|_{t=0}=u_0,
\end{align}where
\begin{align}
\eta := \eta^0+\int_0^tu\mm{d}\tau.
\label{20201111413111}
\end{align}
In addition, by \eqref{2020100sfas11506} and \eqref{2022201261641}, we further obtain
\begin{align}
&(  \eta^k, 1/J^k ,  {u}^k,  {u}^k_t,q^k)\rightarrow (  \eta,  {u}, J^{-1}, {u}_t,q)\mbox{ strongly in }\nonumber \\
& L^\infty_T {H}^2 \times L^\infty_T {H}^1 \times  ( L^\infty_TH^1 \cap L^2_T {H}^2 )\times L^2_TH^1_{\mm{s}} \times L^2_TH^1,
\label{strongconvegneuN}
\end{align}
where $J=\det(\nabla \eta+I)$.

Remembering that \eqref{20201111413111} implies $\eta_t=u$, we infer
from \eqref{iteratiequat} and \eqref{202201261659}--\eqref{strongconvegneuN} that the limit $(\eta,{u},q)$ is a solution to the initial value problem \eqref{01dsaf16asdfasf} and \eqref{20safd45}; moreover, the solution $(\eta,u,q)$ belongs to $C^0(\overline{I_T},\mathcal{H}_{\mm{s}}^3)\times {\mathcal{U}}_{T}\times (C^0(\overline{I_T} ,\underline{H}^1)\cap L^2_T {H}^2)$ by following the argument of the regularity of $(u,q)$ in the proof of Proposition \ref{qwepro:0sadfa401nxdxx}.  The uniqueness of solutions to \eqref{01dsaf16asdfasf} and \eqref{20safd45} is easily verified by using Gronwall's lemma and a similar energy method to derive \eqref{202202121421}, and its proof will be omitted here. We complete the proof of Proposition \ref{202102182115}.

\appendix
\section{Analysis tools}\label{sec:09}
\renewcommand\thesection{A}
This appendix is devoted to providing some mathematical results, which have been used in previous sections.
W should point out that $\Omega$ resp. the simplified notations appearing in what follows are defined by \eqref{0101a} resp.
the same as in Section \ref{subsec:04}. In addition  $ H_0^1:=\{\upsilon\in
H^1~|~\upsilon|_{\partial\Omega}=0\}$ and $a\lesssim b$ still denotes $a\leqslant cb$ where
the positive constant $c$ depends on the parameters and the domain in lemmas in which $c$ appears.
\begin{lem}\label{201806171834}
\begin{enumerate}[(1)]
 \item Embedding inequality (see \cite[4.12 Theorem]{ARAJJFF}): Let $D\subset \mathbb{R}^2$ be a domain satisfying the cone condition, then
\begin{align}
&\label{esmmdforinfty}\|f\|_{C^0(\overline{D})}= \|f\|_{L^\infty(D)}\lesssim  \| f\|_{H^2(D)}.
\end{align}
 \item Interpolation inequality in $H^j$
 (see \cite[5.2 Theorem]{ARAJJFF}): Let $D$ be a domain in $\mathbb{R}^2$ satisfying the cone condition, then for any given $0\leqslant j<i$,
\begin{equation}\label{201807291850}
\|f\|_{H^j(D)}\lesssim\|f\|_{L^2(D)}^{(i-j)/i}\|f\|_{H^i(D)}^{j/i} \lesssim\varepsilon^{-j/(i-j)}\|f\|_{L^2(D)} +\varepsilon \|f\|_{H^i(D)},
\quad\forall\;\varepsilon >0,
\end{equation}
where the two estimate constants in \eqref{201807291850} are independent of $\varepsilon$.
\item
Product estimates (see Section 4.1 in \cite{JFJSNS}): Let $D\in \mathbb{R}^2$ be a domain satisfying the cone condition, and the functions $\varphi$, $\psi$ defined in $D$. Then
\begin{align}
\label{fgestims}
&
 \|\varphi\psi\|_{H^i(D)}\lesssim    \begin{cases}
                      \|\varphi\|_{H^1(D)}\|\psi\|_{H^1(D)} & \hbox{ for }i=0;  \\
  \|\varphi\|_{H^i(D)}\|\psi\|_{H^2(D)} & \hbox{ for }0\leqslant i\leqslant 2.\\
                    \end{cases}
\end{align}
\item  Anisotropic product estimates
(please refer to \cite[Lemma 3.1]{jiang2021asymptotic}): Let  the functions $\varphi$ and $\psi$ be defined in $\Omega$. Then
\begin{align}
\label{fgessfdims}
  \|\varphi\psi\|_0\lesssim  \begin{cases}\sqrt{\|\varphi \|_{0} \|\varphi \|_{\underline{1},0}} \|\psi \|_{1},\\
 \|\varphi \|_{0} \sqrt{\|\varphi \|_{\underline{1},0} \|\varphi \|_{\underline{1},1}} . \end{cases}
\end{align}                \end{enumerate}
\end{lem}

\begin{lem}\label{10220830p}
Friedrich's inequality (see \cite[Lemma 1.42]{NASII04}): Let $1\leqslant p<\infty$, $n\geqslant 2$ and
$D\subset \mathbb{R}^n$ be a bounded Lipchitz domain.
Let a set $\Gamma\subset\partial D$ be measurable with respect to the $(n-1)$-dimensional measure $\tilde{\mu}:=\mathrm{meas}_{n-1}$
defined on $\partial D$ and let $\mathrm{meas}_{n-1}(\Gamma)>0$. Then
\begin{equation}
\label{poincare1}
\|w\|_{W^{1,p}(D)}\lesssim \|\nabla w\|_{L^p(D)}
\end{equation}
  for any  $w\in W^{1,p}(D)$ with $u\big|_{\Gamma}=0$ in the sense of trace.
\end{lem}
\begin{rem}\label{10220sasfasaf830p}
By Lemma \ref{10220830p}, we easily see that
\begin{equation}
\nonumber
\|w\|_{W^{1,p}(D)}\lesssim \| w'\|_{L^p(D)}
\end{equation}
 for any $w\in W^{1,p}(D)$ with
$w(0)=0$ or $w(T)=0$, where $D:=(0,T)$. Hence, we further obtain
\begin{equation}
\label{poinsafdcasfaressadf1}
\|\varpi\|_{0}\lesssim \| \partial_2 \varpi\|_{0}\mbox{ for any }\varpi\in H_0^1 .
\end{equation}
\end{rem}
\begin{lem}\label{10220830}
Poincar\'e's inequality (see \cite[Lemma 1.43]{NASII04}): Let $1\leqslant p<\infty$, and $D$ be a bounded Lipchitz domain in $\mathbb{R}^n$ for $n\geqslant 2$ or a finite interval in $\mathbb{R}$. Then for any $w\in W^{1,p}(D)$,
\begin{equation}
\label{poincare}
\|w\|_{L^p(D)}\lesssim \|\nabla w\|_{L^p(D)}^p+\left|\int_{D}w\mathrm{d}y\right|^p.
\end{equation}
\end{lem}
\begin{rem}\label{10220saf830p}
By Poincar\'e's inequality, we have that for any given $i\geqslant  0$,
\begin{align}
&\label{202012241002}
\| w\|_{1,i}\lesssim \|w\|_{2,i}\; \mbox{ for any } w \mbox{ satisfying }\partial_1w,\ \partial_1^2w\in H^i.
\end{align}
 \end{rem}
\begin{lem}\label{pro4a}
Hodge-type elliptic estimates (see \cite[Lemma A.4]{ZHAOYUI}):
If $w\in H^i_{\mm{s}}$ with $i\geqslant1$,
then
\begin{align}
&\label{202005021302}
\|\nabla w\|_{i-1}
\lesssim\|(\mm{curl}w,\mm{div}w)\|_{i-1}.
\end{align}
\end{lem}

\begin{lem}\label{pro:12x21}  Bogovskii's operator in the standing-wave form (see (2.52) in \cite{JFJSJMFMOSERT}):
There exists Bogovskii's operator $\mathcal{B} :f\in \underline{L}^2 \to H^1_0$. Moreover, $\mathcal{B}(f)$ satisfies
\begin{align*}
 &\mathrm{div}\mathcal{B}(f)=f  ,  \\
&\|{\mathcal{B}}(f)\|_1\lesssim \|f\|_0\mbox{ and }
(\mathcal{B}(f))(2\pi n L   )=0
\end{align*}
 for any integer  $n$.
 \end{lem}
\begin{lem}\label{pro:1221}Diffeomorphism mapping theorem (see \cite[Lemma A.8]{ZHAOYUI}):
There exists a sufficiently small constant $\gamma \in(0,1)$, depending on $\Omega$, such that for any ${\varsigma}\in H_{\mm{s}}^3$
satisfying $\|\nabla {\varsigma}\|_2\leqslant \gamma $,
$\psi:=\varsigma+y$ (after possibly being redefined on a set of measure zero with respect to variable $y$) satisfies
the same diffeomorphism  properties as $\zeta$ in \eqref{20210301715x} and \eqref{20210301715}, and
$\inf_{y\in \Omega}\det(\nabla \varsigma  +I)\geqslant 1/4$.
\end{lem}
\begin{lem}\label{psdf221} Integration by parts for the functions with values in Banach spaces (see \cite[Theorem 1.67]{NASII04}):
Let $H$ be a Hilbert space and $V\hookrightarrow H$ be dense in $H$. If $u$, $v\in L^p_TV$ with $T\in \mathbb{R}^+$,  $1<p<\infty$ and $u_t$, $v_t\in L^q_T{^*V}$, $p^{-1}+q^{-1}=1$, then $u$, $v\in C(\overline{I_T}, H)$ and
\begin{align}
(u(t),v(t))-(u(s),v(s))=\int_s^t(<u_t(\tau),v(\tau)> +<v_t(\tau),u(\tau)>)\mm{d}\tau,
\label{20222012312450}
\end{align}
where $s$, $t\in \overline{I_T}$ and $<\cdot,\cdot>$ is the duality between $V$ and $^* V$.
\end{lem}
\begin{rem}\label{202202111758}
The integral at the right hand of the identity \eqref{20222012312450} presents that
\begin{align}
\label{2022202051810}
 <u_t(\tau),v(\tau)> +<v_t(\tau),u(\tau)> \in L^1(I_T).
 \end{align}
Thus, by the property of absolutely continuous function \cite[Lemma 1.7]{NASII04} and \eqref{20222012312450}, we easily see that
\begin{align}
\frac{\mm{d}}{\mm{d}t}(u(t),v(t))=<u_t,v> +<v_t,u>\mbox{ a.e. in } I_T.
\label{2012312450}
\end{align}
In particular, we easily derive from \eqref{2012312450} that
\begin{itemize}
  \item for $H=L^2$,
\begin{align}
\frac{\mm{d}}{\mm{d}t}\int u(t)v(t)\mm{d}y=<u_t,v> +<v_t,u> \mbox{ a.e. in } I_T.
\label{20122450}
\end{align}
  \item for $H=V=L^2$,
  \begin{align}
\partial_t( u(t)v(t)) = u_t v  + v_tu .
\label{2012fsa2450}
\end{align}
\end{itemize}
 \end{rem}

\begin{lem}\label{20021032019018} Properties of composite functions with values in Banach spaces:
  Let $T>0$,  integers $i \geqslant 0$ be given and  $1\leqslant p\leqslant \infty$. Let $\varphi =\varsigma(y,t)+y$ and
  $$\varsigma\in    \{\psi\in  {C}^0(\overline{I_T},{H}^3_{\mm{s}})~|~\psi\in H^3_\gamma ,\ \inf\nolimits_{(y,t)\in {\Omega_T}}\det(\nabla \psi +I)\geqslant 1/4\}. $$
\begin{enumerate}
\item[(1)] If $f\in C^0(\overline{I_T},H^i)$ or $f\in L^p_TH^i$ with $0\leqslant i\leqslant 3$,  then
 \begin{align}
\label{2020103250855}
&F:=f(\varphi ,t)\in C^0(\overline{I_T},H^i)\mbox{ or }L^p_TH^i
\end{align}
and
\begin{align} \mathcal{F}:=f(\varphi^{-1},t)\in C^0(\overline{I_T},H^i)\mbox{ or } L^p_TH^i.
\label{202104031901} \end{align}
Moreover,
\begin{align}
 & \|F\|_{L^p_TH^i}\lesssim
P(\|\varsigma\|_{L^\infty_TH^3}) \|f \|_{L^p_TH^i},
\label{2021sfa04031901} \\
&  \|\mathcal{F}\|_{L^p_TH^i}\lesssim
P(\|\varsigma\|_{L^\infty_TH^3})\|f \|_{L^p_TH^i}. \label{2022104101908} \end{align}
 \item[(2)]  If $\varsigma$ additionally satisfies $ \varsigma_t \in L^\infty_TH^2$,
 then for any $f\in L^p_TH^i$ satisfying $f_t\in L^p_TH^{i-1}$ with $1\leqslant i\leqslant 3$,
 \begin{align}
 F_t=(f_t(x,t) +\varsigma_t\cdot \nabla f(y,t))|_{x=\varphi}\in  L^p_TH^{i-1} \label{202104032132}
\end{align}
and
\begin{align}
\mathcal{F}_t=(f_t(y,t) -(\nabla \varphi)^{-1}\varsigma_t \cdot \nabla f(y,t))|_{y=\varphi^{-1}}\in  L^p_TH^{i-1}.
\label{2021040sdaf32132}
\end{align}
\end{enumerate}
\end{lem}
\begin{proof}
Please refer to \cite[Lemma A.10]{ZHAOYUI}.
\end{proof}
\begin{lem}\label{lem:0102} Generalized Korn--Poincar\'e inequality:
Let $D\subset\mathbb{R}^2$ be a bounded domain satisfying the cone condition. Assume
that $p> 1$, $u\in H^1(D)$,
\begin{equation}\label{0216}
\chi\geqslant 0,\ 0<a\leqslant { \|\chi\|_{L^1(D)}},\ \|\chi\|_{L^p(D)}\leqslant b.
\end{equation}
Then
\begin{equation*}\label{0217}
\left\|u\right\|_{L^2(D)}\lesssim \left\|\nabla
u\right\|_{L^2(D)}+\left|\int_D \chi u \mathrm{d}y
\right|.
\end{equation*}
\end{lem}
\begin{proof}
We prove the lemma by contradiction. Suppose that
the conclusion of Lemma \ref{lem:0102} fails, then there would be a sequence
$\{\chi _n\}_{n=1}^\infty$ of non-negative functions satisfying
\eqref{0216} with $\chi_n$ in place of $\chi$ for any $n\geqslant 1$ and a sequence $\{u_n\}_{n=1}^\infty \subset H^1(D)$, such that
\begin{equation}\label{0218}
\|u_n\|_{L^2(D)}\geqslant a_n\left( \|\nabla
u_n\|_{L^2(D)}+\left|\int_D \chi _n u_n\mathrm{d}y
\right|\right)\mbox{ and }  a_n\rightarrow +\infty.
\end{equation}
Setting $w_n=u_n\|u_n\|_{L^2(D)}^{-1}$, making use of
the compactness embedding $H^1(D)\hookrightarrow\hookrightarrow L^p(D)$  and \eqref{0218}, we find that
\begin{equation}\label{0219}
w_n\rightarrow w= {{|D|^{-1/2}}} \mbox{ in }L^q(D),
\end{equation}
where $q=p/(p-1)$.

In addition, there exists
  a function $\tilde{\chi }\geqslant 0$ satisfying
\begin{equation}\label{jjy0215}
0< a\leqslant
\|{ \tilde{\chi }}\|_{L^1(\Omega)},\ \|{ \tilde{\chi }}\|_{L^p(D)}\leqslant b
\end{equation}
and
\begin{equation}\label{0220}
\int_D\chi _n \varphi\mathrm{d}y\to
 \int_D{ \tilde{\chi }}
\varphi\mathrm{d}y \mbox{ for any }\varphi\in L^q(D).
\end{equation}
Thus, by virtue of \eqref{0219} and \eqref{0220}, we have
\begin{equation}\label{0221}
\lim_{n\rightarrow  \infty}\int_D\left(\chi _n w_n-
{ \tilde{\chi }}
w\right)\mathrm{d}y=\lim_{n\rightarrow
\infty}\int_D\chi _n (w_n -w)\mathrm{d}y+\lim_{n\rightarrow
\infty}\int_D(\chi _n -{ \tilde{\chi }} )w\mathrm{d}y=0.
\end{equation}
The identity \eqref{0221}, together with \eqref{0219} and \eqref{jjy0215}, yields
\begin{equation}\label{0222}
\lim_{n\rightarrow  \infty}\int_D\chi _n w_n
\mathrm{d}y=\int_D{ \tilde{\chi }} w\mathrm{d}y>0.
\end{equation}

Finally, \eqref{0218} implies
\begin{equation}\label{0223}
\lim_{n\rightarrow \infty}\int_D\chi _n w_n\mathrm{d}y=0,
\end{equation}
which contradicts with (\ref{0222}). Therefore, the conclusion of Lemma \ref{lem:0102} remains true. \end{proof}

\begin{lem}
\label{lem:08181945}Elliptic estimates:  Let $a>0$, $\delta\in (0,1]$, $\chi\in L^\infty$, $\chi\geqslant a$,
\begin{align}
\|A-I\|_2\lesssim \delta.
\label{20222020101604}
\end{align}
If $f\in L^2$,  for sufficiently small $\delta$,
there exists a unique weak solution $p\in \underline{H}^1$ such that
\begin{align}
\label{202202101625}
\int  \chi \nabla_A   q \cdot \nabla_A \psi\mm{d}y = \int f\cdot \nabla\psi\mm{d}y\mbox{ for any } \psi\in H^1.
\end{align}
 Moreover, $q$ satisfies
\begin{align}\label{neumaasdfann1n}
\|q\|_1\lesssim  \| f \|_0.
\end{align}
\end{lem}
\begin{pf}
We define an inner-product of $\underline{H^{1} }$ by
\begin{align*}
(\varphi,\phi)_{\underline{H^{1}}}:=\int   \chi \nabla_A \varphi \cdot \nabla_A\phi\mm{d}y \mbox{ for }\varphi,\ \phi
\in \underline{H^{1}} ,
\end{align*}
and the corresponding norm by $\|\varphi\|_{\mathcal{X}}:=\sqrt{(\varphi,\varphi)_{\underline{H^{1}}}}$.
Obviously, by \eqref{esmmdforinfty}, the Poincar\'e inequality \eqref{poincare} and the  smallness condition \eqref{20222020101604}, we obtain
\begin{align*}
\|\varphi\|_{1}\lesssim\|\varphi\|_{\mathcal{X}}\lesssim  \|\varphi\|_{1}.
\end{align*}

Defining the functional
\begin{align}
\nonumber
F(\varphi):= \int   f \cdot \nabla \varphi\mm{d}y \mbox{ for }\varphi\in \underline{H^{1}},
\end{align}
we easily see that $F$ is a bounded linear functional on $\underline{H^1}$.
By virtue of the Riesz representation theorem, there is a unique $q\in\underline{H^1}$, such that
\begin{eqnarray}\label{202012191937}
&& (q,\varphi)_{\underline{H^{1}}}=F(\varphi) \mbox{ for any }\varphi\in \underline{H^{1}}; \\
&& \label{202012191940}
\|q\|_{1}\lesssim\|q\|_{\mathcal{\chi}} \lesssim \|f\|_0 . \nonumber
\end{eqnarray}

For $\psi\in {H}^{1} $, we denote $\varphi=\psi-(\psi)_{\Omega}$.
Then, $\varphi\in \underline{H^{1}}$. Putting this $\varphi$ in \eqref{202012191937}, we  get
\eqref{202202101625}. This completes the proof.\hfill $\Box$
\end{pf}
\begin{lem}\label{20222011161567}
Stokes estimates: Let $i\geqslant 0$, $(f,\varpi )\in H^i\times  H^{2+i}$  and $\varpi _2|_{{\partial\Omega}}=0$, the Stokes problem with Navier boundary condition
\begin{equation}
\label{Ston}
\begin{cases}
\nabla P-  \Delta v
= f&\mbox{in }{\Omega},\\
 \div v=\div\varpi &\mbox{in }{  \Omega},\\
  (v_2,\partial_2  v_1 )=0&\mbox{on }{\partial\Omega}
\end{cases}
\end{equation}
admits a unique solution $(v_1 ,P)\in  {\underline{\mathcal{H}^{2+i }_{\mm{s}}}}\times \underline{H}^{1+i }$, and the solution satisfies
\begin{align}
\label{202201122130}
\|v\|_{2+i }+\|P\|_{1+i }\lesssim \|(f,\varpi )\|_{i}+\|\mm{div}\varpi \|_{1+i }.
\end{align}
\end{lem}
\begin{rem}
\label{202202202212011}
Obviously, the above lemma with  ${^0\mathcal{H}^{2+i }_{\mm{s}}}$ in place of
$\underline{\mathcal{H}^{2+i }_{\mm{s}}}$  also holds.
\end{rem}
\begin{pf}
(1) We first consider the case $i=0$.
Noting that $\varpi _2|_{{\partial\Omega}}=0$, the boundary value problem
 \begin{equation}
\label{Stadfon}
\begin{cases}
  \Delta \theta
= -\mm{div}\varpi  ,\\
  \partial_2 \theta|_{\partial\Omega} =0
\end{cases}
\end{equation}
admits  a unique solution $ \theta \in \underline{H}^3$, which satisfies
\begin{align}\label{202061755}
\|\theta\|_3 \lesssim \| \varpi \|_0+\|\mm{div}\varpi \|_1,
\end{align}
please refer to \cite[Lemma 4.27]{NASII04} for the proof.

Let $\psi=f+\nabla \mm{div}\varpi  \in L^2$, then the Stokes problem with Navier boundary condition
 \begin{equation}
\label{Stsdfadfon}
\begin{cases}
\nabla P-  \Delta w
=  \psi ,\\
 \div w=0 ,\\
  (w_2, \partial_2   w_1 )|_{\partial\Omega}=0
\end{cases}
\end{equation}
also admits  a unique solution $(w,P)\in  \underline{\mathcal{H}^{2}_{\mm{s}}}\times \underline{H}^{1}$, which satisfies
\begin{align}\label{2022202061755}
\|w\|_2+\|P\|_1\lesssim \|\psi\|_0,
\end{align}
please refer to \cite[Theorem 5.10]{tapia2021stokes} for the proof. Thus let
$v=w-\nabla \theta$, we easily see that $v$ satisfies \eqref{Ston} and \eqref{202201122130} from \eqref{Stadfon}--\eqref{2022202061755}.

(2) Now we turn to the proof of the case $i>1$ by induction. We assume that the problem \eqref{Ston}  admits a unique solution $(v_1 ,P)\in  \underline{\mathcal{H}^{2+j }_{\mm{s}}}\times \underline{H}^{1+j}$, and the solution satisfies
\begin{align}
\label{2022asda2130}
\|v\|_{2+j }+\|P\|_{1+j}\lesssim \|(f,\varpi )\|_{j}+\|\mm{div}\varpi \|_{1+j},
\end{align}
where $0\leqslant j<i$. Obviously, to get the desired conclusion, next it suffices to prove that
\begin{align}
\label{2022a2130}
\|v\|_{3+j}+\|P\|_{2+j }\lesssim \|(f,\varpi )\|_{1+j}+\|\mm{div}\varpi \|_{2+j }.
\end{align}

By the standard method of difference quotient in \cite[Lemma A.7]{ZHAOYUI}, it is easy to see that
\begin{align}
\label{20222130}
\|v\|_{1,2+j }+\|P\|_{1,1+j}\lesssim \|(f,\varpi )\|_{1,j}+\|\mm{div}\varpi \|_{1,1+j}.
\end{align}

We can rewrite \eqref{Ston}$_1$ as the following boundary value problem:
$$
\begin{cases}
  \Delta v_1
=  \partial_1 P -f  ,\\
  v_1 = v_1 |_{\partial\Omega } .
\end{cases}
$$
Applying the classical regularity of elliptic equation to the above  boundary value problem, and then using the trace theorem, we further get
\begin{align}
\|v_1 \|_{3+j}\lesssim  \|f\|_{1+j}+\| P \|_{1,1+j}+ \|v_1 \|_{H^{5/2+j}}\lesssim \|f\|_{1+j}+\| P \|_{1,1+j}+ \|v_1 \|_{\underline{2},1+j}.
\end{align}

In addition, thanks to \eqref{Ston}$_2$, we further have
\begin{align}
\| \partial_2v_2\|_{2+j }\lesssim \|  (\partial_1v_1,\mm{div}\varpi)\|_{2+j }. \label{2022010161911}
\end{align}
Putting \eqref{2022asda2130} and \eqref{20222130}--\eqref{2022010161911} together yields \eqref{2022a2130}.
 This completes the proof. \hfill $\Box$\end{pf}
 \begin{lem}\label{20222011231648}
 Existence of orthogonal basis in $H^1_\sigma$:
  There exists  a countable orthogonal basis
  $\{\varphi^i\}_{i=1}^\infty\subset   {\mathcal{H}}^\infty_{\sigma } $ to $ {H}^1_{\sigma } $. Moreover $\{\varphi^i\}_{i=1}^\infty$ is an orthonormal basis to $ {L}^2_{\sigma }:=\{w\in L^2~|~ \mm{div}w=0\}$.
\end{lem}
\begin{pf}Please refer to \cite[Lemma 2.2]{clopeau1998vanishing}, \cite[Lemma 3.2]{li2019global} or \cite[Theorem 1 in Secction 6.5]{ELGP}.\hfill $\Box$
\end{pf}

\vspace{4mm}
\noindent\textbf{Acknowledgements.}
The research of Fei Jiang was supported by NSFC (Grant Nos.  12022102) and the Natural Science Foundation of Fujian Province of China (2020J02013), and the research of Song Jiang by National Key R\&D Program (2020YFA0712200), National Key Project (GJXM92579), and
NSFC (Grant No. 11631008), the Sino-German Science Center (Grant No. GZ 1465) and the ISF--NSFC joint research program (Grant No. 11761141008).

\renewcommand\refname{References}
\renewenvironment{thebibliography}[1]{%
\section*{\refname}
\list{{\arabic{enumi}}}{\def\makelabel##1{\hss{##1}}\topsep=0mm
\parsep=0mm
\partopsep=0mm\itemsep=0mm
\labelsep=1ex\itemindent=0mm
\settowidth\labelwidth{\small[#1]}
\leftmargin\labelwidth \advance\leftmargin\labelsep
\advance\leftmargin -\itemindent
\usecounter{enumi}}\small
\def\newblock{\ }
\sloppy\clubpenalty4000\widowpenalty4000
\sfcode`\.=1000\relax}{\endlist}
\bibliographystyle{model1b-num-names}

\end{document}